\documentclass[11pt,reqno]{amsart}
\usepackage[english]{babel}
\usepackage{amsmath, amssymb, amsthm, amscd, amsaddr}
\usepackage{enumerate}
\usepackage{palatino}
\usepackage{mathpazo}
\usepackage{mathrsfs}
\usepackage{paralist}
\usepackage{bm}
 \usepackage[a4paper]{geometry}
\usepackage{url}
\usepackage{hyperref}
\usepackage[all]{xy}
\usepackage{nomencl}

\newcommand{\NN}{\mathbb{N}}

\newcommand{\ZZ}{\mathbb{Z}}

\DeclareMathOperator{\im}{im}
\DeclareMathOperator{\id}{id}

\newcommand{\frs}{f_{r,s}}
\newcommand{\bfrs}{\bar{f}_{r,s}}
\newcommand{\tfrs}{\tilde{f}_{r,s}}

\newcommand{\mas}{m\otimes a_1 \otimes \dots \otimes a_n}

\newcommand{\Hom}{\mathrm{Hom}}

\newcommand{\quot}[2]{\left.\raisebox{.2em}{$#1$}\middle/\raisebox{-.2em}{$#2$}\right.}
\newcommand{\Tor}{\mathrm{Tor}}

\theoremstyle{plain}
\newtheorem{theorem}{Theorem}[section]
\newtheorem{prop}[theorem]{Proposition}
\newtheorem{lemma}[theorem]{Lemma}
\newtheorem{cor}[theorem]{Corollary} 

\theoremstyle{definition}
\newtheorem{definition}[theorem]{Definition}

\theoremstyle{remark}
\newtheorem{remark}[theorem]{Remark}
\newtheorem{example}[theorem]{Example}

\numberwithin{equation}{section}

\begin{document}
\setlength{\parindent}{0cm}
 \title{A primer on $A_\infty$-algebras and their Hochschild homology}
 \author{Stephan Mescher}
 \address{Ruhr-Universit\"at Bochum \\ Fakult\"at f\"ur Mathematik \\ Universit\"atsstra\ss{}e 150 \\ 44801 Bochum (Germany)}
 \email{Stephan.Mescher@ruhr-uni-bochum.de}
 
 \begin{abstract}
We present an elementary and self-contained construction of $A_\infty$-algebras, $A_\infty$-bimodules and their Hochschild homology and cohomology groups. In addition, we discuss the cup product in Hochschild cohomology and the spectral sequence of the length filtration of a Hochschild chain complex. 

$A_\infty$-structures arise naturally in the study of based loop spaces and the geometry of manifolds, in particular in Lagrangian Floer theory and Morse homology. In several geometric situations, Hochschild homology may be used to describe homology groups of free loop spaces, see e.g. \cite{Jones}, \cite{Goodwillie} or \cite{SeidelHochschild}.

The objective of this article is not to introduce new material, but to give a unified and coherent discussion of algebraic results from several sources. It further includes detailed proofs of all presented results. 
\end{abstract}

 \date{\today}
 \maketitle
 \tableofcontents

 \section*{Introduction}

$A_\infty$-algebras, also known as strongly homotopy associative algebras, are generalizations of associative algebras. They are obtained by formalizing and refining the notion that the product of an algebra is associative \emph{up to homotopy}. They were introduced by James Stasheff in \cite{StasheffII} for the study of based loop spaces and further used in homotopy theory by several authors. In the geometry of manifolds, $A_\infty$-algebras and their generalizations to $A_\infty$-categories were introduced into Morse homology and Floer homology in symplectic geometry by Kenji Fukaya in \cite{FukayaAinfty} (see also \cite{SeidelBook} for a detailed discussion). In his talk at the ICM 1994, Maxim Kontsevich used $A_\infty$-structures in the formulation of his homological version of the mirror symmetry conjecture, see \cite{KontsevichHMS}.

The author's aim is to give an introduction to $A_\infty$-algebras and their Hochschild (co)\-ho\-mo\-lo\-gy which is as elementary and as self-contained as possible. While $A_\infty$-algebras are introduced and well explained in several places in the literature, e.g. in \cite{KellerIntro}, \cite{GetzlerJones}, \cite{TradlerInner} or \cite{KontsevichSoibelmanNotes}, and while their Hochschild homology has also been described on several occasions, e.g. in \cite[Section 5]{SeidelSubalgebras} or \cite[Section 7]{KontsevichSoibelmanNotes}, the author was not able to obtain a source which discusses the topics contained in this article both in an elementary and in a complete way, especially including all necessary sign computations. This article is intended to provide a remedy. 

\bigskip

After presenting the definitions and some basic constructions of $A_\infty$-algebras and $A_\infty$-bimodules in Section \ref{SectionAinftyalgebras}, we define the Hochschild homology of $A_\infty$-algebras in Section \ref{SectionHochschildHomology} and show that morphisms of $A_\infty$-bimodules induce homomorphisms between the corresponding Hochschild homology groups. The largest part of Section \ref{SectionHochschildHomology} is occupied by the proof of the latter statement and by the proof of Theorem \ref{HochschildDifferential}, stating that the differential of the Hochschild chain complex is indeed a differential.

In Section \ref{SectionHochschildCohomology} we define the notion of Hochschild cohomology of $A_\infty$-algebras and provide a detailed discussion of its duality with Hochschild homology. Additionally, we explicitly describe the case that the coefficient $A_\infty$-bimodule is given by the algebra itself. Afterwards, we define the cup product structure of Hochschild cohomology in Section \ref{SectionHochschildCupProduct} and show that it is induced by a chain map defined on the Hochschild cochain complex. 

We conclude this article by discussing the spectral sequence of the length filtration of a Hochschild chain complex in detail in Section \ref{SectionLengthFiltration}. This spectral sequence has useful convergence properties. Moreover, its first page has a concise description which we will use to prove a simple criterion for certain maps to induce isomorphisms of Hochschild homology groups.

\bigskip 

\textit{Throughout this article let $R$ denote a commutative ring with unit. By an $R$-module, we always mean a bimodule over $R$. All tensor products appearing in this article refer to tensor products of $R$-modules. A graded $R$-module is always assumed to be graded over $\ZZ$.}

\section*{Acknowledgements}
This article is a byproduct of the author's Ph.D. thesis in preparation, written under the guidance of Prof. Dr. Matthias Schwarz at the University of Leipzig and partially supported by a scholarship from the Deutsche Telekom Foundation. The author thanks Matthias Schwarz and Alberto Abbondandolo for valuable discussions on $A_\infty$-algebras and their applications in the geometry of manifolds.

\section{\texorpdfstring{$A_\infty$-algebras and -bimodules}{A-infinity-algebras and -bimodules}}
\label{SectionAinftyalgebras}

We begin by giving an elementary definition of an $A_\infty$-algebra. Note that our sign conventions coincide with those in \cite{SeidelBook}, \cite{SeidelSubalgebras} and \cite{GetzlerJones}, but differ from the conventions used in \cite{TradlerInner} and \cite{TradlerBV}. 

\begin{definition} \index{$A_\infty$-algebra} \nomenclature{$\mu_n$}{$n$-th order multiplication of an $A_\infty$-algebra}
 A graded $R$-module $A = \bigoplus_{j \in \ZZ} A_j$ equipped with a family of homomorphisms of graded $R$-modules
\begin{equation*}
 \mu_n: A^{\otimes n} \to A \ \ , \quad \quad \deg \mu_n = 2-n, \quad \text{for every } n \in \NN,
\end{equation*}
will be called \emph{an $A_\infty$-algebra over $R$} if the following equation is satisfied for every $r \in \NN$ and every $a_1,\dots,a_r \in A$:
\begin{equation}
\label{Ainftyequation}
\sum_{\stackrel{n_1,n_2 \in \NN}{n_1+n_2=r+1}} \sum_{i=1}^{r+1-n_1}(-1)^{\maltese_1^{i-1}}\mu_{n_2}(a_1,\dots,a_{i-1},\mu_{n_1}(a_i,\dots,a_{i+n_1-1}),a_{i+n_1},\dots,a_r) = 0 \ ,
\end{equation} 
where \footnote{The use of the Maltese cross for the coefficients defining the signs in this article happens in accordance with the notation of the works of Mohammed Abouzaid and Paul Seidel, e.g. \cite{SeidelBook} or \cite{AbouzaidCriterion}. In particular, the author distances himself from any political meaning or implication of this symbol.} \nomenclature{$\maltese_i^j$}{sum of the reduced indices from the $(i+1)$-st to the $(j+1)$-st element of a tensor product $M \otimes A^{\otimes n}$}
\begin{equation} 
 \label{EqDefMalteseij}
 \maltese_i^{j} := \maltese_i^{j}(a_1,\dots,a_r)  := \sum_{q=i}^{j} \mu(x_q) -(j-i+1) 
\end{equation} 
for all $i,j \in \{1,2,\dots,r\}$ with $i \leq j$ and where for every $k \in \ZZ$ we write $\mu(a) = k$ iff $a \in A_k$. For a fixed $r \in \NN$ we refer to equation (\ref{Ainftyequation}) as \emph{the $r$-th defining equation} of the $A_\infty$-algebra $(A,(\mu_n)_{n \in \NN})$. \index{defining equations!of an $A_\infty$-algebra}
\end{definition}

\begin{example}
 \begin{enumerate}
  \item Every differential graded algebra (DGA) over $R$, i.e. an associative graded algebra over $R$ which comes with a differential satisfying the graded Leibniz rule, can be given the structure of an $A_\infty$-algebra. We simply put $\mu_1$ to be the differential, $\mu_2$ to be the algebra multiplication and $\mu_n$ to be zero for every $n \geq 3$. In this case, the second defining equation reduces to the graded Leibniz rule, the third defining equation is equivalent to the associativity of the multiplication and all defining equations of higher order simply vanish on both sides.
 \item Stasheff has shown in \cite{StasheffI} and \cite{StasheffII} that the singular chain complex of a based loop space can always be equipped with the structure of an $A_\infty$-algebra. The map $\mu_1$ is given (up to inverting the grading) by the differential of the complex. The map $\mu_2$ is given by the Pontryagin product, i.e. the map induced on singular chains by the composition of loops. While it is easy to see that the Pontryagin product is associative up to a chain homotopy, Stasheff introduced $A_\infty$-algebras to show that this statement can be strongly refined.
 \end{enumerate}
\end{example}

\begin{remark}
 It is possible to give a slightly more concise and elegant definition of an $A_\infty$-algebra and all other notions defined in this section in terms of the usual coalgebra structure on the tensor algebra of $A$. This method was introduced by Getzler and Jones in \cite{GetzlerJones} and further considered by Tradler in \cite{TradlerInner} and \cite{TradlerBV} and several other authors. 
\end{remark}

\emph{Throughout the rest of this section let $(A, (\mu_n)_{n \in \NN})$ be an $A_\infty$-algebra over $R$. }

\begin{definition} \index{reduced index}
Let $M = \bigoplus_{j \in \ZZ} M_j$ be a graded $R$-module.
\begin{enumerate}
 \item For $a \in A$, we put $\mu(a):=j$ iff $a \in A_j$ and call $\mu(a)$ the \emph{index} of $a$. We further call $\|a\| := \mu(a)-1$ the \emph{reduced index} of $a$.
 \item For $m \in M$, we put $\mu_M(m):=j$ iff $m \in M_j$ and call $\mu_M(m)$ the \emph{index} of $m$. We further call $\|m\| := \mu_M(m)-1$ the \emph{reduced index} of $m$.
\end{enumerate} 
\end{definition}

In this notation we have $\displaystyle\maltese_i^j = \sum_{q=i}^j \|a_q\|$ for all $i$ and $j$.

\begin{remark}
 The signs appearing in the defining equations of $A$ are best understood in terms of reduced indices. We consider the \emph{shifted} cochain complex $A[1]=\bigoplus_{j \in \ZZ} A[1]_j$, given by
 $$A[1]_j = A_{j+1} \qquad \forall j \in \ZZ \ , $$
 whose differential is given by $\mu_1$. In other words, the elements of $A[1]_j$ are those elements of $A$ whose \emph{reduced} index equals $j$, i.e. it holds that 
 \begin{equation}
 \label{EqIndexsA}
 \mu_{A[1]}(a) = \|a\| 
 \end{equation}
 for every $a \in A$. If one considers the operation $\mu_d$ as defined on $(A[1])^{\otimes d}$ instead of $A^{\otimes d}$, then the signs in the defining equations of the $A_\infty$-algebra $A$ will be obtained by following the usual Koszul sign convention. 
\end{remark}

The following notion was introduced by Ezra Getzler and John D.S. Jones in \cite{GetzlerJones} and by Martin Markl in \cite{Markl}.

\begin{definition}
\label{DefAinftybimodule} \index{$A_\infty$-bimodule}
 A graded $R$-module $M$ equipped with a family of maps $(\mu^M_{r,s})_{r,s \in \NN_0}$ where
\begin{equation*}
 \mu^M_{r,s}: A^{\otimes r} \otimes M \otimes A^{\otimes s} \to M
\end{equation*} \nomenclature{$\mu^M_{r,s}$}{the $(r,s)$-type multiplication of the $A_\infty$-bimodule $M$}	
is an $R$-module homomorphism with $\deg \mu^M_{r,s} = 1-r-s$ will be called \emph{an $A_\infty$-bimodule over $A$} if the following equation holds for every $r,s \in \NN_0$, $m \in M$ and $a_1,\dots,a_{r+s} \in A$:
\begin{align*}
&\sum_{\stackrel{r_1 \in \NN_0, \ r_2 \in \NN}{r_1+r_2=r+1}} \sum_{i=1}^{r_1} (-1)^{\maltese_1^{i-1}} \mu_{r_1,s}^M(a_1,\dots,a_{i-1},\mu_{r_2}(a_i,\dots,a_{i+r_2-1}),a_{i+r_2},\dots,a_r,m,a_{r+1},\dots,a_{r+s}) \\
&+ \sum_{\stackrel{r_1,r_2 \in \NN_0}{r_1+r_2=r}} \sum_{\stackrel{s_1,s_2 \in \NN_0}{s_1+s_2=s}} (-1)^{\maltese_1^{r_1}} \mu^M_{r_1,s_1}(a_1,\dots,a_{r_1},\mu^M_{r_2,s_2}(a_{r_1+1},\dots,a_r,m,a_{r+1},\dots,a_{r+s_2}),\dots,a_{r+s}) \\
&+ \sum_{\stackrel{s_1 \in \NN_0,s_2 \in \NN}{s_1+s_2=s+1}} \sum_{j=1}^{s_1} (-1)^{\maltese_1^{r-j+1}+\mu_M(m) } \\ &\qquad \mu_{r,s_1}^M (a_1,\dots,a_r,m,a_{r+1},\dots,a_{r+j-1},\mu_{s_2}(a_{r+j},\dots,a_{r+j+s_2-1}),a_{r+j+s_2},\dots,a_{r+1}) = 0 \ .
\end{align*}
We refer to this equation as the \emph{defining equation of type $(r,s)$ for the $A_\infty$-bimodule $M$}. \index{defining equations!of an $A_\infty$-bimodule}
\end{definition}

Note that the defining equation of type $(0,0)$ is equivalent to the map $$\mu^M := \mu^M_{0,0}: M \to M$$ being a differential of degree $+1$ on $M$. 

\begin{remark} \label{RemarkBimoduleonA1}  \index{$A_\infty$-bimodule!structure on $A[1]$}
It follows from \eqref{EqIndexsA} that $A[1]$ is an $A_{\infty}$-bimodule over $A$ with the operations defined by 
\begin{align*}
&\mu^{A[1]}_{r,s}: A^{\otimes r} \otimes A[1] \otimes A^{\otimes s} \to A[1] \ , \\ 
&\mu^{A[1]}_{r,s}(a_1,\dots,a_r,a_0,a_{r+1},\dots,a_{r+s}) = \mu_{r+s+1}(a_1,\dots,a_r,a_0,a_{r+1},\dots,a_{r+s}) \ ,
\end{align*}
for all $r,s \in \NN_0$. 
\end{remark}

Next we construct a slightly more sophisticated example which is taken from \cite[Section 4]{AbouzaidCriterion}. Consider the tensor product $A \otimes A$ of graded $R$-modules. In general, the tensor product of two $A_\infty$-algebras can \emph{not} be given the structure of an $A_\infty$-algebra, see \cite[Section 5.2]{KontsevichSoibelmanNotes}. Nevertheless, we show that $A \otimes A$ will always admit the structure of an $A_\infty$-bimodule over $A$ if we equip it with the grading \index{tensor product of $A_\infty$-algebras}
$$\mu_{A \otimes A}(b_1 \otimes b_2) = \|b_1\|+\|b_2\| \ , $$
i.e. if we actually consider the product grading on $A[1] \otimes A[1]$. 

Define maps $\mu^\otimes_{r,s}$ by 
\begin{align*}
 &\mu^\otimes_{r,s}:A^{\otimes r} \otimes (A\otimes A)\otimes A^{\otimes s} \to A \otimes A \ , \ \ \mu^\otimes_{r,s} = 0 \quad \text{for } \  r,s>0 \ , \\
 &\mu^\otimes_{r,0}: A^{\otimes r} \otimes (A \otimes A) \to A \otimes A \ , \\
 &a_1 \otimes \dots \otimes a_r \otimes (b_1 \otimes b_2) \mapsto \mu_{r+1}(a_1,\dots,a_r,b_1) \otimes b_2 \quad \text{for } \  r>0 \ , \\
&\mu^\otimes_{0,s}: (A \otimes A) \otimes A^{\otimes s} \to A \otimes A \ , \\
&(b_1 \otimes b_2) \otimes a_1 \otimes \dots \otimes a_s \mapsto (-1)^{\|b_1\|} b_1 \otimes \mu_{s+1}(b_2,a_1,\dots,a_s)\quad \text{for } s > 0 \ ,  \\
&\mu^\otimes_{0,0}:A \otimes A \to A \otimes A \ , \quad b_1 \otimes b_2 \mapsto \mu_1(b_1) \otimes b_2 + (-1)^{\|b_1\|} b_1 \otimes \mu_1(b_2) \ .
\end{align*}

\begin{theorem}[\cite{AbouzaidCriterion}, Proposition 4.7] \index{$A_\infty$-bimodule!structure on $A\otimes A$}
\label{AinftyTensorProduct}
 $(A\otimes A, (\mu^\otimes_{r,s})_{r,s \in \NN_0})$ is an $A_\infty$-bimodule over $A$.
\end{theorem}

\begin{proof}
By definition, $\mu^\otimes_{r,s}$ vanishes if both $r > 0$ and $s > 0$ and so does the left-hand side of the defining equation of an $A_\infty$-bimodule. If $r = 0$, $s=0$, the defining equation is equivalent to $\mu^\otimes_{0,0} \circ \mu^\otimes_{0,0} = 0$, which is clear from the definition of $\mu^\otimes_{0,0}$ as a product differential.

Assume that $r > 0$ and $s=0$. Then the third sum on the left-hand side of the defining equation of type $(r,0)$ vanishes by definition, while for every $a_1,\dots,a_r,b_1,b_2 \in A$ the first sum is given by:
\begin{align}
 &\sum_{\stackrel{r_1 \in \NN_0,r_2 \in \NN}{r_1+r_2=r+1}} \sum_{i=1}^{r_1}(-1)^{\maltese_1^{i-1}} \mu^\otimes_{r_1,0}(a_1,\dots,a_{i-1},\mu_{r_2}(a_i,\dots,a_{i+r_2-1}),a_{i+r_2},\dots,a_r,b_1 \otimes b_2) \notag \\
 &=\sum_{\stackrel{r_1,r_2 \in \NN}{r_1+r_2=r+2}} \sum_{i=1}^{r_1-1}(-1)^{\maltese_1^{i-1}} \mu_{r_1}(a_1,\dots,a_{i-1},\mu_{r_2}(a_i,\dots,a_{i+r_2-1}),\dots,a_r,b_1) \otimes b_2 \ . \label{bimodul1}
\end{align}
The reader may notice the similarity of this sum to the left-hand side of the $(r+1)$-st defining equation of the $A_\infty$-algebra $A$ for the elements $a_1,\dots,a_r,b_1 \in A$ tensorized with $b_2$, except that some parts of the defining equation are still missing. We will see that these parts are contained in the second sum of the defining equation of type $(r,s)$ which we next write down explicitly.
\begin{align}
 &\sum_{\stackrel{r_1,r_2 \in \NN_0}{r_1+r_2=r}} (-1)^{\maltese_1^{r_1}} \mu^\otimes_{r_1,0}(a_1,\dots,a_{r_1},\mu^\otimes_{r_2,0}(a_{r_1+1},\dots,a_r,b_1 \otimes b_2)) \notag \\
 &= (-1)^{\maltese_1^r} \mu^\otimes_{r,0}(a_1,\dots,a_r,\mu^\otimes_{0,0}(b_1\otimes b_2)) + \mu^\otimes_{0,0}(\mu_{r,0}(a_1,\dots,a_r,b_1 \otimes b_2)) \notag  \\
 &\quad + \sum_{\stackrel{r_1,r_2 \in \NN}{r_1+r_2=r}} (-1)^{\maltese_1^{r_1}} \mu^\otimes_{r_1,0}(a_1,\dots,a_{r_1},\mu^\otimes_{r_2,0}(a_{r_1+1},\dots,a_r,b_1 \otimes b_2)) \notag  \\
&= (-1)^{\maltese_1^r} \mu^\otimes_{r,0}(a_1,\dots,a_r,\mu_1(b_1)\otimes b_2 + (-1)^{\|b_1\|} b_1 \otimes \mu_1(b_2)) \notag \\  
&\quad + \mu^\otimes_{0,0}(\mu_{r+1}(a_1,\dots,a_r,b_1) \otimes b_2) \notag \\
 &\quad \quad + \sum_{\stackrel{r_1,r_2 \in \NN}{r_1+r_2=r}} (-1)^{\maltese_1^{r_1}} \mu_{r_1+1}(a_1,\dots,a_{r_1},\mu_{r_2+1}(a_{r_1+1},\dots,a_r,b_1)) \otimes b_2 \notag \\
&= (-1)^{\maltese_1^r} \mu_{r+1}(a_1,\dots,a_r,\mu_1(b_1))\otimes b_2 + (-1)^{\maltese_1^r +\|b_1\|} \mu_{r+1}(a_1,\dots,a_r, b_1) \otimes \mu_1(b_2)  \notag \\
&\quad + \mu_1(\mu_{r+1}(a_1,\dots,a_r,b_1)) \otimes b_2 + (-1)^{\|\mu_{r+1}(a_1,\dots,a_r,b_1)\|} \mu_{r+1}(a_1,\dots,a_r,b_1)\otimes \mu_1(b_2) \notag\\
 &\quad + \sum_{\stackrel{r_1,r_2 \geq 2}{r_1+r_2=r+2}} (-1)^{\maltese_1^{r_1-1}} \mu_{r_1}(a_1,\dots,a_{r_1-1},\mu_{r_2}(a_{r_1},\dots,a_r,b_1)) \otimes b_2 \notag \\
&= \sum_{r_1 \geq 2, r_2 \in \NN} (-1)^{\maltese_1^{r-1}}  \mu_{r_1}(a_1,\dots,a_{r_1-1},\mu_{r_2}(a_{r_1},\dots,a_r,b_1))\otimes b_2 \label{bimodul2} \\
&\qquad \qquad \qquad + \mu_1(\mu_{r+1}(a_1,\dots,a_r,b_1))\otimes b_2 \ , \notag
\end{align}
where we have used the fact that
\begin{equation*}
 \|\mu_{r+1}(a_1,\dots,a_r,b_1) \| = \sum_{q=1}^r \|a_q\| + \|b_1\|+1 = \maltese_1^r + \|b_1\|+1 \ , 
 \end{equation*}
since $\deg \mu_{r+1} = 1-r$. Thus, we have shown that for $A\otimes A$, the left-hand side of the defining equation for $A_\infty$-bimodules is given by the sum of (\ref{bimodul1}) and (\ref{bimodul2}), which amounts to:
\begin{align*}
 &\sum_{\stackrel{r_1,r_2 \in \NN}{r_1+r_2=r+2}} \sum_{i=1}^{r_1-1} (-1)^{\maltese_1^{i-1}} \mu_{r_1}(a_1,\dots,a_{i-1},\mu_{r_2}(a_i,\dots,a_{i+r_2-1}),a_{i+r_2},\dots,a_r,b_1) \otimes b_2 \\
 &\quad + \sum_{\stackrel{r_1,r_2 \in \NN}{r_1+r_2=r+2}} (-1)^{\maltese_1^{r_1-1}} \mu_{r_1}(a_1,\dots,a_{r_1-1},\mu_{r_2}(a_{r_1},\dots,a_r,b_1))\otimes b_2 \\
 &= \sum_{\stackrel{r_1,r_2 \in \NN}{r_1+r_2=r+2}} \sum_{i=1}^{r_1} (-1)^{\maltese_1^{i-1}} \mu_{r_1}(a_1,\dots,a_{i-1},\mu_{r_2}(a_i,\dots,a_{i+r_2-1}),a_{i+r_2},\dots,a_r,b_1) \otimes b_2 \ .
\end{align*}
The latter sum equals the left-hand side of the $A_\infty$-equation for $r+1$ tensorized with $b_2$. Since $A$ is an $A_\infty$-algebra, the sum vanishes.

We omit the remaining case, $r=0$ and $s>0$, since it is proven analogously.
\end{proof}

The following theorem shows that the dual $R$-module of an $A_\infty$-bimodule can always be equipped with the structure of an $A_\infty$-bimodule over the same $A_\infty$-algebra. \index{dual of an $A_\infty$-bimodule}

\begin{theorem}[\cite{TradlerInner}, Lemma 3.9] \nomenclature{$M_{-*}$, $M^{-*}$}{graded module $M$ with inverted grading}
\label{DualBimodule}
Let $\left(M, (\mu^M_{r,s})_{r,s \in \NN_0}\right)$ be an $A_\infty$-bimodule over $R$ and let $M^*= \Hom_R(M,R)$ denote its dual $R$-bimodule. Then $\left(M^{-*}, (\mu_{r,s}^*)_{r,s \in \NN_0}\right)$ is an $A_\infty$-bimodule over $A$, where $M^{-*}$ denotes $M^*$ with inverted grading, i.e. $$M^{-j}:= (M^{-*})^j := \Hom_R(M_{-j},R) \ , $$ and where \nomenclature{$\mu^*_{r,s}$}{dual $A_\infty$-bimodule operation of type $(r,s)$} $\mu^*_{r,s}: A^{\otimes r} \otimes M^{-*} \otimes A^{\otimes s} \to M^{-*}$ is for all $r,s \in \NN_0$, $a_1,\dots,a_{r+s} \in A$, $m \in M$ and $m^* \in M^{-*}$ defined by
\begin{equation*}
 \left(\mu^*_{r,s}(a_1,\dots,a_r,m^*,a_{r+1},\dots,a_{r+s}) \right)(m) = (-1)^{\ddag_{r,s}} m^* \left(\mu^M_{s,r}(a_{r+1},\dots,a_{r+s},m,a_1,\dots,a_r) \right) 
\end{equation*}
with \nomenclature{$\ddag_{r,s}$}{exponents of the signs defining $\mu^*_{r,s}$}
\begin{align*}
 \ddag_{r,s} &:= \ddag_{r,s}(a_1,\dots,a_r,m^*,a_{r+1},\dots,a_{r+s},m) \\
 &:= \maltese_1^r \cdot \left( \maltese_{r+1}^{r+s} + \mu_{M^*}(m^*) + \mu_M(m)\right)  + \mu_{M^*}(m^*)+1 \ .
\end{align*}
\end{theorem}

\begin{proof}
 (We repeat Tradler's proof and add a sign computation since we use different sign conventions.) 

 We exhibit the operations $\left(\mu^*_{r,s}\right)_{r,s \in \NN_0}$ to fulfill the $A_\infty$-bimodule equation. We need to show for all $r,s \in \NN_0$, $a_1,\dots,a_{r+s} \in A$, $m \in M$ and $m^* \in M^{-*}$ that
\begin{align*}
&\sum_{r_1+r_2=r+1} \sum_{i=1}^{r_1} (-1)^{\maltese_1^{i-1}} \mu_{r_1,s}^*(a_1,\dots,\mu_{r_2}(a_i,\dots,a_{i+r_2-1}),a_{i+r_2},\dots,a_r,m^*,a_{r+1},\dots,a_{r+s}) (m)\\
&+ \sum_{r_1+r_2=r} \sum_{s_1+s_2=s} (-1)^{\maltese_1^{r_1}} \mu^*_{r_1,s_1}(a_1,\dots,\mu^*_{r_2,s_2}(a_{r_1+1},\dots,a_r,m^*,a_{r+1},\dots,a_{r+s_2}),\dots,a_{r+s})(m) \\
&+ \sum_{s_1+s_2=s+1} \sum_{j=1}^{s_1} (-1)^{\maltese_1^{r+j-1}+\mu(m^*)} \\ &\quad \mu_{r,s_1}^* (a_1,\dots,a_r,m^*,a_{r+1},\dots,a_{r+j-1},\mu_{s_2}(a_{r+j},\dots,a_{r+j+s_2-1}),a_{r+j+s_2},\dots,a_{r+1})(m) \stackrel{!}{=} 0 \ .
\end{align*}
where $\mu(m^*) := \mu_{M^*}(m^*)$. By definition of the $\mu^*_{r,s}$, the left-hand side of this equation is given by
\begin{align*}
 &\sum_{r_1+r_2 = r+1} \sum_{i=1}^{r_1} (-1)^{\maltese_1^{i-1}+\ddag_{r_1,s}(a_1,\dots,\mu_{r_2}(a_i,\dots,a_{i+r_2-1}),\dots,a_r,m^*,a_{r+1},\dots,a_{r+s},m)} \\ 
&\qquad m^* \left(\mu^M_{s,r_1}(a_{r+1},\dots,a_{r+s},m,a_1,\dots,\mu_{r_2}(a_i,\dots,a_{i+r_2-1}),\dots,a_r) \right) \\
&+ \sum_{r_1+r_2=r} \sum_{s_1+s_2=s} (-1)^{\maltese_1^{r_1}+\ddag_{r_1,s_1}(a_1,\dots,a_{r_1},\mu^*_{r_2,s_2}(a_{r_1+1},\dots,a_r,m^*,a_{r+1},\dots,a_{r+s_2}),a_{r+s_2+1},\dots,a_{r+s},m)} \\
&\qquad \left(\mu^*_{r_2,s_2}(a_{r_1+1},\dots,a_r,m^*,a_{r+1},\dots,a_{r+s_2}) \right) \left(\mu^M_{s,r_1}(a_{r+s_2+1},\dots,a_{r+s},m,a_1,\dots,a_{r_1})\right) \\
&+\sum_{s_1+s_2=s+1} \sum_{j=1}^{s_1} (-1)^{\maltese_1^{r+j-1}+\mu(m^*)+\ddag_{r,s_1}(a_1,\dots,a_r,m^*,a_{r+1},\dots,\mu_{s_2}(a_{r+j},\dots,a_{r+j+s_2-1}),\dots,a_{r+s},m)} \\
&\qquad m^* \left( \mu^M_{s_1,r}(a_{r+1},\dots,a_{r+j-1},\mu_{s_2}(a_{r+j},\dots,a_{r+j+s_2-1}),a_{r+j+s_2},\dots, a_{r+s},m,a_1,\dots,a_r)\right) \\
&=\sum_{r_1+r_2 = r+1} \sum_{i=1}^{r_1} (-1)^{\maltese_1^{i-1}+\ddag_{r_1,s}(a_1,\dots,\mu_{r_2}(a_i,\dots,a_{i+r_2-1}),\dots,a_r,m^*,a_{r+1},\dots,a_{r+s},m)} \\
&\qquad  m^* \left(\mu^M_{s,r_1}(a_{r+1},\dots,a_{r+s},m,a_1,\dots,\mu_{r_2}(a_i,\dots,a_{i+r_2-1}),\dots,a_r) \right) \\
&+ \sum_{r_1+r_2=r} \sum_{s_1+s_2=s} \\ 
&(-1)^{\maltese_1^{r_1}+\ddag_{r_1,s_1}(a_1,\dots,\mu^*_{r_2,s_2}(a_{r_1+1},\dots,m^*,\dots,a_{r+s_2}),\dots,a_{r+s},m)+ \ddag_{r_2,s_2}(a_{r_1+1},\dots,m^*,\dots,a_{r+s_2},\mu^M_{s_1,r_1}(a_{r+s_2+1},\dots,m,\dots,a_{r_1}))}  \\
&\qquad m^*\left(\mu^M_{s_2,r_2}(a_{r+1},\dots,a_{r+s_2},\mu^M_{s_1,r_1}(a_{r+s_2+1},\dots,a_{r+s},m,a_1,\dots,a_{r_1}),a_{r_1+1},\dots,a_r) \right) \\
&+\sum_{s_1+s_2=s+1} \sum_{j=1}^{s_1} (-1)^{\maltese_1^{r+j-1}+\mu(m^*)+\ddag_{r,s_1}(a_1,\dots,a_r,m^*,a_{r+1},\dots,\mu_{s_2}(a_{r+j},\dots,a_{r+j+s_2-1}),\dots,a_{r+s},m)} \\
&\qquad m^* \left( \mu^M_{s_1,r}(a_{r+1},\dots,a_{r+j-1},\mu_{s_2}(a_{r+j},\dots,a_{r+j+s_2-1}),a_{r+j+s_2},\dots, a_{r+s},m,a_1,\dots,a_r)\right)  \ ,
\end{align*}
Clearly this sum vanishes for every $m^* \in M^{-*}$ if and only if
\begin{align*}
&\sum_{r_1+r_2 = r+1} \sum_{i=1}^{r_1} (-1)^{S_1} \mu^M_{s,r_1}(a_{r+1},\dots,a_{r+s},m,a_1,\dots,\mu_{r_2}(a_i,\dots,a_{i+r_2-1}),\dots,a_r) \\
&+ \sum_{r_1+r_2=r} \sum_{s_1+s_2=s} (-1)^{S_2} \mu^M_{s_2,r_2}(a_{r+1},\dots,\mu^M_{s_1,r_1}(a_{r+s_2+1},\dots,a_{r+s},m,a_1,\dots,a_{r_1}),\dots,a_r) \\
&+\sum_{s_1+s_2=s+1} \sum_{j=1}^{s_1} (-1)^{S_3} \mu^M_{s_1,r}(a_{r+1},\dots,\mu_{s_2}(a_{r+j},\dots,a_{r+j+s_2-1}),\dots, a_{r+s},m,a_1,\dots,a_r) = 0 \ ,
\end{align*}
where we denote the exponents of $(-1)$ in the previous computation by $S_1$, $S_2$ and $S_3$, respectively. 


Since $(M, (\mu^M_{r,s})_{r,s \in \NN_0})$ is an $A_\infty$-bimodule over $A$, the latter equality holds true if we can show that modulo 2, we have
\begin{equation}
\label{SignsDualBimod}
\begin{aligned}
 S_1 &\equiv \maltese_{r+1}^{r+s} + \maltese_1^{i-1} + \mu_M(m)+k \ ,  \\
 S_2 &\equiv \maltese_{r+1}^{r+s_2} + k \ , \qquad S_3 \equiv \maltese_{r+1}^{r+j-1} + k \ , 
\end{aligned}
\end{equation}
for some $k \in \ZZ$. In this case, multiplying the desired equation with $(-1)^k$ yields the defining equation of type $(r,s)$ for the $A_\infty$-bimodule $M$. So it only remains to check the signs. 

Concerning $S_1$, we compute that
\begin{align*}
 &\ddag_{r_1,s}(a_1,\dots,\mu_{r_1}(a_i,\dots,a_{i+r_1-1}),\dots,a_r,m^*,a_{r+1},\dots,a_{r+s},m) \\
 &=\Bigl( \sum_{q=1}^{i-1} \|a_q\| + \|\mu_{r_2}(a_i,\dots,a_{i+r_2-1})\| + \sum_{q=i+r_2}^r \|a_q\|\Bigr) \left(\maltese_{r+1}^{r+s} + \mu(m^*)+ \mu_M(m) \right) \\
 &\phantom{booooooooooooooooooooooooooooooooooooooooooooooooooooooo} + \mu(m^*)+1 \\
&\equiv \left( \maltese_1^r + 1\right) \left(\maltese_{r+1}^{r+s} + \mu(m^*) + \mu_M(m) \right) + \mu(m^*)+1\\
&\equiv \maltese_{r+1}^{r+s} + \mu_M(m)+ \mu(m^*)+ \ddag_{r,s}(a_1,\dots,a_r,m^*,a_{r+1},\dots,a_{r+s},m) 
\end{align*}
and therefore by definition of $S_1$: 
\begin{equation*}
 S_1 \equiv \maltese_{r+1}^{r+s} +\maltese_1^{i-1} + \mu_M(m)+k_0 
\end{equation*}
if we put $k_0 := \mu(m^*)+\ddag_{r,s}$. This shows the first line of (\ref{SignsDualBimod}). Considering the sign given by $S_2$, we first compute
\begin{align*}
 &\ddag_{r_1,s_1}(a_1,\dots,a_{r_1},\mu^*_{r_2,s_2}(a_{r_1+1},\dots,a_r,m^*,a_{r+1},\dots,a_{r+s_2}),a_{r+s_2+1},\dots,a_{r+s},m) \\
&\equiv \maltese_1^{r_1} \left(\maltese_{r+s_2+1}^{r+s} + \mu\left( \mu_{r_2,s_2}^* (a_{r_1+1},\dots,a_r,m^*,a_{r+1},\dots,a_{r+s_2})\right) + \mu_M(m)\right) \\
&\qquad \qquad \qquad +\mu\left( \mu_{r_2,s_2}^* (a_{r_1+1},\dots,a_r,m^*,a_{r+1},\dots,a_{r+s_2})\right) +1 \\
&\equiv \maltese_1^{r_1} \left( \maltese_{r_1+1}^{r+s} + \mu(m^*) + \mu_M(m) + 1\right) + \maltese_{r_1+1}^{r+s_2} + \mu(m^*) \\
&\equiv \maltese_1^{r+s_2} + \maltese_1^{r_1} \left( \maltese_{r_1+1}^{r+s} + \mu(m^*) + \mu_M(m) \right) + \mu(m^*) \ .
\end{align*}
Furthermore, notice that
\begin{align*}
 &\ddag_{r_2,s_2}(a_{r_1+1},\dots,a_r,m^*,a_{r+1},\dots,a_{r+s_2},\mu^M_{s_1,r_1}(a_{r+s_2+1},\dots,a_{r+s},m,a_1,\dots,a_{r_1})) \\
&\equiv \maltese_{r_1+1}^r \left( \maltese_{r+1}^{r+s_2} + \mu(m^*) + \mu_M\left( \mu^M_{s_1,r_1} (a_{r+s_2+1},\dots,a_{r+s},m,a_1,\dots,a_{r_1}) \right) \right) + \mu(m^*)+1 \\
&\equiv \maltese_{r_1+1}^r \left( \maltese_{r+1}^{r+s} + \maltese_1^{r_1} + \mu(m^*) + \mu_M(m)+1 \right) + \mu(m^*)+1 \\
&\equiv \maltese_{r_1+1}^r + \maltese_1^{r_1}\maltese_{r_1+1}^r + \maltese_{r_1+1}^r \left( \maltese_{r+1}^{r+s} + \mu(m^*) + \mu_M(m)\right) + \mu(m^*)+1 \ . 
\end{align*}
Combining these last two computations, the definition of $S_2$ yields
\begin{align*}
 S_2 \equiv &\maltese_1^{r_1} + \maltese_1^{r+s_2} + \maltese_{r_1+1}^r + \maltese_1^{r_1} \left( \maltese_{r_1+1}^{r+s} + \mu(m^*) + \mu_M(m) \right) + \maltese_1^{r_1}\maltese_{r_1+1}^r \\
    &\qquad \qquad + \maltese_{r_1+1}^r \left( \maltese_{r+1}^{r+s} + \mu(m^*) + \mu_M(m) \right) + 1 \\
\equiv &\maltese_{r+1}^{r+s_2} + \maltese_1^r \left(\maltese_{r+1}^{r+s}+ \mu(m^*) + \mu_M(m) \right)+1 \equiv \maltese_{r+1}^{r+s_2} + k_0 \ , 
\end{align*}
implying the first congruence from the second line of (\ref{SignsDualBimod}). Finally, for $S_3$ we consider:
\begin{align*}
 &\ddag_{r,s_1}(a_1,\dots,a_r,m^*,a_{r+1},\dots,\mu_{s_2}(a_{r+j},\dots,a_{r+j+s_2-1}),\dots,a_{r+s},m) \\
&\equiv \maltese_1^r \left(\maltese_{r+1}^{r+j-1} + \|\mu_{s_2}(a_{r+j},\dots,a_{r+j+s_2-1})\| + \maltese_{r+j+s_2}^{r+s} + \mu(m^*) + \mu_M(m) \right)+ \mu(m^*)+1  \\
&\equiv \maltese_1^r \left( \maltese_{r+1}^{r+s} + 1 + \mu(m^*) + \mu_M(m)\right)+ \mu(m^*)+1 
\equiv \maltese_1^r + \mu(m^*) + k_0 \ .
\end{align*}
We consequently obtain that $S_3 \equiv \maltese_{r+1}^{r+j-1} + k_0$. Thus, we have shown all three equations of (\ref{SignsDualBimod}) with $k = k_0$ and therefore completed the proof.
\end{proof}

We conclude this section by defining the notion of morphisms of $A_\infty$-bimodules. 

\begin{definition}
\label{DefMorphismAinftyBimodules} \index{morphism of $A_\infty$-bimodules}
 Let $(M, (\mu^M_{r,s})_{r,s \in \NN_0})$ and $(N, (\mu^N_{r,s})_{r,s \in \NN_0})$ be $A_\infty$-bimodules over $A$. A family of module homomorphisms $f = (f_{r,s})_{r,s \in \NN_0}$, where:
\begin{equation*}
 f_{r,s}: A^{\otimes r} \otimes M \otimes A^{\otimes s} \to N \ ,
\end{equation*}
will be called \emph{a morphism of $A_\infty$-bimodules from $M$ to $N$ of degree $d$ }, denoted by $f:M\to N$, of degree $d \in \ZZ$ if e have
\begin{equation*}
 \deg f_{r,s} = d-r-s
\end{equation*}
for all $r,s \in \NN_0$ and if
\begin{align*}
 &\sum_{\stackrel{r_1,r_2 \in \NN_0}{r_1+r_2 = r}} \sum_{\stackrel{s_1,s_2 \in \NN_0}{s_1+s_2 = s}} (-1)^{d \cdot \maltese_1^{r_1}}  \\ &\quad \mu^N_{r_1,s_1}(a_1,\dots,a_{r_1},f_{r_2,s_2}(a_{r_1+1},\dots,a_r,m,a_{r+1},\dots,a_{r+s_1}),a_{r+s_1+1},\dots,a_{r+s}) \\
&= \sum_{\stackrel{r_1,r_2 \in \NN}{r_1+r_2 = r+1}} \sum_{i=1}^{r_1} (-1)^{\maltese_1^{i-1}+d} \\ &\quad f_{r_1,s}(a_1,\dots,a_{i-1},\mu_{r_2}(a_i,\dots,a_{i+r_2-1}),a_{i+r_2},\dots,a_r,m,a_{r+1},\dots,a_{r+s}) \\
&+ \sum_{\stackrel{r_1,r_2 \in \NN_0}{r_1+r_2 = r}} \sum_{\stackrel{s_1,s_2 \in \NN_0}{s_1+s_2 = s}} (-1)^{\maltese_1^{r_1}+d}  \\ &\quad  f_{r_1,s_1}(a_1,\dots,a_{r_1},\mu^M_{r_2,s_2}(a_{r_1+1},\dots,a_r,m,a_{r+1},\dots,a_{r+s_2}),a_{r+s_2+1},\dots,a_{r+s}) \\
&+ \sum_{\stackrel{s_1,s_2 \in \NN}{s_1+s_2=s+1}} \sum_{i=1}^{s_1} (-1)^{\maltese_1^{r+i-1}+\mu(m) + d}  \\ &\quad  f_{r,s_1}(a_1,\dots,a_r,m,a_{r+1},\dots,a_{r+i-1},\mu_{s_2}(a_{r+i},\dots,a_{r+i+s_2-1}),a_{r+i+s_2},\dots,a_{r+s}) \ 
\end{align*}
for all $a_1,\dots,a_{r+s} \in A$ and $m \in M$. We denote $f$ by $f: M \to N$. 
\end{definition}

\begin{remark}
The sign coefficients appearing in the defining equations of a morphism of $A_\infty$-bimodules might differ from those used other definitions in the literature, e.g. in \cite[Section 4]{TradlerInner}. However, our choice of signs implies that a morphism of $A_\infty$-bimodules induces a map between the corresponding Hochschild chain complexes which will be the content of Theorem \ref{inducedHochschildchainmap} in the next section.
\end{remark}

\begin{remark}
\label{RemarkAinftycategory}
 In the same way as a ring can be seen as an additive category with one single object, the notion of an $A_\infty$-algebra can be generalized to the notion of an $A_\infty$-category, such that an $A_\infty$-algebra is an $A_\infty$-category with one single object. 
 
 $A_\infty$-categories are discussed extensively and carefully by Paul Seidel in \cite[Part 1]{SeidelBook} and many of the results Seidel presents have obvious consequences for $A_\infty$-algebras. \index{$A_\infty$-categories}
\end{remark}

\section{\texorpdfstring{Hochschild homology of $A_\infty$-algebras}{Hochschild homology of A-infinity-algebras}}
\label{SectionHochschildHomology}

Hochschild homology has originally been defined for associative algebras and bimodules over them. The original definition has been extended to differential graded algebras. This notion of Hochschild homology has been used in \cite{Jones} to construct an isomorphism between loop space cohomology and Hochschild homology of singular cochains. For Hochschild homology of differential graded algebras see also \cite{AbbaspourSurvey} and \cite[Chapter 9]{Weibel}. 

The generalized definition of Hochschild homology for $A_\infty$-algebras which we are about to give reduces to the previously mentioned definition in the case of a differential graded algebra. 

Hochschild homology for $A_\infty$-algebras is also defined in \cite[Section 5]{SeidelHochschild}, \cite[Section 7.2]{KontsevichSoibelmanNotes} and \cite{PenkavaSchwarz}. Except for a minor change in the sign convention and the ordering of the tensor algebra, our definition corresponds to the one in \cite[Section 5]{SeidelHochschild}. \\

\emph{Throughout this section, let $(A = \bigoplus_{j \in \ZZ} A_j, (\mu_n)_{n \in \NN})$ be an $A_\infty$-algebra over $R$ and let $(M = \bigoplus_{j \in \ZZ} M_j, (\mu^M_{r,s})_{r,s \in \NN_0})$ be an $A_\infty$-bimodule over $A$. We further denote the index of $m \in M$ by $\mu(m):=\mu_M(m)$. } \\	

\index{Hochschild!chain complex} \nomenclature{$CH_*(A;M)$}{Hochschild chain complex of $A$ with coefficients in $M$} 
Consider the graded $R$-bimodule
\begin{equation*}
 CH_*(A;M) := \bigoplus_{n=0}^{\infty} M \otimes A^{\otimes n} \ ,
\end{equation*}
where a grading on $CH_*(A;M)$ is given by
\begin{equation*}
 CH_j(A;M) = \bigoplus_{n \in \NN_0} \bigoplus_{j= n-j_0-j_1-\dots-j_n} M_{j_0} \otimes A_{j_1} \otimes \dots \otimes A_{j_n} \ .
\end{equation*}

\begin{definition} \index{Hochschild!degree} \nomenclature{$\deg$}{Hochschild degree, degree of a graded map}
 For $a \in CH_*(A;M)$ we write $\deg(a) = j$ iff $a \in CH_j(A;M)$ and call it the \emph{degree} of $a$. Note that for all $m \in M$, $n \in \NN_0$ and $a_1,a_2,\dots,a_n\in A$ the degree of $m \otimes a_1 \otimes \dots \otimes a_n$ is explicitly given by
\begin{equation*}
 \deg(m \otimes a_1 \otimes \dots \otimes a_n) = n - \mu(m) - \sum_{j=1}^n \mu(a_j) = -\mu(m) - \sum_{j=1}^n \|a_j\| \ .
\end{equation*}
\end{definition}
\begin{remark}
 The degree on $CH_*(A;M)$ is best understood in terms of the shifted $A_\infty$-algebra $A[1]$. We may identify $CH_*(A;M)$ as a group with 
 $$ \bigoplus_{n=0}^\infty M \otimes A^{\otimes n} \ . $$ 
 The degree on $CH_*(A;M)$ then coincides with the usual product degree of this tensor algebra.
\end{remark}

In the following we construct a differential on $CH_*(A;M)$ which lowers the degree by one. The homology of the thus constructed chain complex will be the desired Hochschild homology. \\

Let $n \in \NN_0$ and $l \in \{1,2,\dots,n+1\}$. Define a module homomorphism by
\begin{align*}
 \mu^n_{0,l}: M \otimes A^{\otimes n} &\to M \otimes A^{\otimes (n-l+1)} \ , \\
  m \otimes a_1 \otimes \dots \otimes a_n &\mapsto \mu^M_{0,l-1}(m,a_1,\dots,a_{l-1}) \otimes a_l \otimes \dots \otimes a_n \ .
\end{align*}
For $1 \leq i \leq j \leq n$, we let $\maltese_i^j$ be given as in \eqref{EqDefMalteseij} while for each $i \in \NN_0$ we put 
\begin{equation*}
\maltese_0^i := \maltese_0^i(m \otimes a_1 \otimes \dots \otimes a_n) := \mu(m) + \sum_{q=1}^i \|a_q\|= \mu(m) + \maltese_1^i \ .
\end{equation*}
We just write $\maltese_i^j$ without further mentioning the element of the tensor product if it is clear which element we are referring to. \\

Let $n$ and $l$ be as above. For $i \in \{1,2,\dots,n-l+1\}$ we define
\begin{align*}
 \mu^n_{i,l}: M \otimes A^{\otimes n} &\to M \otimes A^{\otimes (n-l+1)} \ , \\
 m \otimes a_1 \otimes \dots \otimes a_n &\mapsto (-1)^{\maltese_{i-1}} m \otimes a_1 \otimes \dots \otimes a_{i-1} \otimes \mu_l(a_i,\dots,a_{i+l-1})\otimes a_{i+l} \otimes \dots \otimes a_n \ .
\end{align*}
For $i \in \{n-l+1,n-l+2,\dots,n\}$ we define
\begin{align*}
 &\mu^n_{i,l}: M \otimes A^{\otimes n} \to M \otimes A^{\otimes (n-l+1)} \ , \\
 &m \otimes a_1 \otimes \dots \otimes a_n \mapsto \\ &\qquad (-1)^{*_{i-1}} \mu^M_{n-i+1,i+l-n-2}(a_i,\dots,a_n,m,a_1,\dots,a_{i+l-n-2})\otimes a_{l-n+i-1} \otimes \dots \otimes a_{i-1} \ ,
 \end{align*}
 where we put: \nomenclature{$*_i$}{exponents of the signs defining the overlapping parts of the Hoch\-schild differential}
 \begin{equation*}
  *_{i-1} := *_{i-1}(m\otimes a_1 \otimes \dots \otimes a_n) := \maltese_0^{i-1} \cdot \maltese_i^n = \left(\mu(m) + \sum_{q=1}^{i-1} \|a_q\| \right) \left(\sum_{q=i}^n \|a_q\| \right) \ .
 \end{equation*}
We extend every map $\mu^n_{i,l}$ by zero to all of $CH_*(A;M)$. We further define $\mu^n_{i,l} := 0$ if $l > n+1$ or $i > n$ and consider the map
\begin{equation*}
 b_{i,l} := \sum_{n \in \NN_0} \mu^n_{i,l}
\end{equation*} \nomenclature{$b_{i,l}$}{component of the Hochschild differential}
for every $i \in \NN_0$, $l \in \NN$. 

\begin{definition} \index{Hochschild!differential} \nomenclature{$b$}{Hochschild differential}
 The map 
 \begin{equation*} 
  b: CH_*(A;M) \to CH_*(A;M) \ , \qquad b := \sum_{l \in \NN} \sum_{i \in \NN_0} b_{i,l} \ ,
 \end{equation*}
 is called \emph{the Hochschild differential of $A$ with coefficients in $M$.}
\end{definition}

\begin{remark}
\label{RemarkHochschildDiff}
 \begin{enumerate}
  \item Note that on elements of $M \otimes A^{\otimes n}$ for a fixed $n$ we have
  \begin{equation*}
   b(m\otimes a_1\otimes \dots \otimes a_n) = \sum_{l=1}^{n+1} \sum_{i=0}^n b_{i,l}(m \otimes a_1 \otimes \dots \otimes a_n) \ .
  \end{equation*}
  For reasons which become apparent when looking at the definition of the $b_{i,l}$, we further call 
 \begin{equation*}
  \sum_{l=1}^{n+1} \sum_{i=n-l+2}^n b_{i,l}(m \otimes a_1 \otimes \dots \otimes a_n)
 \end{equation*}
 \emph{the overlapping part} of $b$ on $m \otimes a_1 \otimes\dots \otimes  a_n$.
  \item If $M=A[1]$ is equipped with the $A_\infty$-bimodule structure from Remark \ref{RemarkBimoduleonA1}, then the Hochschild differential will take a slightly simpler form. Indeed, we compute for $a_0,a_1,\dots,a_n \in A$ that
  \begin{align*}
   &b(a_0\otimes a_1 \otimes \dots \otimes a_n) \\ 
   &= \sum_{l=1}^{n+1} \Big( \sum_{i=0}^{n-l+1} (-1)^{\maltese_0^{i-1}} a_0 \otimes a_1 \otimes \dots \otimes \mu_l(a_i,\dots,a_{i+l-1}) \otimes a_{i+l} \otimes \dots \otimes a_n \\
     &+\sum_{i=n-l+2}^n (-1)^{*_{i-1}} \mu_l(a_i,\dots,a_n,a_0,a_1,\dots,a_{i+l-n-2})\otimes a_{i+l-n-1} \otimes \dots \otimes a_{i-1} \Big) \ ,
  \end{align*}
  where $\maltese_0^{i-1} = \sum_{j=0}^{i-1} \|a_j\|$ and $*_{i-1} = \Bigl(\sum_{j=0}^{i-1} \|a_j\|\Bigr)\cdot \Bigl(\sum_{j=i}^n\|a_j\|\Bigr)$.
 \end{enumerate}
\end{remark}

The proof of the next theorem makes heavy use of the defining equations of the $A_\infty$-algebra $A$ and the $A_\infty$-bimodule $M$.

\begin{theorem}
\label{HochschildDifferential} \index{Hochschild!differential}
 $b \circ b =0$, i.e. $b$ is a differential on $CH_*(A;M)$ which lowers degree by one.
\end{theorem}

\begin{definition}\index{Hochschild!chain complex} \index{Hochschild!homology} \nomenclature{$HH_*(A;M)$}{Hochschild homology of $A$ with coefficients in $M$}
 The homology group of the chain complex $(CH_*(A;M),b)$ is called \emph{the Hochschild homology of $A$ with coefficients in $M$} and denoted by $HH_*(A;M)$.
\end{definition}

\begin{proof}[Proof of Theorem \ref{HochschildDifferential}]
 The statement about the degrees is not difficult to check using that every map $\mu_l$ has degree $2-l$ and every map $\mu^M_{r,s}$ has degree $1-r-s$. In terms of (reduced) indices, this means:
 \begin{align}
  &\left\|\mu_l(a_1,\dots,a_l)\right\| = \sum_{q=1}^l \|a_q\|+1 \ \ \forall a_1,\dots,a_l \in A  \ , \label{degreeofmu} \\
  &\mu\left(\mu^M_{r,s}(a_1,\dots,a_r,m,a_{r+1},\dots,a_{r+s}) \right) = \mu(m) + \sum_{q=1}^{r+s} \|a_q\|+1 \ \ \forall a_1,\dots,a_{r+s} \in A, m \in M \notag \ .
 \end{align}
This implies for all $m \in M$, $a_1,\dots,a_n \in A$, $l \in \{1,2,\dots,n+1\}$ and $i \in \{0,1,\dots,n\}$ that
\begin{equation*}
 \deg(b_{i,l}(m\otimes a_1 \otimes \dots \otimes a_n)) = - \mu(m) - \sum_{j=1}^n \|a_j\|-1 = \deg(m \otimes a_1 \otimes \dots \otimes a_n) -1
\end{equation*}
Hence $b$ lowers degree by one. \bigskip

We will show $b \circ b=0$ by a long, but straightforward computation using the defining equations for $A$ as an $A_\infty$-algebra and $M$ as an $A_\infty$-bimodule over $A$.  \bigskip

\emph{In the following we assume w.l.o.g. that $M=A[1]$, equipped with the $A_\infty$-bimodule structure from Remark \ref{RemarkBimoduleonA1}. This helps us to simplify our notation in the sense of Remark \ref{RemarkHochschildDiff}. The attentive reader will have no difficulties extending the proof to the general case.} \bigskip

It suffices to check $b \circ b=0$ on elements of $CH_*(A;M)$ of type $m \otimes a_1 \otimes \dots \otimes a_n$ for $m,a_1,\dots,a_n \in A$, since all the results extend $R$-linearly. We start by writing down $b \circ b$ on such an element in greater detail. The linearity of $b$ implies
\begin{align*}
 (b\circ b)(m \otimes a_1 \otimes \dots \otimes a_n) 
  &= b \left( \sum_{l_1=1}^{n+1} \sum_{i_1=0}^n b_{i_1,l_1}(m\otimes a_1 \otimes \dots \otimes a_n) \right) \\
 &= \sum_{l_1=1}^{n+1} \sum_{i_1=0}^n b(b_{i_1,l_1}(\mas)) \\
 &= \sum_{l_1=1}^{n+1} \sum_{i_1=0}^n \sum_{l_2=1}^{n-l_1+2} \sum_{i_2=0}^{n-l_1+1} \left(b_{i_2,l_2} \circ b_{i_1,l_1}\right)(\mas) \\
 &=: A_1 + A_2 + A_3 + A_4 \ , 
 \end{align*}
 where we put
 \begin{align*}
  A_1 &:= \sum_{l_1=1}^{n+1} \sum_{i_1=0}^{n-l_1+1} \sum_{l_2=1}^{n-l_1+2} \sum_{i_2=0}^{n-l_1-l_2+2} \left(b_{i_2,l_2} \circ b_{i_1,l_1}\right)(\mas) \ , \\
  A_2 &:= \sum_{l_1=1}^{n+1} \sum_{i_1=0}^{n-l_1+1} \sum_{l_2=1}^{n-l_1+2} \sum_{i_2=n-l_1-l_2+3}^{n-l_1+1} \left(b_{i_2,l_2} \circ b_{i_1,l_1}\right)(\mas) \ , \\
  A_3 &:= \sum_{l_1=1}^{n+1} \sum_{i_1=n-l_1+2}^n \sum_{l_2=1}^{n-l_1+2} \sum_{i_2=0}^{n-l_1-l_2+2} \left(b_{i_2,l_2} \circ b_{i_1,l_1}\right)(\mas) \ , \\
  A_4 &:= \sum_{l_1=1}^{n+1} \sum_{i_1=n-l_1+2}^n \sum_{l_2=1}^{n-l_1+2} \sum_{i_2=n-l_1-l_2+3}^{n-l_1+1} \left(b_{i_2,l_2} \circ b_{i_1,l_1}\right)(\mas) \ . \\
 \end{align*}
More intuitively, $A_1$ consists of all summands for which both $b_{i_1,l_1}$ and $b_{i_2,l_2}$ do not belong to the overlapping part of the differential, while $A_2$ and $A_3$ collect all summands for which either one of the two maps belongs to it and $A_4$ contains all summands for which both maps are overlapping.
 
 We have to treat all four sums seperately to exploit the combinatorics of the Hochschild differential. First we show that
 \begin{equation*}
  A_1 \stackrel{!}{=} 0 \ .
 \end{equation*}
Consider fixed numbers $l_1 \in \{1,2,\dots,n+1\}$ and $i_1 \in \{0,1,\dots,n-l_1+1\}$. We can write the corresponding inner sum of $A_1$ as follows:
\begin{align*}
 &\sum_{l_2=1}^{n-l_1+2} \sum_{i_2=0}^{n-l_1-l_2+2} \left(b_{i_2,l_2} \circ b_{i_1,l_1}\right)(\mas) \\
 &= \sum_{l_2=1}^{n-l_1+2} \sum_{i_2=0}^{n-l_1-l_2+2} (-1)^{\maltese_0^{i_1-1}} \\ &\qquad b_{i_2,l_2}(m \otimes a_1 \otimes \dots \otimes a_{i_1-1} \otimes \mu_{l_1}(a_{i_1},\dots,a_{i_1+l_1-1})\otimes a_{i_1+l_1}\otimes \dots \otimes a_n)  \\
 &= \sum_{l_2=1}^{n-l_1+2} \left( \sum_{i_2=0}^{i_1-l_2} (-1)^{\maltese_0^{i_1-1} + \maltese_0^{i_2-1}} \right. \\ 
     &\qquad m \otimes a_1 \otimes \dots \otimes \mu_{l_2}(a_{i_2},\dots,a_{i_2+l_2-1})\otimes \dots \otimes \mu_{l_1}(a_{i_1},\dots,a_{i_1+l_1-1})\otimes \dots \otimes a_n \\
  &+ \sum_{i_2=i_1-l_2+1}^{i_1} (-1)^{\maltese_0^{i_1-1} + \maltese_0^{i_2-1}} \\ 
     &\qquad m \otimes a_1 \otimes \dots \otimes \mu_{l_2}(a_{i_2},\dots,\mu_{l_1}(a_{i_1},\dots,a_{i+l_1-1}),\dots,a_{i_2+l_1+l_2-2}) \otimes \dots \otimes a_n \\
  &+\sum_{i_2=i_1+l_1}^{n-l_2+1} (-1)^{\maltese_0^{i_1-1} + \maltese_0^{i_1-1}+\|\mu_{l_1}(a_{i_1},\dots,a_{i_1+l_1-1})\|+\maltese_{i_1+l_1}^{i_2-1}} \\ 
     & \left. \ \ \phantom{\sum_{i=0}^n}  m \otimes a_1 \otimes \dots \otimes \mu_{l_1}(a_{i_1},\dots,a_{i_1+l_1-1})\otimes \dots \otimes \mu_{l_2}(a_{i_2},\dots,a_{i_2+l_2-1})\otimes \dots \otimes a_n ) \ \right) \ .
\end{align*}
Let '$\equiv$' denote congruence modulo two. Then (\ref{degreeofmu}) implies:
\begin{align*}
 \maltese_0^{i_1-1} + \maltese_0^{i_2-1} &\equiv \maltese_{i_2}^{i_1-1} \ , &\text{ if } i_2 \leq i_1 \ , \\
 \maltese_0^{i_1-1} + \maltese_0^{i_1-1} + \|\mu_{l_1}(a_{i_1},\dots,a_{i_1+l_1-1})\|+\maltese_{i_1+l_1}^{i_2-1} \ , &\equiv \maltese_{i_1}^{i_2-1}+1  &\text{ if } i_2 > i_1 \ .
\end{align*}
By the last computation we obtain a splitting
\begin{equation*}
 A_1 = A_{11}+A_{12}+A_{13} \ ,
\end{equation*}
where we put:
\begin{align*}
 A_{11} :=  &\sum_{l_1=1}^{n+1} \sum_{i_1=0}^{n-l_1+1} \sum_{l_2=1}^{n-l_1+2}  \sum_{i_2=0}^{i_1-l_2} (-1)^{\maltese_{i_2}^{i_1-1}} \\ 
    &m \otimes a_1 \otimes \dots \otimes \mu_{l_2}(a_{i_2},\dots,a_{i_2+l_2-1})\otimes \dots \otimes \mu_{l_1}(a_{i_1},\dots,a_{i_1+l_1-1})\otimes \dots \otimes a_n \ , \\ 
    A_{12} := &\sum_{l_1=1}^{n+1} \sum_{i_1=0}^{n-l_1+1} \sum_{l_2=1}^{n-l_1+2}\sum_{i_2=i_1-l_2+1}^{i_1} (-1)^{\maltese_{i_2}^{i_1-1}} \\ 
    &m \otimes a_1 \otimes \dots \otimes \mu_{l_2}(a_{i_2},\dots,\mu_{l_1}(a_{i_1},\dots,a_{i+l_1-1}),\dots,a_{i_2+l_1+l_2-2})\otimes \dots \otimes a_n \ , \\
    A_{13} := &- \sum_{l_1=1}^{n+1} \sum_{i_1=0}^{n-l_1+1} \sum_{l_2=1}^{n-l_1+2} \sum_{i_2=i_1+l_1}^{n-l_2+1} (-1)^{\maltese_{i_1}^{i_2-1}} \\ 
     &\quad m \otimes a_1 \otimes \dots \otimes \mu_{l_1}(a_{i_1},\dots,a_{i_1+l_1-1})\otimes \dots \otimes \mu_{l_2}(a_{i_2},\dots,a_{i_2+l_2-1})\otimes \dots \otimes a_n \ .
\end{align*}
After interchanging the names of $i_1$ and $i_2$, $l_1$ and $l_2$, resp., we can permute the sums in $A_{11}$ as follows (where for brevity's sake the dots '$\cdots$' always indicate that we are summing over the same term as in the above line):
\begin{align*}
 A_{11} &= \sum_{l_2=1}^{n+1} \sum_{i_2=0}^{n-l_2+1} \sum_{l_1=1}^{n-l_2+2}  \sum_{i_1=0}^{i_2-l_1} (-1)^{\maltese_{i_1}^{i_2-1}} \\ 
    &\quad m \otimes a_1 \otimes \dots \otimes \mu_{l_1}(a_{i_1},\dots,a_{i_1+l_1-1})\otimes \dots \otimes \mu_{l_2}(a_{i_1},\dots,a_{i_2+l_2-1})\otimes \dots \otimes a_n \\ 
   &= \sum_{l_1=1}^{n+1} \sum_{l_2=1}^{n-l_1+2} \sum_{i_1=0}^{n-l_1+1} \sum_{i_2=i_1+l_1}^{n-l_2+1} \cdots \\
   &= \sum_{l_1=1}^{n+1} \sum_{i_1=0}^{n-l_1+1} \sum_{l_2=1}^{n-l_1+2} \sum_{i_2=i_1+l_1}^{n-l_2+1}  (-1)^{\maltese_{i_1}^{i_2-1}} 	\\
&\quad m \otimes a_1 \otimes \dots \otimes \mu_{l_1}(a_{i_1},\dots,a_{i_1+l_1-1})\otimes \dots \otimes \mu_{l_2}(a_{i_1},\dots,a_{i_2+l_2-1})\otimes \dots \otimes a_n =-A_{13} \ .
\end{align*}
Thus, we have shown that $A_{11}+A_{13}=0$. Next we consider $A_{12}$. We reorder the sums in $A_{12}$ as follows:
\begin{align*}
 A_{12} &= \sum_{l_1=1}^{n+1} \sum_{i_1=0}^{n-l_1+1} \sum_{i_2=i_1-l_2+1}^{i_1} (-1)^{\maltese_{i_2}^{i_1-1}} \\ 
    &\quad m \otimes a_1 \otimes \dots \otimes \mu_{l_2}(a_{i_2},\dots,\mu_{l_1}(a_{i_1},\dots,a_{i+l_1-1}),\dots,a_{i_2+l_1+l_2-2}) \otimes \dots \otimes a_n \\
     &= \sum_{l_1=1}^{n+1} \sum_{i_1=0}^n \sum_{i_2=0}^{i_1} \sum_{l_2=i_1-i_2+1}^{n-l_1+2} \cdots 
     = \sum_{i_2=0}^n \sum_{d=1}^{n+1} \sum_{l_1+l_2=d+1} \sum_{i_1=i_2}^{i_2+l_2+1}  \cdots \\     
     &= \sum_{i_2=0}^n \sum_{d=1}^{n+1} m \otimes a_1 \otimes \dots \otimes a_{i_2-1} \otimes \\
     &\quad \left( \sum_{l_1+l_2=d+1} \sum_{i_1=i_2}^{i_2+l_2+1} (-1)^{\maltese_{i_2}^{i_1-1}}  \mu_{l_2}(a_{i_2},\dots,a_{i_1-1},\mu_{l_1}(a_{i_1},\dots,a_{i+l_1-1}),a_{i_1+l_1},\dots,a_{i_2+d-1}) \right) \\ 
     &\quad \quad \quad \quad \quad \quad  \otimes a_{i_2+d}\otimes \dots \otimes a_n = 0 \ ,
\end{align*}
since every inner sum coincides with the left-hand side of the $d$-th defining equation of the $A_\infty$-algebra $A$ for $a_{i_2},\dots,a_{i_2+d-1}\in A$. So we have derived
\begin{equation*}
 A_1 = A_{11}+A_{12}+A_{13}=0 \ .
\end{equation*}
It remains to show that
\begin{equation*}
 A_2+A_3+A_4 = 0 \ .
\end{equation*}
We start by writing down $A_2$ explicitly. Namely, we write it as
\begin{align*}
 A_2 &=  \sum_{l=1}^{n+1} \sum_{i_1=0}^{n-l_1+1} \sum_{l_2=1}^{n-l_1+2} \sum_{i_2=n-l_1-l_2+3}^{n-l_1+1} \left(b_{i_2,l_2} \circ b_{i_1,l_1}\right)(\mas) \\
   &= \sum_{l=1}^{n+1} \sum_{i_1=0}^{n-l_1+1} \sum_{i_2=0}^{n-l_1+1} \sum_{l_2=n-i_2-l_1+3}^{n-l_1+2} (-1)^{\maltese_0^{i_1-1}} \\ &\qquad b_{i_2,l_2}(m \otimes a_1 \otimes \dots \otimes \mu_{l_1}(a_{i_1},\dots,a_{i_1+l_1-1}) \otimes \dots \otimes a_n) \\
   &= A_{21}+A_{22}+A_{23} \ ,
\end{align*}
where we put
\begin{align*}
 A_{21} &:= \sum_{l_1=1}^{n+1} \sum_{i_1=0}^{n-l_1+1} \sum_{i_2=0}^{i_1} \sum_{l_2=n-i_2-l_1+3}^{n-l_1+2} (-1)^{\maltese_0^{i_1-1}} \\ 
 &\qquad \qquad \qquad b_{i_2,l_2}(m \otimes a_1 \otimes \dots \otimes \mu_{l_1}(a_{i_1},\dots,a_{i_1+l_1-1}) \otimes \dots \otimes a_n) \ , \\
 A_{22} &:= \sum_{l_1=1}^{n+1} \sum_{i_1=0}^{n-l_1+1} \sum_{i_2=i_1+1}^{n-l_1+1} \sum_{l_2=n-i_2-l_1+3}^{n+i_1-i_2-l_1+2} (-1)^{\maltese_0^{i_1-1}}\\ 
 &\qquad \qquad \qquad b_{i_2,l_2}(m \otimes a_1 \otimes \dots \otimes \mu_{l_1}(a_{i_1},\dots,a_{i_1+l_1-1}) \otimes \dots \otimes a_n) \ , \\
 A_{23} &:= \sum_{l_1=1}^{n+1} \sum_{i_1=0}^{n-l_1+1} \sum_{i_2=i_1+1}^{n-l_1+1} \sum_{l_2=n+i_1-i_2-l_1+3}^{n-l_1+2} (-1)^{\maltese_0^{i_1-1}} \\ 
 &\qquad \qquad \qquad b_{i_2,l_2}(m \otimes a_1 \otimes \dots \otimes \mu_{l_1}(a_{i_1},\dots,a_{i_1+l_1-1}) \otimes \dots \otimes a_n) \ .
\end{align*}
The meaning of this decomposition will clarify if we write down the maps $b_{i_2,l_2}$ in greater detail in the three different cases. $A_{21}$ corresponds to the case that the $\mu_{l_1}$-term is among the inputs of the map $\mu_{l_2}$ such that it stands \emph{in front of} $m$, while $A_{23}$ corresponds to the case that the $\mu_{l_1}$-term is also among the inputs of the map $\mu_{l_2}$, but that it stands \emph{behind} $m$  with respect to the ordering of the inputs. Finally, $A_{22}$ corresponds to the case that the $\mu_{l_1}$-term is \emph{not} among the inputs of $\mu_{l_2}$. \\

We continue by explicitly describing $A_{21}$ as
\begin{align*}
 &A_{21} = \sum_{l_1=1}^{n+1} \sum_{i_1=0}^{n-l_1+1} \sum_{i_2=0}^{i_1} \sum_{l_2=n-i_2-l_1+3}^{n-l_1+2} (-1)^{\maltese_0^{i_1-1}(\mas)+*_{i_2-1}(b_{i_1,l_1}(\mas))} \\
        &\qquad \mu_{l_2}(a_{i_2},\dots,\mu_{l_1}(a_{i_1},\dots,a_{i_1+l_1-1}),\dots,a_n,m,a_1,\dots , a_{i_2+l_1+l_2-n-3}) \otimes \dots \otimes a_{i_2-1} \ .
\end{align*}
The sign simplifies according to part 1 of Lemma \ref{SignsHochschildProof} as 
\begin{align*}
 A_{21} &= \sum_{l_1=1}^{n+1} \sum_{i_1=0}^{n-l_1+1} \sum_{i_2=0}^{i_1} \sum_{l_2=n-i_2-l_1+3}^{n-l_1+2} (-1)^{*_{i_2-1}+\maltese_{i_2}^{i_1-1}} \\
        &\mu_{l_2}(a_{i_2},\dots,\mu_{l_1}(a_{i_1},\dots,a_{i_1+l_1-1}),\dots,a_n,m,a_1,\dots , a_{i_2+l_1+l_2-n-3}) \otimes \dots \otimes a_{i_2-1} \\
        &= \sum_{l_1=1}^{n+1} \sum_{i_2=0}^{n-l_1+1} \sum_{i_1=i_2}^{n-l_1+1} \sum_{l_2=n-i_2-l_1+3}^{n-l_1+2} \cdots \\
        &= \sum_{i_2=0}^{n} \sum_{l_1=1}^{n-i_2+1} \sum_{i_1=i_2}^{n-l_1+1} \sum_{l_2=n-i_2-l_1+3}^{n-l_1+2} \cdots \\
   &= \sum_{i_2=0}^{n} \sum_{l_1=1}^{n-i_2+1} \sum_{l_2=n-i_2-l_1+3}^{n-l_1+2} \sum_{i_1=i_2}^{n-l_1+1} (-1)^{*_{i_2-1}+\maltese_{i_2}^{i_1-1}} \\
        &\mu_{l_2}(a_{i_2},\dots,\mu_{l_1}(a_{i_1},\dots,a_{i_1+l_1-1}),\dots,a_n,m,a_1,\dots , a_{i_2+l_1+l_2-n-3}) \otimes \dots \otimes a_{i_2-1} \ .
\end{align*}
Shifting the index $l_2$ by $l_1-1$ and renaming it by $d$, we derive: 
\begin{align}
 A_{21} &= \sum_{i_2=0}^{n} \sum_{l_1=1}^{n-i_2+1} \sum_{d=n-i_2+2}^{n+1} \sum_{i_1=i_2}^{n-l_1+1} (-1)^{*_{i_2-1}+\maltese_{i_2}^{i_1-1}} \notag \\
           &\mu_{d+1-l_1}(a_{i_2},\dots,\mu_{l_1}(a_{i_1},\dots,a_{i_1+l_1-1}),\dots,a_n,m,a_1,\dots , a_{i_2+d-n-2}) \otimes\dots \otimes a_{i_2-1} \notag \\
       &= \sum_{i_2=0}^{n} \sum_{d=n-i_2+2}^{n+1} \sum_{l_1=1}^{n-i_2+1} \sum_{i_1=i_2}^{n-l_1+1} (-1)^{*_{i_2-1}+\maltese_{i_2}^{i_1-1}} \notag \\
           &\mu_{d+1-l_1}(a_{i_2},\dots,\mu_{l_1}(a_{i_1},\dots,a_{i_1+l_1-1}),\dots,a_n,m,a_1,\dots , a_{i_2+d-n-2}) \otimes\dots \otimes a_{i_2-1} \label{a21} .
\end{align}
We conclude our treatment of $A_{21}$ for the moment and equation (\ref{a21}) will become important later. We continue by considering $A_{23}$ in a similar way and reformulate it as 
\begin{align*}
A_{23} &= \sum_{l_1=1}^{n+1} \sum_{i_1=0}^{n-l_1+1} \sum_{i_2=i_1+1}^{n-l_1+1} \sum_{l_2=n+i_1-i_2-l_1+3}^{n-l_1+2} (-1)^{\maltese_0^{i_1-1}+*_{i_2-1}(b_{i_1,l_1}(\mas))} \\ 
        &\mu_{l_2}(a_{i_2+l_1-1},\dots,a_n,m,a_1,\dots,\mu_{l_1}(a_{i_1},\dots,a_{i_1+l_1-1}),\dots, a_{i_2+l_2-n-2}) \otimes \dots \otimes a_{i_2+l_1-2} \ .
\end{align*}
We can again simplify the sign, see part 2 of Lemma \ref{SignsHochschildProof}. If we further shift the index $i_2$ by $l_1-1$, we will derive:
\begin{align*}
A_{23} &= \sum_{l_1=1}^{n+1} \sum_{i_1=0}^{n-l_1+1} \sum_{i_2=i_1+l_1}^{n} \sum_{l_2=n+i_1-i_2+2}^{n-l_1+2} (-1)^{*_{i_2-1}+\maltese_{i_2}^n + \maltese_0^{i_1-1}} \\ 
        &\mu_{l_2}(a_{i_2},\dots,a_n,m,a_1,\dots,\mu_{l_1}(a_{i_1},\dots,a_{i_1+l_1-1}),\dots, a_{i_2+l_1+l_2-n-3}) \otimes \dots \otimes a_{i_2-1} \ .
\end{align*}
Another shift of $l_2$ by $l_1-1$ and renaming it by $d$ enables us to write
\begin{align}
A_{23} &:= \sum_{l_1=1}^{n+1} \sum_{i_1=0}^{n-l_1+1} \sum_{i_2=i_1+l_1}^{n} \sum_{d=n+i_1-i_2+l_1+1}^{n+1} (-1)^{*_{i_2-1}+\maltese_{i_2}^n + \maltese_0^{i_1-1}} \notag \\ 
        &\mu_{d+1-l_1}(a_{i_2},\dots,a_n,m,a_1,\dots,\mu_{l_1}(a_{i_1},\dots,a_{i_1+l_1-1}),\dots, a_{i_2+d-n-2}) \otimes \dots \otimes a_{i_2-1} \notag \\
        &= \sum_{l_1=1}^{n+1} \sum_{i_2=l_1}^{n} \sum_{i_1=0}^{i_2-l_1} \sum_{d=n+i_1-i_2+l_1+1}^{n+1} \cdots \notag \\
        &= \sum_{i_2=0}^{n} \sum_{l_1=1}^{i_2} \sum_{d=n-i_2+l_1+1}^{n+1} \sum_{i_1=0}^{i_2+d-n-l_1-1} \cdots \notag \\
        &= \sum_{i_2=0}^{n} \sum_{d=n-i_2+2}^{n+1} \sum_{l_1=1}^{i_2+d-n-1} \sum_{i_1=0}^{i_2+d-n-l_1-1} (-1)^{*_{i_2-1}+\maltese_{i_2}^n + \maltese_0^{i_1-1}} \label{a23} \\ 
        &\mu_{d+1-l_1}(a_{i_2},\dots,a_n,m,a_1,\dots,\mu_{l_1}(a_{i_1},\dots,a_{i_1+l_1-1}),\dots, a_{i_2+d-n-2})\otimes \dots \otimes a_{i_2-1} \ . \notag
\end{align}
We have brought $A_{21}$ and $A_{23}$ into a similar form and continue by applying similar techniques to $A_4$ to find a general scheme for the sum $A_{21}+A_{23}+A_4$.
\begin{align*}
 A_4 &= \sum_{l_1=1}^{n+1} \sum_{i_1=n-l_1+2}^n \sum_{l_2=1}^{n-l_1+2} \sum_{i_2=n-l_1-l_2+3}^{n-l_1+1} \left(b_{i_2,l_2} \circ b_{i_1,l_1}\right)(\mas) \\
  &=  \sum_{l_1=1}^{n+1} \sum_{i_1=n-l_1+2}^n \sum_{l_2=1}^{n-l_1+2} \sum_{i_2=n-l_1-l_2+3}^{n-l_1+1} (-1)^{*_{i_1-1}} \\
  &b_{i_2,l_2}\left(\mu_{l_1}(a_{i_1},\dots,a_n,m,a_1,\dots,a_{i_1+l_1-n-2})\otimes a_{i_1+l_1-n-1}\otimes \dots \otimes a_{i_1-1}\right) \\
  &=  \sum_{l_1=1}^{n+1} \sum_{i_1=n-l_1+2}^n \sum_{l_2=1}^{n-l_1+2} \sum_{i_2=n-l_1-l_2+3}^{n-l_1+1} (-1)^{*_{i_1-1}(\mas)+*_{i_2-1}(b_{i_1,l_1}(\mas))} \\
  &\mu_{l_2}(a_{i_1+i_2+l_1-n-2},\dots,\mu_{l_1}(a_{i_1},\dots,a_n,m,a_1,\dots,a_{i_1+l_1-n-2}),\dots)\otimes \dots \otimes a_{i_1+i_2+l_1-n-3} \ .
\end{align*}
Applying the sign computation from part 3 of Lemma \ref{SignsHochschildProof}, we obtain:
\begin{align*}
A_4 &=  \sum_{l_1=1}^{n+1} \sum_{i_1=n-l_1+2}^n \sum_{l_2=1}^{n-l_1+2} \sum_{i_2=n-l_1-l_2+3}^{n-l_1+1} (-1)^{*_{i_1+i_2+l_1-n-3}+\maltese_{i_1+i_2+l_1-n-2}^{i_1-1}} \\
  &\mu_{l_2}(a_{i_1+i_2+l_1-n-2},\dots,\mu_{l_1}(a_{i_1},\dots,a_n,m,a_1,\dots,a_{i_1+l_1-n-2}),\dots)\otimes \dots \otimes a_{i_1+i_2+l_1-n-3} \ .
\end{align*}
Shifting the index $i_2$ by $i_1+l_1-n-2$ provides:
\begin{align*}
 A_4 &= \sum_{l_1=1}^{n+1} \sum_{i_1=n-l_1+2}^n \sum_{l_2=1}^{n-l_1+2} \sum_{i_2=i_1-l_2+1}^{i_1-1} (-1)^{*_{i_2-1}+\maltese_{i_2}^{i_1-1}} \\
   & \qquad \mu_{l_2}(a_{i_2},\dots,\mu_{l_1}(a_{i_1},\dots,a_n,m,a_1,\dots,a_{i_1+l_1-n-2}),\dots,a_{i_2+l_1+l_2-n-3})\otimes \dots \otimes a_{i_2-1} \\
   &=\sum_{l_1=1}^{n+1} \sum_{i_1=n-l_1+2}^n \sum_{i_2=i_1+l_1-n-1}^{i_1-1} \sum_{l_2=i_1-i_2+1}^{n-l_1+2} \cdots 
\end{align*}
Once again shifting $l_2$ by $l_1-1$ and renaming it by $d$ gives:
\begin{align*}
 A_4 &= \sum_{l_1=1}^{n+1} \sum_{i_1=n-l_1+2}^n \sum_{i_2=i_1+l_1-n-1}^{i_1-1} \sum_{d=i_1-i_2+l_1}^{n+1} (-1)^{*_{i_2-1}+\maltese_{i_2}^{i_1-1}} \\
   &\mu_{d+1-l_1}(a_{i_2},\dots,\mu_{l_1}(a_{i_1},\dots,a_n,m,a_1,\dots,a_{i_1+l_1-n-2}),\dots,a_{i_2+d-n-2})\otimes \dots \otimes a_{i_2-1} \\
   &= A_{41}+A_{42} \ ,
\end{align*}
where we put:
\begin{align*}
 A_{41} &:= \sum_{l_1=1}^{n+1} \sum_{i_1=n-l_1+2}^n \sum_{i_2=i_1+l_1-n-1}^{n-l_1+1} \sum_{d=i_1-i_2+l_1}^{n+1} (-1)^{*_{i_2-1}+\maltese_{i_2}^{i_1-1}} \\
   &\mu_{d+1-l_1}(a_{i_2},\dots,\mu_{l_1}(a_{i_1},\dots,a_n,m,a_1,\dots,a_{i_1+l_1-n-2}),\dots,a_{i_2+d-n-2})\otimes \dots \otimes a_{i_2-1} \\
 A_{42} &:= \sum_{l_1=1}^{n+1} \sum_{i_1=n-l_1+2}^n \sum_{i_2=n-l_1+2}^{i_1-1} \sum_{d=i_1-i_2+l_1}^{n+1} (-1)^{*_{i_2-1}+\maltese_{i_2}^{i_1-1}} \\
   &\mu_{d+1-l_1}(a_{i_2},\dots,\mu_{l_1}(a_{i_1},\dots,a_n,m,a_1,\dots,a_{i_1+l_1-n-2}),\dots,a_{i_2+d-n-2})\otimes \dots \otimes a_{i_2-1} \ .
\end{align*}
For $A_{41}$ we obtain:
\begin{align}
A_{41} &= \sum_{l_1=1}^{n+1} \sum_{i_2=0}^{n-l_1+1} \sum_{i_1=n-l_1+2}^{n+1+i_2-l_1}  \sum_{d=i_1-i_2+l_1}^{n+1} (-1)^{*_{i_2-1}+\maltese_{i_2}^{i_1-1}} \notag \\
   &\mu_{d+1-l_1}(a_{i_2},\dots,\mu_{l_1}(a_{i_1},\dots,a_n,m,a_1,\dots,a_{i_1+l_1-n-2}),\dots,a_{i_2+d-n-2})\otimes \dots \otimes a_{i_2-1}  \notag \\
    &= \sum_{i_2=0}^{n} \sum_{l_1=1}^{n-i-2+1} \sum_{d=n-i_2+2}^{n+1} \sum_{i_1=n-l_1+2}^{i_2+d-l_1}\cdots \notag \\
\Rightarrow A_{41} &= \sum_{i_2=0}^{n} \sum_{d=n-i_2+2}^{n+1} \sum_{l_1=1}^{n-i-2+1} \sum_{i_1=n-l_1+2}^{i_2+d-l_1}(-1)^{*_{i_2-1}+\maltese_{i_2}^{i_1-1}} \label{a41} \\
    &\mu_{d+1-l_1}(a_{i_2},\dots,\mu_{l_1}(a_{i_1},\dots,a_n,m,a_1,\dots,a_{i_1+l_1-n-2}),\dots,a_{i_2+d-n-2})\otimes \dots \otimes a_{i_2-1}  \ .\notag
\end{align}
Combining (\ref{a21}) with (\ref{a41}) results in:
\begin{align*}
&A_{21}+A_{41} = \sum_{i_2=0}^{n} \sum_{d=n-i_2+2}^{n+1} \sum_{l_1=1}^{n-i-2+1} (-1)^{*_{i_2-1}} \\ 
&\left( \sum_{i_1=i_2}^{n-l_1+1} (-1)^{\maltese_{i_2}^{i_1-1}} \mu_{d+1-l_1}(a_{i_2},\dots,\mu_{l_1}(a_{i_1},\dots,a_{i_1+l_1-1}),\dots, a_n,m,a_1,\dots,a_{i_2+d-n-2}) \right. \\
&+ \left. \sum_{i_1=n-l_1+2}^{i_2+d-l_1} (-1)^{\maltese_{i_2}^{i_1-1}} \mu_{d+1-l_1}(a_{i_2},\dots,\mu_{l_1}(a_{i_1},\dots,a_n,m,a_1,\dots,a_{i_1+l_1-n-2}),\dots,a_{i_2+d-n-2}) \right) \\
&\qquad \otimes a_{i_2+d-n-1} \otimes \dots \otimes a_{i_2-1} \ ,
\end{align*}
which, by abuse of notation, we write as follows:
\begin{align}
&A_{21}+A_{41} = \sum_{i_2=0}^{n} \sum_{d=n-i_2+2}^{n+1} \sum_{l_1=1}^{n-i-2+1} \sum_{i_1=i_2}^{i_2+d-l_1} (-1)^{*_{i_2-1}+\maltese_{i_2}^{i_1-1}} \label{a21a41} \\
 &\quad \mu_{d+1-l_1}(a_{i_2},\dots,\mu_{l_1}(a_{i_1},\dots,a_n,m,a_1,\dots,a_{i_1+l_1-n-2}),\dots,a_{i_2+d-n-2})\otimes \dots \otimes a_{i_2-1} \ . \notag
\end{align}
Considering $A_{42}$, we arrive at
\begin{align}
  A_{42} &= \sum_{l_1=1}^{n+1}  \sum_{i_2=n-l_1+2}^{n} \sum_{i_1=i_2+1}^n \sum_{d=i_1-i_2+l_1}^{n+1} (-1)^{*_{i_2-1}+\maltese_{i_2}^{i_1-1}} \notag  \\
   &\mu_{d+1-l_1}(a_{i_2},\dots,\mu_{l_1}(a_{i_1},\dots,a_n,m,a_1,\dots,a_{i_1+l_1-n-2}),\dots,a_{i_2+d-n-2})\otimes \dots \otimes a_{i_2-1} \notag \\
    &= \sum_{i_2=0}^{n} \sum_{l_1=n-i_2+2}^{n+1} \sum_{i_1=i_2+1}^n \sum_{d=i_1-i_2+l_1}^{n+1} \cdots = \sum_{i_2=0}^{n} \sum_{l_1=n-i_2+2}^{n+1} \sum_{d=l_1+1}^{n+1} \sum_{i_1=i_2+1}^{i_2+d-l_1}  \cdots \notag  \\
 &= \sum_{i_2=0}^{n} \sum_{d=n-i_1+2}^{n+1} \sum_{l_1=n-i_2+2}^{d-1} \sum_{i_1=i_2+1}^{i_2+d-l_1}  (-1)^{*_{i_2-1}+\maltese_{i_2}^{i_1-1}} \label{a42} \\
 &\mu_{d+1-l_1}(a_{i_2},\dots,\mu_{l_1}(a_{i_1},\dots,a_n,m,a_1,\dots,a_{i_1+l_1-n-2}),\dots,a_{i_2+d-n-2})\otimes \dots \otimes a_{i_2-1} \ . \notag
\end{align}
From (\ref{a21a41}) and (\ref{a42}) we can conclude:
\begin{align*}
 &A_{21}+A_4 = \sum_{i_2=0}^{n} \sum_{d=n-i_1+2}^{n+1} \sum_{l_1=1}^{d} \sum_{i_1=i_2}^{i_2+d-l_1}  (-1)^{*_{i_2-1}+\maltese_{i_2}^{i_1-1}} \\
 &\qquad \mu_{d+1-l_1}(a_{i_2},\dots,\mu_{l_1}(a_{i_1},\dots,a_n,m,a_1,\dots,a_{i_1+l_1-n-2}),\dots,a_{i_2+d-n-2})	\otimes \dots \otimes a_{i_2-1} \\
 &- \sum_{i_2=0}^n \sum_{d=n-i_2+2}^{n+1} \sum_{l_1=n-i_1+2}^d (-1)^{*_{i_2-1}} \\ 
 &\mu_{d+1-l_1}(\mu_{l_1}(a_{i_2},\dots,a_{i_2+l_1-1}),\dots,a_n,m,a_1,\dots,a_{i_1+l_1-n-2}),\dots,a_{i_2+d-n-2})\otimes \dots \otimes a_{i_2-1} \\
  &=: B_1 - B_2 \ ,
\end{align*}
where $B_1$ and $B_2$ are self-explanatorily defined as the respective big sums. We further compute that
\begin{align*}
B_1 &=  \sum_{i_2=0}^{n} \sum_{d=n-i_1+2}^{n+1} \sum_{l_1=1}^{i_2+d-n-1} \sum_{i_1=i_2}^{i_2+d-l_1}  (-1)^{*_{i_2-1}+\maltese_{i_2}^{i_1-1}} \\
 &\mu_{d+1-l_1}(a_{i_2},\dots,\mu_{l_1}(a_{i_1},\dots,a_n,m,a_1,\dots,a_{i_1+l_1-n-2}),\dots,a_{i_2+d-n-2})\otimes \dots \otimes a_{i_2-1} \\
 &+ \sum_{i_2=0}^{n} \sum_{d=n-i_1+2}^{n+1} \sum_{l_1=i_2+d-n}^{d} \sum_{i_1=i_2}^{i_2+d-l_1}  (-1)^{*_{i_2-1}+\maltese_{i_2}^{i_1-1}} \\
 &\mu_{d+1-l_1}(a_{i_2},\dots,\mu_{l_1}(a_{i_1},\dots,a_n,m,a_1,\dots,a_{i_1+l_1-n-2}),\dots,a_{i_2+d-n-2})	\otimes \dots \otimes a_{i_2-1} \ .
\end{align*}
Comparing this with (\ref{a23}) yields
\begin{align*}
  &A_{23} + B_1 = \sum_{i_2=0}^n \sum_{d=n-i_2+2}^{n+1} \sum_{l_1=1}^{i_2+d-n-1} (-1)^{*_{i_2-1}} \\ 
  &\left( \sum_{i_1=i_2}^n (-1)^{\maltese_{i_2}^{i_1-1}} \mu_{d+1-l_1}(a_{i_2},\dots,\mu_{l_1}(a_{i_1},\dots,a_{i_1+l_1-1}),\dots,m,\dots,a_{i_2+d-n-2}) \right. \\
      &+ \left. \sum_{i_1=0}^{i_2+d-n-l_1-1} (-1)^{\maltese_{i_2}^m + \maltese_0^{i_1-1}} \mu_{d+1-l_1}(a_{i_2},\dots,m,\dots,\mu_{l_1}(a_{i_1},\dots,a_{i_1+l_1-1}),\dots,a_{i_2+d-n-2}) \right)\\
      &\qquad \qquad \qquad \otimes a_{i_2+d-n-1} \otimes \dots \otimes a_{i_2-1}     \\
      &+\sum_{i_2=0}^n \sum_{d=n-i_2+2}^{n+1} \sum_{l_1=i_2+d-n}^d \sum_{i_1=i_2}^{i_2+d-l_1} (-1)^{*_{i_2-1}+\maltese_{i_2}^{i_1-1}} \\
      &\quad \mu_{d+1-l_1}(a_{i_2},\dots,\mu_{l_1}(a_{i_1},\dots,a_{i_1+l_1-1}),\dots,m,\dots,a_{i_2+d-n-2}) \otimes \dots \otimes a_{i_2-1} \ .
\end{align*}
Fixing $i_2$ and $d$ for the moment and writing down the $d$-th defining equation of the $A_\infty$-algebra $A$ for $a_{i_2},\dots,a_n,m,a_1,\dots,a_{i_2+d-n-2} \in A$, one obtains that:
\begin{align*}
&\sum_{l_1=1}^{i_2+d-n-1} \left( \sum_{i_1=i_2}^n (-1)^{\maltese_{i_2}^{i_1-1}} \mu_{d+1-l_1}(a_{i_2},\dots,\mu_{l_1}(a_{i_1},\dots,a_{i_1+l_1-1}),\dotsb,m,a_1,\dots,a_{i_2+d-n-2}) \right. \\
      &+ \left. \sum_{i_1=0}^{i_2+d-n-l_1-1} (-1)^{\maltese_{i_2}^n + \maltese_0^{i_1-1}} \mu_{d+1-l_1}(a_{i_2},\dots,m,a_1\dots,\mu_{l_1}(a_{i_1},\dots,a_{i_1+l_1-1}),\dots,a_{i_2+d-n-2}) \right) \\
      &+\sum_{l_1=i_2+d-n}^d \sum_{i_1=i_2}^{i_2+d-l_1} (-1)^{\maltese_{i_2}^{i_1-1}} \\  &\qquad \mu_{d+1-l_1}(a_{i_2},\dots,\mu_{l_1}(a_{i_1},\dots,a_{i_1+l_1-1}),\dots,a_n,m,a_1,\dots,a_{i_2+d-n-2}) = 0 \ .
\end{align*}
Thus, it holds that $A_{23}+B_1=0$ and consequently:
\begin{equation*}
 A_{21}+A_{23}+A_4 = - B_2 \quad \Leftrightarrow \quad  A_2 + A_4 = A_{22}-B_2 \ .
\end{equation*}
Since we have already shown that $A_1=0$, it remains to show that
\begin{equation}
\label{a3b2a22}
 A_3 = B_2-A_{22}  \ .
\end{equation}
Indeed, if (\ref{a3b2a22}) holds true, then we have shown that $b \circ b = A_1+A_2+A_3+A_4 = 0$, so we will spend the rest of this proof by checking (\ref{a3b2a22}). 

\bigskip

We give an explicit description of $A_3$:
\begin{align*}
   A_3 &= \sum_{l_1=1}^{n+1} \sum_{i_1=n-l_1+2}^n \sum_{l_2=1}^{n-l_1+2} \sum_{i_2=0}^{n-l_1-l_2+2} \left(b_{i_2,l_2} \circ b_{i_1,l_1}\right)(\mas) \\
     &= \sum_{l_1=1}^{n+1} \sum_{i_1=n-l_1+2}^n \sum_{l_2=1}^{n-l_1+2} \sum_{i_2=0}^{n-l_1-l_2+2} (-1)^{*_{i_1-1}} \\
     &\qquad b_{i_2,l_2}(\mu_{l_1}(a_{i_1},\dots,a_n,m,a_1,\dots,a_{i_1+l_1-n-2})\otimes a_{i_1+l_1-n-1}\otimes \dots \otimes a_{i_1-1} \\
     &= A_{31}+A_{32} \ ,
\end{align*}
where we put
\begin{align*}
 A_{31} &= \sum_{l_1=1}^{n+1} \sum_{i_1=n-l_1+2}^n \sum_{l_2=1}^{n-l_1+2}(-1)^{*_{i_1-1}} \\ 
 &\mu_{l_2}(\mu_{l_1}(a_{i_1},\dots,a_n,m^,\dots,a_{i_1+l_1-n-2}),\dots,a_{i_1+l_1+l_2-n-3})\otimes \dots \otimes a_{i_1-1} \ , \\
 A_{32} &= \sum_{l_1=1}^{n+1} \sum_{i_1=n-l_1+2}^n \sum_{l_2=1}^{n-l_1+2} \sum_{i_2=1}^{n-l_1-l_2+2} (-1)^{*_{i_1-1}(\mas)+\maltese_0^{i_2-1}(b_{i_1,l_1}(\mas))} \\
     &\mu_{l_1}(a_{i_1},\dots,a_n,m,\dots,a_{i_1+l_1-n-2})\otimes \dots \\ &\qquad \qquad \qquad \qquad \otimes \mu_{l_2}(a_{i_1+i_2+l_1-n-2},\dots,a_{i_1+i_2+l_1+l_2-n-3})\otimes \dots \otimes a_{i_1-1} \ .
\end{align*}
Renaming $i_1$ by $i_2$, shifting $l_2$ by $l_1-1$ and renaming it by $d$ leads to:
\begin{align*}
 A_{31} &= \sum_{l_1=1}^{n+1} \sum_{i_2=n-l_1+2}^n \sum_{d=l_1}^{n+1}(-1)^{*_{i_2-1}} \\ 
 &\quad \mu_{d+1-l_1}(\mu_{l_1}(a_{i_2},\dots,a_n,m,a_1,\dots,a_{i_2+l_1-n-2}),\dots,a_{i_1+d-n-2})\otimes \dots \otimes a_{i_2-1} \\
 &= \sum_{i_2=0}^n \sum_{l_1=n-i_1+2}^{n+1} \sum_{d=l_1}^{n+1} \cdots = \sum_{i_2=0}^n \sum_{d=n-i_2+2}^{n+1} \sum_{l_1=n-i_2+2}^d \cdots = B_2 \ .
\end{align*}
Therefore, (\ref{a3b2a22}) is a direct consequence of
\begin{equation*}
 A_{32}\stackrel{!}{=}-A_{22} \ .
\end{equation*}
If we again interchange the names of $i_1$ and $i_2$, $l_1$ and $l_2$, resp., the sum $A_{32}$ will look as follows:
\begin{align*}
 &A_{32} = \sum_{l_2=1}^{n+1} \sum_{i_2=n-l_2+2}^n \sum_{l_1=1}^{n-l_2+2} \sum_{i_1=1}^{n-l_1-l_2+2} (-1)^{*_{i_2-1}(m \otimes \dots \otimes a_n)+\maltese_0^{i_1-1}(b_{i_2,l_2}(m \otimes \dots  \otimes a_n))} \\
     &\mu_{l_2}(a_{i_2},\dots,m,\dots,a_{i_2+l_2-n-2})\otimes \dots \otimes \mu_{l_1}(a_{i_1+i_2+l_2-n-2},\dots,a_{i_1+i_2+l_1+l_2-n-3})\otimes \dots \otimes a_{i_2-1} \ .
\end{align*}
By part 4 of Lemma \ref{SignsHochschildProof}, this coincides with
\begin{align*}
 &A_{32} = - \sum_{l_2=1}^{n+1} \sum_{i_2=n-l_2+2}^n \sum_{l_1=1}^{n-l_2+2} \sum_{i_1=1}^{n-l_1-l_2+2} (-1)^{\maltese_0^{i_1+i_2+l_2-n-3}(m\otimes\dots\otimes a_n) + *_{i_2-l_1}\left(b_{i_1+i_2+l_2-n-2,l_1}(m\otimes\dots\otimes a_n)\right)} \\
     &\mu_{l_2}(a_{i_2},\dots,m,\dots,a_{i_2+l_2-n-2})\otimes \dots \otimes \mu_{l_1}(a_{i_1+i_2+l_2-n-2},\dots,a_{i_1+i_2+l_1+l_2-n-3})\otimes \dots \otimes a_{i_2-1} \\
     &= -\sum_{l_2=1}^{n+1} \sum_{i_2=n-l_2+2}^n \sum_{l_1=1}^{n-l_2+2} \sum_{i_1=i_2+l_2-n-1}^{i_2-l_1} (-1)^{\maltese_0^{i_1-1}(\mas) + *_{i_2-l_1}(b_{i_1,l_1}(\mas))} \\
     &\qquad \mu_{l_2}(a_{i_2},\dots,a_n,m,a_1,\dots,a_{i_2+l_2-n-2})\otimes \dots \otimes \mu_{l_1}(a_{i_1},\dots,a_{i_1+l_1-1})\otimes \dots \otimes a_{i_2-1} \\    
     &= -\sum_{l_2=1}^{n+1} \sum_{i_2=n-l_2+2}^n \sum_{l_1=1}^{n-l_2+2} \sum_{i_1=i_2+l_2-n-1}^{i_2-l_1}  \\
     &\qquad b_{i_2-l_1+1,l_2}\left((-1)^{\maltese_0^{i_1-1}}(m\otimes a_1 \otimes \dots \otimes \mu_{l_1}(a_{i_1},\dots,a_{i_1+l_1-1})\otimes \dots \otimes a_n \right) \\
     &= - \sum_{l_2=1}^{n+1} \sum_{l_1=1}^{n-l_2+2} \sum_{i_2=n-l_2+2}^n  \sum_{i_1=i_2+l_2-n-1}^{i_2-l_1} \left(b_{i_2-l_1+1,l_2} \circ b_{i_1,l_1} \right)(\mas) \\
     &= - \sum_{l_1=1}^{n+1} \sum_{i_2=1}^{n-l_1+1} \sum_{l_2=n-i_2-l_1+3}^{n-l_1+2}  \sum_{i_1=i_2+l_1+l_2-n-2}^{i_2-1} \left(b_{i_2,l_2} \circ b_{i_1,l_1} \right)(\mas) \\
     &= - \sum_{l_1=1}^{n+1} \sum_{i_1=0}^{n-l_1+1} \sum_{i_2=i_1+1}^{n-l_1+1}  \sum_{l_2=n-i_2-l_1+3}^{n+i_1-i_2-l_1+2}   \left(b_{i_2,l_2} \circ b_{i_1,l_1} \right)(\mas) = - A_{22} \ , 
\end{align*}
which eventually concludes our proof.
\end{proof}

The following lemma completes this section by delivering the missing sign computations from the proof of Theorem \ref{HochschildDifferential}.

\begin{lemma}
 \label{SignsHochschildProof}
 Let $m,a_1,\dots,a_n \in A$ and let '$\equiv$' denote congruence modulo two.
 \begin{enumerate} 
  \item If  $l_1 \in \{1,2,\dots,n+1\}$, $i_1 \in \{0,1,\dots,n-l_1+1\}$ and $0 \leq i_2 \leq i_1$, then
 \begin{align*}
  &\maltese_0^{i_1-1}(\mas)+*_{i_2-1}(b_{i_1,l_1}(\mas)) \\ &\equiv  *_{i_2-1}(\mas) + \maltese_{i_2}^{i_1-1}(\mas) \ .
 \end{align*}
 \item If  $l_1 \in \{1,2,\dots,n+1\}$, $i_1 \in \{0,1,\dots,n-l_1+1\}$ and $i_1 < i_2 \leq n-l_1+1$, then
 \begin{align*}
  &*_{i_2-1}(b_{i_1,l_1}(\mas)) \\
  &=  *_{i_2+l_1-2}(\mas) + \maltese_{i_2+l_1-1}^n(\mas) \ .
 \end{align*}
 \item If  $l_1 \in \{1,2,\dots,n+1\}$, $i_1 \in \{n-l_1+2,\dots,n\}$, $l_2 \in \{1,2\dots,n-l_1+2\}$ and $i_2 \in \{ n-l_1-l_2+3,\dots, n-l_1+1\}$, then
 \begin{align*}
  &*_{i_1-1}(\mas) + *_{i_2-1}(b_{i_1,l_1}(\mas)) \\
  &\equiv *_{i_1+i_2+l_1-n-3}(\mas) + \maltese_{i_1+i_2+l_1-n-2}^{i_1-1}(\mas) \ .
 \end{align*}
 \item If $l_2 \in \{1,2,\dots,n+1\}$, $i_2 \in \{n-l_2+2,\dots,n\}$, $l_1 \in \{1,2,\dots,n-l_2+2\}$ and $i_1 \in \{1,2,\dots,n-l_1-l_2+2\}$, then
 \begin{align*}
  &*_{i_2-1}(\mas)+ \maltese_0^{i_1-1}(b_{i_2,l_2}(\mas)) +1\\
  &\equiv \maltese_0^{i_1+i_2+l_2-n-3}(\mas) + *_{i_2-l_1}(b_{i_1+i_2+l_2-n-2,l_1}(\mas))  \ .
 \end{align*}
 \end{enumerate}
\end{lemma}

\begin{proof} Recall that for every $l \in \NN$ and $a_1,\dots,a_l \in A$ we have:
\begin{equation*}
 \| \mu_l(a_1,\dots,a_l)\| \equiv \maltese_1^l(a_1\otimes \dots \otimes a_l) +1 \ .
\end{equation*}
We put $\maltese_i^j := \maltese_i^j(\mas)$, $*_i := *_i(\mas)$, whenever defined.
 \begin{enumerate}
  \item If $i_2 \leq i_1 \leq n-l_1+1$, then
\begin{align*}
 *_{i_2-1}(b_{i_1,l_1}(\mas))  &\equiv \maltese_0^{i_2-1} \left(\maltese_{i_2}^{i_1-1} + \|\mu_{l_1}(a_{i_1},\dots,a_{i_1+l_1-1})\| + \maltese_{i_1+l_1}^n \right) \\
 &\equiv \maltese_0^{i_2-1}(\maltese_{i_2}^n + 1) \equiv *_{i_2-1}+\maltese_0^{i_2-1} \ .
\end{align*}
The claim immediately follows.
\item If $i_1<i_2 \leq n-l_1+1$, then
\begin{align*}
 &*_{i_2-1}(b_{i_1,l_1}(\mas)) \\ &\equiv \left(\maltese_0^{i_1-1}+ \|\mu_{l_1}(a_{i_1},\dots,a_{i_1+l_1-1}) \| + \maltese_{i_1+l_1}^{i_2+l_1-2}\right) \maltese_{i_2-l_1-1}^n \\      
&\equiv \left( \maltese_0^{i_2+l_1-2}+1\right)\maltese_{i_2+l_1-1}^n \equiv *_{i_2+l_1-2} + \maltese_{i_2+l_1-1}^n \ .
\end{align*}

\item If $n-l_1+2 \leq i_1 \leq n$, then we get
\begin{align*}
 &*_{i_2-1}(b_{i_1,l_1}(\mas)) \\ &\equiv \left(\mu_M\left( \mu_{l_1}(a_{i_1},\dots,a_n,m,a_1,\dots,a_{i_1+l_1-n-2})\right) + \maltese_{i_1+l_1-n-1}^{i_1+i_2+l_1-n-3} \right) \maltese_{i_1+i_2+l_1-n-2}^{i_1-1} \\
&\equiv \left(\maltese_{i_1}^n + \maltese_0^{i_1+i_2+l_1-n-3}+1 \right) \maltese_{i_1+i_2+l_1-n-2}^{i_1-1} \\
&\equiv \maltese_{i_1+i_2+l_1-n-2}^	{i_1-1} \cdot \maltese_{i_1}^n + \maltese_0^{i_1+i_2+l_1-n-3} \cdot \maltese_{i_1+i_2+l_1-n-2}^{i_1-1} + \maltese_{i_1+i_2+l_1-n-2}^{i_1-1} \ .
\end{align*}
Consequently:
\begin{align*}
 &*_{i_1-1}(\mas)+*_{i_2-1}(b_{i_1,l_1}(\mas)) \\
&\equiv \maltese_0^{i_1-1} \cdot \maltese_{i_1}^n+ \maltese_{i_1+i_2+l_1-n-2}^{i_1-1} \cdot \maltese_{i_1}^n + \maltese_0^{i_1+i_2+l_1-n-3} \cdot \maltese_{i_1+i_2+l_1-n-2}^{i_1-1} + \maltese_{i_1+i_2+l_1-n-2}^{i_1-1} \\
&\equiv \maltese_0^{i_1+i_2+l_1-n-3} \cdot \maltese_{i_1}^n+ \maltese_0^{i_1+i_2+l_1-n-3} \cdot \maltese_{i_1+i_2+l_1-n-2}^{i_1-1} + \maltese_{i_1+i_2+l_1-n-2}^{i_1-1} \\
&\equiv *_{i_1+i_2+l_1-n-3} + \maltese_{i_1+i_2+l_1-n-2}^{i_1-1} \ ,
\end{align*}
which we had to show.
\item Note that in this case
\begin{align*}
 \maltese_0^{i_1-1}(b_{i_2,l_2}(\mas))  &\equiv \mu_M\left(\mu_{l_2}(a_{i_2},\dots,m,\dots,a_{i_2+l_2-n-2}) \right) + \maltese_{i_2+l_2-n-1}^{i_1+i_2+l_2-n-3} \\
 &\equiv \maltese_{i_2}^n + \maltese_0^{i_1+i_2+l_1-n-3}+1 
\end{align*}
and therefore
\begin{align*}
 &*_{i_2-1}(\mas) + \maltese_0^{i_1-1}(b_{i_2,l_2}(\mas)) \\
&\equiv 1+ \maltese_0^{i_1+i_2+l_1-n-3} + \left(\maltese_0^{i_2-1}+1\right)\maltese_{i_2}^n \\
&\equiv 1+ \maltese_0^{i_1+i_2+l_1-n-3}+  \\ &\left(\maltese_0^{i_1+i_2+l_2-n-3}+\|\mu_{l_1}(a_{i_1+i_2+l_2-n-2}, \dots, a_{i_1+i_2+l_1+l_2-n-3})\|+ \maltese_{i_1+i_2+l_1+l_2-n-2}^{i_2-1}\right)\maltese_{i_2}^n \\
&\equiv 1+ \maltese_0^{i_1+i_2+l_1-n-3} + *_{i_2-l_1}(b_{i_1+i_2+l_2-n-2,l_1}(\mas)) \ ,
\end{align*}
which completes the proof.
 \end{enumerate}
\end{proof}

As we announced at the end of Section \ref{SectionAinftyalgebras}, we show that a morphism of $A_\infty$-bimodules $M \to N$ induces a chain map $CH_*(A;M) \to CH_*(A;N)$. This is the appropriate functoriality result for our considerations. 

\bigskip
\label{Defbfrs} 
Let $f=(f_{r,s})_{r,s \in \NN_0}: M \to N$ be a morphism of $A_\infty$-bimodules of degree $d$ and define for every $r,s \in \NN_0$ an $R$-module homomorphism
\begin{equation*}
 \bfrs: CH_*(A;M) \to CH_*(A;N)
\end{equation*}
as the $\ZZ$-linear extension of 
\begin{equation*}
 m \otimes a_1 \otimes a_n \mapsto \begin{cases}
                                    f_{r,s}(a_{n-r+1},\dots,a_n,m,a_1,\dots,a_s)\otimes a_{s+1}\otimes \dots \otimes a_{n-r} & n \leq r+s \ , \\
					0									& n>r+s \ .
                                   \end{cases}
\end{equation*}
Define another $R$-module homomorphism by: \index{morphism of $A_\infty$-bimodules!induced Hochschild chain map}
\begin{align*}
 f_*:CH_*(A;M) &\to CH_*(A;N) \ , \\
 \mas &\mapsto \sum_{r,s \in \NN_0} (-1)^{\dag_r} \bfrs(\mas) \ ,
\end{align*}
where \nomenclature{$\dag_r$}{exponents of the signs defining induced Hochschild chain maps}
\begin{equation*}
\dag_r := \dag_r(\mas) :=*_{n-r} + d \cdot \deg(m \otimes a_1 \otimes \dots \otimes a_n) \ .
\end{equation*}

\begin{theorem} \index{morphism of $A_\infty$-bimodules!induced Hochschild chain map}
\label{inducedHochschildchainmap}
 If $f: M \to N$ is a morphism of $A_\infty$-bimodules over $A$ of degree $d$, then the induced map $f_*: CH_*(A;M) \to CH_*(A;N)$ will be a chain map of degree $-d$.
\end{theorem}

\begin{proof}
First we check the degree statement. Since $\deg \frs = d-r-s$ for all $r,s \in \NN_0$, we obtain:
\begin{align*}
 &\deg \bfrs(\mas) = \deg \left(f_{r,s}(a_{n-r+1},\dots,a_n,m,a_1,\dots,a_s)\otimes a_{s+1}\otimes \dots \otimes a_{n-r} \right) \\
  &= -\mu(\frs(a_{n-r+1},\dots,a_n,m,a_1,\dots,a_s)) - \sum_{q=s+1}^{n-r} \|a_q\| \\
 &= -d + r + s - \sum_{q=n-r+1}^n \mu(a_q)- \mu(m) - \sum_{q=1}^s \mu(a_q) - \sum_{q=s+1}^{n-r} \|a_q\| 
\end{align*}
 \begin{align*}
 &= -d - \sum_{q=n-r+1}^n \|a_q\|- \mu(m) - \sum_{q=1}^s \|a_q\| - \sum_{q=s+1}^{n-r} \|a_q\| \\
  &= -d + \deg(\mas) \ .
\end{align*}
So every $\bfrs$, thus $f_*$ as well, is a graded $R$-module homomorphism of degree $-d$. 

 We need to show that $b \circ f_* = f_* \circ b$ and start by considering the maps $b_{i,l} \circ \bfrs$ in greater detail for different choices of $i$, $l$, $r$ and $s$. 

Assume that $l$, $r$ and $s$ are given with $r+s \leq n$, $1 \leq l \leq n+1$. For $i=0$ we have:
\begin{align}
 &(b_{0,l} \circ \bfrs)(m \otimes a_1 \otimes \dots \otimes a_n) \notag  \\
 &= b_{0,l}(\frs(a_{n-r+1},\dots,a_n,m,a_1,\dots,a_s)\otimes a_{s+1}\otimes \dots \otimes a_{n-r}) \notag \\
 &= \mu^N_{0,l-1}(\frs(a_{n-r+1},\dots,a_n,m,a_1,\dots,a_s), a_{s+1},\dots,a_{s+l-1})\otimes a_{s+l} \otimes \dots \otimes a_{n-r}) \ . \label{inull}
\end{align}
For $1 \leq i \leq n-r-s-l+1$ we get:
\begin{align*}
 &(b_{i,l} \circ \bfrs)(m \otimes a_1 \otimes \dots \otimes a_n) \\
 &= b_{i,l}(\frs(a_{n-r+1},\dots,a_n,m,a_1,\dots,a_s)\otimes a_{s+1} \otimes \dots \otimes a_{n-r}) \\
 &= (-1)^{\mu(\frs(a_{n-r+1},\dots,a_n,m,a_1,\dots,a_s)) + \maltese_{s+1}^{s+i-1}} \\ &\quad \frs(a_{n-r+1},\dots,a_n,m,a_1,\dots,a_s)\otimes \dots \otimes a_{s+i-1}\otimes \mu_l(a_{s+i},\dots,a_{s+i+l-1})\otimes \dots \otimes a_{n-r} \\
 &= (-1)^{\maltese_{n-r+1}^n + \maltese_0^{s+i-1}+d } \\ &\quad \frs(a_{n-r+1},\dots,a_n,m,a_1,\dots,a_s)\otimes \dots \otimes a_{s+i-1}\otimes \mu_l(a_{s+i},\dots,a_{s+i+l-1})\otimes \dots \otimes a_{n-r} \\
 &= (-1)^{\maltese_{n-r+1}^n + d} \bfrs\left((-1)^{\maltese_0^{s+i-1}} m \otimes a_1 \otimes \dots \otimes \mu_l(a_{s+i},\dots,a_{s+i+l-1})\otimes \dots \otimes a_n \right) \ ,
\end{align*}
which finally leads to
\begin{equation}
 \label{idazwischen}
(b_{i,l} \circ \frs)(m\otimes a_1 \otimes \dots \otimes a_n) = (-1)^{\maltese_{n-r+1}^n + d} (\bfrs \circ b_{s+i,l})(\mas)
\end{equation}
for every $1 \leq i \leq n-r-s-l+1$. 


For the third and last case, i.e. $n-r-s-l+2 \leq i \leq n-r-s$, we obtain
\begin{align*}
 &(b_{i,l} \circ \bfrs)(m \otimes a_1 \otimes \dots \otimes a_n) \\
 &= b_{i,l}( \frs(a_{n-r+1},\dots,a_n,m,a_1,\dots,a_s)\otimes a_{s+1}\otimes \dots a_{n-r}) \\
&=(-1)^{\left(\mu(\frs(a_{n-r+1},\dots,a_n,m,a_1,\dots,a_s)) + \maltese_{s+1}^{s+i-1}\right)\maltese_{s+i}^{n-r}} \\
&\qquad \mu^N_{n-r-s-i+1,l-n+r+s+i-2}(a_{s+i},\dots,\frs(a_{n-r+1},\dots,m,\dots,a_s),\dots,a_{l-n+r+2s+i-2})\\
&\phantom{boooooooooooooooooooooooooooooooooooooooooooooooooooooooooooooooo}\otimes \dots \otimes a_{s+i-1} \\
 &= (-1)^{\left( \maltese_{n-r+1}^n +\maltese_0^{s+i-1}+d\right) \maltese_{s+i}^{n-r}} \\
 &\qquad \mu^N_{n-r-s-i+1,l-n+r+s+i-2}(a_{s+i},\dots,\frs(a_{n-r+1},\dots,m,\dots,a_s),\dots,a_{l-n+r+2s+i-2}) \\ 
 &\phantom{boooooooooooooooooooooooooooooooooooooooooooooooooooooooooooooooo}\otimes \dots \otimes a_{s+i-1} \ .
\end{align*}
To simplify the sign, note that modulo 2 it holds that
\begin{align*}
 \left( \maltese_{n-r+1}^n + \maltese_0^{s+i-1} +d \right)  \maltese_{s+i}^{n-r} &= \maltese_{n-r+1}^n  \cdot \maltese_{s+i}^{n-r} +  \maltese_0^{s+i-1} \cdot \maltese_{s+i}^{n-r}  + d \cdot \maltese_{s+i}^{n-r} \\
&= \maltese_{n-r+1}^n \cdot \maltese_{s+i}^{n-r}  + *_{s+i-1} + \maltese_0^{s+i-1} \cdot \maltese_{n-r+1}^n +d \cdot \maltese_{s+i}^{n-r} \\
&= \maltese_0^{n-r} \cdot \maltese_{n-r+1}^n + *_{s+i-1} + d \cdot \maltese_{s+i}^{n-r} \\
&= *_{n-r} + *_{s+i-1} + d \cdot \maltese_{s+i}^{n-r} \ .
\end{align*}
Hence if $n-r-s-l+2 \leq i \leq n-r-s$, we will have
\begin{align}
&( b_{i,l} \circ \bfrs)(m \otimes a_1 \otimes \dots \otimes a_n) = (-1)^{*_{n-r}+*_{s+i-1}+d \cdot \maltese_{s+i}^{n-r}} \label{igross} \\ 
&\qquad \mu^N_{n-r-s-i+1,l-n+r+s+i-2}(a_{s+i},\dots,\frs(a_{n-r+1},\dots,m,\dots,a_s),\dots,a_{l-n+r+2s+i-2}) \notag \\
&\phantom{booooooooooooooooooooooooooooooooooooooooooooooooooooooooooooooo}\otimes  \dots \otimes a_{s+i-1} \ . \notag
\end{align}
If we put $b_l := \sum_{i \in \NN_0} b_{i,l}$ for every $l \in \NN$, then equations (\ref{inull}), (\ref{idazwischen}) and (\ref{igross}) will together yield
\begin{align}
 &(b_l \circ \bfrs)(m \otimes a_1 \otimes \dots \otimes a_n) \notag \\
 &= \sum_{i=1}^{n-r-s-l+1} (-1)^{\maltese_{n-r+1}^n+d} ( \bfrs \circ b_{s+i,l})(m \otimes a_1 \otimes \dots \otimes a_n) \tag{$S_1$} \label{S1} \\
 &\quad + \mu^N_{0,l-1}(\frs(a_{n-r+1},\dots,a_n,m,a_1,\dots,a_s),a_{s+1},\dots,a_{s+l-1})\otimes a_{s+l}\otimes \dots \otimes a_{n-r} \tag{$S_2$} \label{S2} \\
 &\quad + \sum_{i=n-r-s-l+2}^{n-r-s} (-1)^{*_{n-r}+*_{s+i-1}+ d \cdot \maltese_{s+i}^{n-r}} \tag{$S_3$} \label{S3} \\ 
 &\qquad \quad \mu^N_{n-r-s-i+1,l-n+r+s+i-2}(a_{s+i},\dots,\frs(a_{n-r+1},\dots,m,\dots,a_s),\dots,a_{l-n+r+2s+i-2}) \notag \\
 &\phantom{boooooooooooooooooooooooooooooooooooooooooooooooooooooooooooooo}\otimes \dots \otimes a_{s+i-1} \ . \notag
\end{align}
We denote the three parts of the right-hand side by $S_1$, $S_2$ and $S_3$, resp., according to the labeling of the three lines. Performing additional index shifts to $S_1$ and $S_3$ results in
\begin{align*}
 S_1 = &\sum_{i=s+1}^{n-r-l+1} (-1)^{\maltese_{n-r+1}^n + d} (\bfrs \circ b_{i,l})(m \otimes a_1 \otimes \dots a_n) \ , \\
 S_3 = &\sum_{i=1}^{l-1} (-1)^{*_{n-r}+*_{n-r+i-l}+ d \cdot \maltese_{n-r+i-l+1}^{n-r}} \\
 &\mu^N_{l-i,i-1}(a_{n-r+i-l+1},\dots,\frs(a_{n-r+1},\dots,m,\dots,a_s),\dots,a_{s+i-1})\otimes \dots \otimes a_{n-r+i-l} \ .
\end{align*}
One can easily see that adding $S_2$ to $S_3$ corresponds to extending the sum in $S_3$ to the case $i = l$. In conclusion:
\begin{align*}
 &\left(b_l \circ \bfrs\right)(m \otimes a_1 \otimes \dots \otimes a_n) \\
&= \sum_{i=s+1}^{n-r-l+1} (-1)^{\maltese_{n-r+1}^n + d} \left( \bfrs \circ b_{i,l} \right)(m \otimes a_1 \otimes \dots \otimes a_n) \\
&+ \sum_{i=1}^l (-1)^{*_{n-r}+*_{n-r+i-l}+d \cdot \maltese_{n-r+i-l+1}^{n-r}} \\
&\quad \mu^N_{l-i,i-1}(a_{n-r+i-l+1},\dots,\frs(a_{n-r+1},\dots,m,\dots,a_s),\dots,a_{s+i-1})\otimes \dots \otimes a_{n-r+i-l} \\
&= (-1)^{\maltese_{n-r+1}^n + d} \Bigl( \bfrs \circ \Bigl( \sum_{i=s+1}^{n-r-l+1} b_{i,l} \Bigr) \Bigr)(m \otimes a_1 \otimes \dots \otimes a_n) \\
&+ (-1)^{*_{n-r}} \Bigl(\sum_{i=1}^l (-1)^{*_{n-r+i-l}+d \cdot \maltese_{n-r+i-l+1}^{n-r}} \Bigr. \\
&\qquad  \Bigl. \mu^N_{l-i,i-1}(a_{n-r+i-l+1},\dots,\frs(a_{n-r+1},\dots,m,\dots,a_s),\dots,a_{s+i-1})\otimes \dots \otimes a_{n-r+i-l} \Bigr) \ .
\end{align*}
Summing up over all $l$, we therefore obtain:
\begin{align}
 &\left(b \circ \bfrs\right)(m \otimes a_1 \otimes \dots \otimes a_n) = \sum_{l=1}^{n+1-r-s} \left(b_l \circ \bfrs \right)(m \otimes a_1 \otimes \dots \otimes a_n) \notag \\
 &= (-1)^{\maltese_{n-r+1}^n + d} \Bigl( \bfrs \circ \Bigl( \sum_{l=1}^{n+1-r-s} \sum_{i=s+1}^{n-r-l+1} b_{i,l} \Bigr) \Bigr)(m \otimes a_1 \otimes \dots \otimes a_n) \notag  \\
&+ (-1)^{*_{n-r}} \Bigl(\sum_{l=1}^{n+1-r-s} \sum_{i=1}^l (-1)^{*_{n-r+i-l}+d \cdot \maltese_{n-r+i-l+1}^{n-r}} \Bigr. \tag{$S_4$} \label{S4} \\
&\Bigl. \mu^N_{l-i,i-1}(a_{n-r+i-l+1},\dots,\frs(a_{n-r+1},\dots,m,\dots,a_s),\dots,a_{s+i-1})\otimes \dots \otimes a_{n-r+i-l} \Bigr) \ . \notag
\end{align}
We define $S_4$ as the respective part of the right-hand side which is labeled by ($S_4$). Modifying the indices of the sums in $S_4$, we further compute:
\begin{align*}
 &S_4 = (-1)^{*_{n-r}} \Bigl( \sum_{i=1}^{n+1-r-s} \sum_{l=i}^{n+1-r-s}  (-1)^{*_{n-r+i-l}+d \cdot \maltese_{n-r+i-l+1}^{n-r}} \Bigr. \\
&\qquad \left. \mu^N_{l-i,i-1}(a_{n-r+i-l+1},\dots,\frs(a_{n-r+1},\dots,m,\dots,a_s),\dots,a_{s+i-1})\otimes \dots \otimes a_{n-r+i-l} \right) \\
&= (-1)^{*_{n-r}} \Bigl(\sum_{i=0}^{n-r-s} \sum_{l=0}^{n-r-s-i}  (-1)^{*_{n-r-l}+d \cdot \maltese_{n-r-l+1}^{n-r}} \Bigr. \\
&\qquad \left. \mu^N_{l,i}(a_{n-r-l+1},\dots,a_{n-r},\frs(a_{n-r+1},\dots,m,\dots,a_s),a_{s+1},\dots,a_{s+i})\otimes \dots \otimes a_{n-r-l} \right) \ .
\end{align*}
Finally, we want to compute the map $b \circ f_*$. Writing $a := m \otimes a_1 \otimes \dots \otimes a_n$ leads to
\begin{align}
 &\left(b \circ f_* \right)(m \otimes a_1 \otimes \dots \otimes a_n)= \sum_{r,s \in \NN_0} (-1)^{*_{n-r}+d \cdot \deg(a)} \left( b \circ \bfrs \right)(a)  \notag \\
 &= \sum_{r,s \in \NN_0} (-1)^{*_{n-r}+\maltese_{n-r+1}^n + d \cdot (\deg(a)-1)} \left( \bfrs \circ \left( \sum_{l=1}^{n+1-r-s} \sum_{i=s+1}^{n+1-r-l} b_{i,l} \right) \right)(a) \tag{$S_5$} \label{S5} \\
 &\quad + \sum_{r,s \in \NN_0}  \sum_{i=0}^{n-r-s} \sum_{l=0}^{n-r-s-i}  (-1)^{d \cdot \deg(a)+*_{n-r-l}+d \cdot \maltese_{n-r-l+1}^{n-r}} \tag{$S_6$} \label{S6} \\
&\qquad \mu^N_{l,i}(a_{n-r-l+1},\dots,a_{n-r},\frs(a_{n-r+1},\dots,m,\dots,a_s),a_{s+1},\dots,a_{s+i})\otimes \dots \otimes a_{n-r-l} \ . \notag
\end{align}
We need to take a closer look at the sums $S_5$ and $S_6$ which we define in the same way as the sums $S_1$ to $S_4$. 
\bigskip 

Before we start, we introduce some additional notation. Define $\tfrs$ by 
\begin{equation*}
 \tfrs(x) := (-1)^{*_{n-r}+d\cdot \deg(x) } \bfrs
\end{equation*}
for every $r,s \in \NN_0$ and $x \in CH_*(A;M)$. In this notation, it holds that $f_* = \sum_{r,s \in \NN_0} \tfrs$ and the sum $S_5$ can be reformulated as follows:
\begin{align}
 S_5 &= \sum_{r,s \in \NN_0} (-1)^{*_{n-r}+\maltese_{n-r+1}^n + d \cdot (\deg(a)-1)} \left( \bfrs \circ \left( \sum_{l=1}^{n+1-r-s} \sum_{i=s+1}^{n+1-r-l} b_{i,l} \right) \right)(a) \notag \\
&= \sum_{r,s \in \NN_0} \sum_{l=1}^{n+1-r-s} \sum_{i=s+1}^{n+1-r-l} (-1)^{*_{n-r}(b_{i,l}(a))+d\cdot \deg(b_{i,l}(a))} \left(\bfrs \circ b_{i,l}\right)(a) \notag \\
&= \sum_{r,s \in \NN_0} \sum_{l=1}^{n+1-r-s} \sum_{i=s+1}^{n+1-r-l} \left(\tfrs \circ b_{i,l}\right)(a) \notag \\
&=\sum_{l=1}^{n+1} \sum_{r=0}^{n-l+1} \sum_{s=0}^{n-l+1-r}\sum_{i=s+1}^{n-l+1-r} \left(\tfrs \circ b_{i,l}\right)(a) =\sum_{l=1}^{n+1} \sum_{r=0}^{n-l+1} \sum_{i=0}^{n-l+1-r} \sum_{s=0}^{i-1} \left(\tfrs \circ b_{i,l}\right)(a) \notag \\
\Leftrightarrow S_5 &=\sum_{l=1}^{n+1} \sum_{i=0}^{n-l+1} \sum_{r=0}^{n-l+1-i} \sum_{s=0}^{i-1} \left(\tfrs \circ b_{i,l}\right)(a) \ . \label{S5schoener}
\end{align}
For the moment, we conclude our treatment of $S_5$ and turn our attention to $S_6$ which can be described more conveniently with respect to our purposes. Note that
\begin{align*}
 S_6 &= \sum_{r=0}^n \sum_{s=0}^{n-r} \sum_{i=0}^{n-r-s} \sum_{l=0}^{n-r-s-i} (-1)^{d \cdot \deg(a)+ *_{n-r-l}+d \cdot \maltese_{n-r-l+1}^{n-r}} \\
&\qquad \mu^N_{l,i}(a_{n-r-l+1},\dots,\frs(a_{n-r+1},\dots,m,\dots,a_s),\dots,a_{s+i})\otimes \dots \otimes a_{n-r-l} \\
&= \sum_{r=0}^n \sum_{s=0}^{n-r} \sum_{l=0}^{n-r-s} \sum_{i=0}^{n-r-s-l} \  \cdots \ = \sum_{r=0}^n \sum_{l=0}^{n-r} \sum_{s=0}^{n-r-l} \sum_{i=0}^{n-r-l-s} \  \cdots \\
&= \sum_{p=0}^n \sum_{r+l=p} \sum_{s=0}^{n-p} \sum_{i=0}^{n-p-l} (-1)^{d \cdot \deg(a)+ *_{n-p}+d \cdot \maltese_{n-p+1}^{n-r}} \\
&\qquad \mu^N_{l,i}(a_{n-p+1},\dots,\frs(a_{n-r+1},\dots,m,\dots,a_s),\dots,a_{s+i})\otimes \dots \otimes a_{n-p} \\
&= \sum_{p=0}^n \sum_{q=0}^{n-p} (-1)^{d \cdot \deg(a) + *_{n-p}} \\ 
&\left( \sum_{r+l=p}\sum_{s+i=q} (-1)^{d\cdot \maltese^{n-r}_{n-p+1} }\mu^N_{l,i}(a_{n-p+1},\dots,\frs(a_{n-r+1},\dots,m,\dots,a_s),\dots,a_{q}) \right) \\ 
&\qquad \qquad \qquad \otimes a_{q+1} \otimes \dots \otimes a_{n-p} \ .
\end{align*}
Since $f$ is a morphism of $A_\infty$-bimodules, it holds for fixed numbers $p$ and $q$ that
\begin{align*}
 &\sum_{r+l=p}\sum_{s+i=q} (-1)^{d\cdot \maltese^{n-r}_{n-p+1} }\mu^N_{l,i}(a_{n-p+1},\dots,a_{n-r},\frs(a_{n-r+1},\dots,m,\dots,a_s),a_{s+1},\dots,a_{q}) \\
 &= \sum_{r+l=p+1} \sum_{j=n-p+1}^{n-l+1} (-1)^{\maltese_{n-p+1}^{j-1}+d} f_{r,q}(a_{n-p+1},\dots,\mu_l(a_j,\dots,a_{j+l-1}),\dots,a_n,m,a_1,\dots,a_q) \\
 &+ \sum_{r+l=p} \sum_{s+i=q} (-1)^{\maltese^{n-l}_{n-p+1}+d} f_{r,s}(a_{n-p+1},\dots,a_{n-l},\mu^M_{l,i}(a_{n-l+1},\dots,m,\dots,a_i),a_{i+1},\dots,a_q) \\
 &+ \sum_{s+l=q+1} \sum_{j=1}^{q-l+1} (-1)^{\maltese_{n-p+1}^n + \maltese_0^{j-1}+d} f_{p,s}(a_{n-p+1},\dots,m,\dots,\mu_l(a_j,\dots,a_{j+l-1}),\dots,a_q) \ .
\end{align*}
Inserting this into our description of $S_6$, we obtain
\begin{align}
 S_6 &= \sum_{p=0}^n \sum_{q=0}^{n-p} \sum_{r+l=p+1} \sum_{j=n-p+1}^{n-l+1} (-1)^{*_{n-p} + \maltese_{n-p+1}^{j-1} +d(\deg(a)-1)} \tag{$S_7$} \label{S7} \\
      &\qquad \qquad f_{r,q}(a_{n-p+1},\dots,\mu_l(a_j,\dots,a_{j+l-1},\dots,m,\dots,a_q) \otimes a_{q+1}\otimes \dots \otimes a_{n-p} \notag \\
      &+ \sum_{p=0}^n \sum_{q=0}^{n-p}\sum_{r+l=p} \sum_{s+i=q}  (-1)^{*_{n-p}+ \maltese^{n-l}_{n-p+1}+d(\deg(a)-1)} \tag{$S_8$} \label{S8} \\
      &\qquad \qquad f_{r,s}(a_{n-p+1},\dots,\mu^M_{l,i}(a_{n-l+1},\dots,m,\dots,a_i),\dots,a_q)\otimes a_{q+1}\otimes \dots \otimes a_{n-p} \notag \\
      &+ \sum_{p=0}^n \sum_{q=0}^{n-p}  \sum_{s+l=q+1} \sum_{j=1}^{q-l+1} (-1)^{*_{n-p} + \maltese_{n-p+1}^n + \maltese_0^{j-1}+d(\deg(a)-1)} \tag{$S_9$} \label{S9} \\
      &\qquad \qquad f_{p,s}(a_{n-p+1},\dots,m,\dots,\mu_l(a_j,\dots,a_{j+l-1}),\dots,a_q)\otimes a_{q+1}\otimes \dots \otimes a_{n-p}\ . \notag
\end{align}
We define the sums $S_7$, $S_8$ and $S_9$ similar to the preceding ones $S_1$ to $S_4$. We want to simplify the three sums to get a simpler form for $S_6$ and eventually for the map $b \circ \bfrs $ that we were actually considering. 

We start by simplifying $S_8$. We relabel its indices for notational purposes.
\begin{align}
 S_8 = &\sum_{p=0}^n \sum_{q=0}^{n-p}\sum_{r+i=p} \sum_{s+l=q}  (-1)^{*_{n-p}+ \maltese^{n-l}_{n-p+1}+d(\deg(a)-1)}   \label{S9relabel} \\
      &f_{r,s}(a_{n-p+1},\dots,a_{n-i},\mu^M_{i,l}(a_{n-i+1},\dots,m,\dots,a_l),a_{l+1},\dots,a_q)\otimes a_{q+1}\otimes \dots \otimes a_{n-p} \notag
\end{align}
We compute in the $\tfrs$-notation that
\begin{align*}
 &\frs(a_{n-p+1},\dots,a_{n-i},\mu^M_{i,l}(a_{n-i+1},\dots,a_n,m,a_1,\dots,a_l),a_{l+1},\dots,a_{l+s})\otimes \dots \otimes a_{n-p} \\
  &= (-1)^{\maltese_{n-p+1}^{n-i}  \cdot \left( \|\mu^M_{i,l}(a_{n-i+1},\dots,a_n,m,a_1,\dots,a_l) \| + \maltese_{l+1}^{n-p}\right) + d\cdot (\deg(a)-1)} \\
 &\qquad \qquad \tfrs\left(\mu^M_{i,l}(a_{n-i+1},\dots,a_n,m,a_1,\dots,a_l)\otimes a_{l+1}\otimes \dots \otimes a_{n-i}\right) \\
 &= (-1)^{\maltese_{n-p+1}^{n-i} \left(\maltese_{n-i+1}^n + \maltese_0^{n-p}+1\right)+ d\cdot (\deg(a)-1)}  \\ &\qquad \qquad \tfrs\left(\mu^M_{i,l}(a_{n-i+1},\dots,a_n,m,a_1,\dots,a_l)\otimes a_{l+1}\otimes \dots \otimes a_{n-i}\right) \ .
\end{align*}
This last sign can be expressed in a simpler way. Modulo two, we derive that
\begin{align*}
 &\maltese_{n-p+1}^{n-i} \left(\maltese_{n-i+1}^n + \maltese_0^{n-p}+1\right) \equiv \maltese_{n-p+1}^{n-i} + \maltese_{0}^{n-p} \cdot \maltese_{n-p+1}^{n-i} + \maltese_{n-p+1}^{n-i}\cdot \maltese_{n-i+1}^n \\
 &\equiv \maltese_{n-p+1}^{n-i} + *_{n-p} + \maltese_0^{n-p}\cdot \maltese_{n-i+1}^n  + *_{n-i} + \maltese_{n-i+i}^n \cdot \maltese_0^{n-p} \equiv \maltese_{n-p+1}^{n-i} + *_{n-p} + *_{n-i} \ .
\end{align*}
Consequently,
\begin{align*}
 &\frs(a_{n-p+1},\dots,a_{n-i},\mu^M_{i,l}(a_{n-i+1},\dots,a_n,m,a_1,\dots,a_l),a_{l+1},\dots,a_{l+s})\otimes \dots \otimes a_{n-p} \\
 &= (-1)^{\maltese_{n-p+1}^{n-i} + *_{n-p} + *_{n-i}+ d\cdot (\deg(a)-1)} \tfrs\left(\mu^M_{i,l}(a_{n-i+1},\dots,m,\dots,a_l)\otimes a_{l+1}\otimes \dots \otimes a_{n-i}\right) \\
&= (-1)^{\maltese_{n-p+1}^{n-i}+*_{n-p}+ d\cdot (\deg(a)-1)} \left( \tfrs \circ b_{n-i+1,i+l+1} \right)(a) \ .
\end{align*}
Inserting this into (\ref{S9relabel}) leads to:
\begin{align*}
 S_8 &= \sum_{p=0}^n \sum_{q=0}^{n-p}\sum_{r+i=p} \sum_{s+l=q} \left(\tfrs \circ b_{n-i+1,i+l+1}\right)(a) = \sum_{r+i\leq n} \ \sum_{s+l \leq n-r-i} \left(\tfrs \circ b_{n-i+1,i+l+1}\right)(a) \\
&= \sum_{i=0}^n \sum_{l=0}^{n-i} \sum_{r=0}^{n-i-l} \sum_{s=0}^{n-r-i-l}\left(\tfrs \circ b_{n-i+1,i+l+1}\right)(a) \ .
\end{align*}

We perform two more index shifts to finally arrive at the formula we want to use:
\begin{align}
 S_8 &= \sum_{i=0}^n \sum_{l=i+1}^{n+1} \sum_{r=0}^{n-l+1} \sum_{s=0}^{n-l+1-r}\left(\tfrs \circ b_{n-i+1,l}\right)(a) \notag \\
  &= \sum_{i=0}^n \sum_{l=n-i+2}^{n+1} \sum_{r=0}^{n-l+1} \sum_{s=0}^{n-l+1-r}\left(\tfrs \circ b_{i,l}\right)(a)= \sum_{l=1}^{n+1} \sum_{i=n-l+2}^{n} \Bigl(\Bigl(\sum_{r+s\leq n-l+1}\tfrs \Bigr) \circ b_{i,l}\Bigr)(a) \notag \\
\Leftrightarrow S_8 &= \sum_{l=1}^{n+1} \sum_{i=n-l+2}^{n} \left(f_* \circ b_{i,l}\right)(a) \ , \label{S9neu} 
\end{align}
which is precisely the part of $f_* \circ b$ which involves the overlapping part of the Hoch\-schild differential. 

We next consider the other two sums $S_7$ and $S_9$, starting with $S_7$. In the new notation it follows that for every $j \in \{n-p+1,n-p+2,\dots,n-l+1\}$ we have
\begin{align*}
 &f_{r,q}(a_{n-p+1},\dots,a_{j-1},\mu_l(a_j,\dots,a_{j+l-1}),\dots,a_n,m,a_1,\dots,a_q)\otimes a_{q+1} \otimes \dots \otimes a_{n-p} \\
 &= (-1)^{\maltese_0^{n-p} \cdot \left(\maltese_{n-p+1}^{j-1} + \| \mu_l(a_j,\dots,a_{j+l-1}) \| + \maltese_{j+l}^n  \right)+d \cdot (\deg(a)-1)} \\ 
&\qquad \tilde{f}_{r,q}(m\otimes a_1 \otimes \dots \otimes a_{j-1} \otimes \mu_l(a_j,\dots,a_{j+l-1})\otimes a_{j+l}\otimes \dots \otimes a_n) \\
 &= (-1)^{\maltese_0^{n-p} \cdot \left(\maltese_{n-p+1}^{n} +1\right)+d \cdot (\deg(a)-1)} \\ 
&\qquad \tilde{f}_{r,q}(m\otimes a_1 \otimes \dots \otimes a_{j-1} \otimes \mu_l(a_j,\dots,a_{j+l-1})\otimes a_{j+l}\otimes \dots \otimes a_n) \\
 &= (-1)^{*_{n-p} + \maltese_0^{n-p}+ d \cdot (\deg(a)-1)} \\
&\qquad \tilde{f}_{r,q}(m \otimes a_1 \otimes \dots \otimes a_{j-1} \otimes \mu_l(a_j,\dots,a_{j+l-1})\otimes a_{j+l}\otimes \dots \otimes a_n) \ .
\end{align*}
So $S_7$ becomes
\begin{align*}
 S_7 &= \sum_{p=0}^n \sum_{q=0}^{n-p} \sum_{r+l=p+1} \sum_{j=n-p+1}^{n-l+1}  (-1)^{\maltese_0^{n-p}+ \maltese_{n-p+1}^{j-1}} \\
&\qquad \tilde{f}_{r,q}(m \otimes a_1 \otimes \dots \otimes a_{j-1} \otimes \mu_l(a_j,\dots,a_{j+l-1})\otimes a_{j+l}\otimes \dots \otimes a_n) \\
&=  \sum_{p=0}^n \sum_{q=0}^{n-p} \sum_{r+l=p+1} \sum_{j=n-p+1}^{n-l+1} \tilde{f}_{r,q}\left((-1)^{\maltese_0^{j-1}} m \otimes a_1 \otimes \dots \otimes \mu_l(a_j,\dots,a_{j+l-1})\otimes \dots \otimes a_n \right) \\
&= \sum_{p=0}^n \sum_{q=0}^{n-p} \sum_{r+l=p+1} \sum_{j=n-p+1}^{n-l+1} \left(\tilde{f}_{r,q} \circ b_{j,l}\right)(a) \ .
\end{align*}
We relabel the indices for notational purposes and reorder them in a convenient way.
\begin{align}
 S_7 &= \sum_{p=0}^n \sum_{s=0}^{n-p} \sum_{r+l=p+1} \sum_{i=n-p+1}^{n-l+1} \left( \tfrs \circ b_{i,l}\right)(a)= \sum_{p=0}^n \sum_{r+l=p+1} \sum_{i=n-p+1}^{n-l+1}\sum_{s=0}^{n-p+1} \left( \tfrs \circ b_{i,l}\right)(a)  \notag \\
= &\sum_{r+l \leq n+1} \sum_{i=n-r-l+2}^{n-l+1} \sum_{s=0}^{n-r-l+1}\left( \tfrs \circ b_{i,l}\right)(a) = \sum_{l=1}^{n+1} \sum_{r=0}^{n+1-l}\sum_{i=n-r-l+2}^{n-l+1} \sum_{s=0}^{n-r-l+1}\left( \tfrs \circ b_{i,l}\right)(a) \notag \\
\Leftrightarrow S_7 &= \sum_{l=1}^{n+1} \sum_{i=1}^{n-l+1} \sum_{r=n-i-l+2}^{n+1-l}\sum_{s=0}^{n-r-l+1}\left( \tfrs \circ b_{i,l}\right)(a) \ .\label{S8neu} 
\end{align}
This concludes our treatment of $S_7$ for the moment and we continue by considering $S_9$. We start again by relabeling the indices as follows:
\begin{align}
 S_9 &= \sum_{r=0}^n \sum_{q=0}^{n-r} \sum_{s+l = q+1} \sum_{i=1}^{q-l+1} (-1)^{*_{n-r}+\maltese_{n-r+1}^n + \maltese_0^{i-1} + d \cdot (\deg(a)-1)} \label{S7relabel} \\
      &\qquad \frs(a_{n-r+1},\dots,a_n,m,a_1,\dots,\mu_l(a_i,\dots,a_{i+l-1}),\dots,a_q)\otimes a_{q+1}\otimes \dots \otimes a_{n-r} \ . \notag
\end{align}
One checks that
\begin{align*}
 &\frs(a_{n-r+1},\dots,a_n,m,a_1,\dots,a_{i-1},\mu_l(a_i,\dots,a_{i+l-1}),a_{i+l},\dots,a_q)\otimes a_{q+1}\otimes \dots \otimes a_{n-r} \\
 &= (-1)^{\maltese_{n-r+1}^n \cdot \left(\maltese_0^{i-1} + \|\mu_l(a_i,\dots,a_{i+l-1}) \| + \maltese_{q}^{n-r} \right)+d(\deg(a)-1)} \\
 &\qquad \tfrs(m \otimes a_1 \otimes \dots \otimes \mu_l(a_i,\dots,a_{i+l-1})\otimes a_{i+l}\otimes \dots \otimes a_n ) \\
 &= (-1)^{\maltese_{n-r+1}^n \cdot \left(\maltese_0^{n-r} + 1 \right)+d(\deg(a)-1)} \\
 &\qquad \tfrs(m \otimes a_1 \otimes \dots \otimes \mu_l(a_i,\dots,a_{i+l-1})\otimes \dots \otimes a_n ) \\
 &= (-1)^{*_{n-r}+\maltese_{n-r+1}^n +d(\deg(a)-1)}\tfrs(m \otimes a_1 \otimes \dots \otimes \mu_l(a_i,\dots,a_{i+l-1})\otimes \dots \otimes a_n ) \ .
\end{align*}
Inserting this into (\ref{S7relabel}) gives us
\begin{align}
 S_9 &= \sum_{r=0}^n \sum_{q=0}^{n-r} \sum_{s+l = q+1} \sum_{i=1}^{q-l+1} (-1)^{\maltese_0^{i-1}} \tfrs(m \otimes a_1 \otimes \dots \otimes \mu_l(a_i,\dots,a_{i+l-1})\otimes \dots \otimes a_n ) \notag \\
   &= \sum_{r=0}^n \sum_{q=0}^{n-r} \sum_{s+l = q+1} \sum_{i=1}^{q-l+1} \tfrs \left((-1)^{\maltese_0^{i-1}} m \otimes a_1 \otimes \dots \otimes \mu_l(a_i,\dots,a_{i+l-1})\otimes \dots \otimes a_n \right) \notag \\
   &= \sum_{r=0}^n \sum_{s+l \leq n+1-r} \sum_{i=0}^s \left(\tfrs \circ b_{i,l}\right)(a)= \sum_{r=0}^n \sum_{l=1}^{n+1-r} \sum_{s=0}^{n-l+1-r}\sum_{i=0}^s \left(\tfrs \circ b_{i,l}\right)(a) \notag \\
\Leftrightarrow S_9  &= \sum_{l=1}^{n+1} \sum_{i=0}^{n-l+1} \sum_{r=0}^{n-l+1-i}\sum_{s=i}^{n-l+1-r} \left(\tfrs \circ b_{i,l}\right)(a) \ .\label{S9schoener}
\end{align}
Comparing equations (\ref{S5schoener}) and (\ref{S9schoener}), one sees that
\begin{equation}
\label{summeS5S9}
 S_5+S_9 = \sum_{l=1}^{n+1} \sum_{i=0}^{n-l+1} \sum_{r=0}^{n-l+1-i}\sum_{s=0}^{n-l+1-r} \left(\tfrs \circ b_{i,l}\right)(a) \ .
\end{equation}
Equations (\ref{S8neu}) and (\ref{summeS5S9}) in turn yield:
\begin{align*}
 S_5+S_7+S_9 &= \sum_{l=1}^{n+1} \sum_{i=0}^{n-l+1} \sum_{r=0}^{n-l+1} \sum_{s=0}^{n-l+1-r} \left(\tfrs \circ b_{i,l}\right)(a) \\
&= \sum_{l=1}^{n+1} \sum_{i=0}^{n-l+1} \left(\left(\sum_{r+s\leq n-l+1}\tfrs\right) \circ b_{i,l}\right)(a) = \sum_{l=1}^{n+1} \sum_{i=0}^{n-l+1} \left(f_* \circ b_{i,l}\right)(a) \ .
\end{align*}
Finally, bringing this last result together with (\ref{S9neu}), we obtain:
\begin{align*}
 \left(b \circ f_*\right)(a)&= S_5+S_6 =S_5+S_7+S_8+S_9 \\
 &= \sum_{l=1}^{n+1} \sum_{i=0}^{n-l+1} \left(f_* \circ b_{i,l}\right)(a) + \sum_{l=1}^{n+1} \sum_{i=n-l+2}^{n} \left(f_* \circ b_{i,l}\right)(a) = \left(f_* \circ b\right)(a) \ ,
\end{align*}
which was to be shown.
\end{proof}

\begin{remark} \index{morphism of $A_\infty$-bimodules!composition}
 If $f: M \to N$ and $g: N \to P$ are morphisms of $A_\infty$-bimodules over $A$, one can define the composition of $f$ and $g$ as a morphism of $A_\infty$-bimodules $g\circ f: M \to P$ by requiring that
\begin{equation*}
 (g \circ f)_* =  g_* \circ f_*: CH_*(A;M) \to CH_*(A;P) \ .
\end{equation*}
We refrain from writing down the defining equations for $f \circ g$ explicitly.
\end{remark}

We conclude this section by discussing a simple criterion which guarantees that a morphism of $A_\infty$-bimodules induces an isomorphism in Hochschild homology if the ring is given by the integers. 

\bigskip 

Let $f: M\to N$ be a morphism of $A_\infty$-bimodules over $A$. The defining equation for morphisms of $A_\infty$-bimodules for $r=0$ and $s=0$ is equivalent to
\begin{equation*}
 f_{0,0} \circ \mu^M_{0,0} = \mu^N_{0,0} \circ f_{0,0} \ .
\end{equation*}

We have mentioned right after Definition \ref{DefAinftybimodule} that the map $\mu^M_1 = \mu^M_{0,0}$ is a differential of degree $+1$ for every $M$, so if we put $f_1:= f_{0,0}$, the map $f_1:M\to N$ will commute with the differentials $\mu^M_1$ and $\mu^N_1$:
\begin{equation*}
 f_1 \circ \mu^M_1 = \mu^N_1 \circ f_1 \ .
\end{equation*}

The following statement will be very helpful for our computations since it greatly simplifies many proofs concerning maps between Hochschild homology groups by reducing the argument to a much simpler case.

\begin{theorem}[\cite{GetzlerJones}, Lemma 5.3]
\label{firstpageisenough}
 Let $A$ be an $A_\infty$-algebra over $\ZZ$ and let $M$ and $N$ be $A_\infty$-bimodules over $A$, which are both torsion-free abelian groups. Let $f: M \to N$ be a morphism of $A_\infty$-bimodules over $A$, $f_1:= f_{0,0}$. If the induced map
\begin{equation*}
 (f_1)_* : H^*(M,\mu^M_1) \to H^*(N,\mu^N_1)
\end{equation*}
is an isomorphism of graded $R$-modules, then the map induced by $f_*$ on homology level,
\begin{equation*}
 [f_*]: HH_*(A;M) \to HH_*(A;N) \ ,
\end{equation*}
 will be an isomorphism of graded abelian groups.
\end{theorem}

The proof of this theorem makes use of spectral sequence techniques and is deferred to Section \ref{SectionLengthFiltration}. There, we will construct filtrations on Hochschild chain complexes and investigate the associated spectral sequences. \index{spectral sequence of the length filtration}

Another proof of Theorem \ref{firstpageisenough} using minimal models of $A_\infty$-algebras is given in \cite[Section 2.2]{TZS}.

\begin{remark}
 The theorem can be generalized to arbitrary coefficient rings if one replaces the condition that $A$ and $M$ are torsion-free by a more sophisticated flatness condition. More precisely, the proof makes use of the K\"unneth formulas for the pairs of complexes $(M,A^{\otimes p})$ and $(N,A^{\otimes p})$, where $p \in \NN$, which only hold under certain algebraic assumptions. 
\end{remark}

\section{\texorpdfstring{Hochschild cohomology of $A_\infty$-algebras}{Hochschild cohomology of A-infinity-algebras}}
\label{SectionHochschildCohomology}

We are going to define the dual notion of Hochschild homology, the Hochschild cohomology of $A_\infty$-algebras with coefficients in $A_\infty$-bimodules over them. Hochschild cohomology for $A_\infty$-algebras has been considered by a number of authors in different contexts, for example in \cite{Markl} and \cite{TradlerBV}. It is a straighforward generalization of the Hochschild cohomology of a differential graded algebra, as it is for example presented in \cite[Chapter 9]{Weibel} and \cite[Section 1]{AbbaspourSurvey}. \index{Hochschild!cohomology}

We give an elementary construction of the Hochschild cochain complex, as in \cite[Section 2e]{SeidelAinftyneu} and \cite[Section 7.1]{KontsevichSoibelmanNotes}. Since its relation and especially its duality to the Hochschild chain complex is not that obvious, we afterwards discuss certain duality aspects of this complex in Remark \ref{CohomBimodules} and Theorem \ref{HochschildDuality}.
\bigskip

\textit{Throughout this section let $(A,(\mu_n)_{n\in \NN})$ be an $A_\infty$-algebra over $R$ and $(M,(\mu^M_{r,s})_{r,s\in\NN_0})$ be an $A_\infty$-bimodule over $A$.} 

\bigskip 

For any $n \in \NN_0$ let $\Hom^*(A^{\otimes n},M)$ denote the $R$-module of graded $R$-module homomorphisms from $A^{\otimes n}$ to $M$. Here, $A^{\otimes n}$ is equipped with the usual tensor product grading, i.e. for every $j\in \ZZ$:
\begin{equation*}
 \left(A^{\otimes n} \right)_j = \bigoplus_{j = j_1+\dots+j_n} A_{j_1}\otimes A_{j_2}\otimes \dots \otimes A_{j_n} \ .
\end{equation*}
For $k \in \ZZ$ we denote by $\Hom^k(A^{\otimes n},M) \subset \Hom^*(A^{\otimes n},M)$ the set of all homomorphisms of degree $k$. This number $k$ is called the \emph{internal degree} and we will write $\mu_M(f)=k$ if $f\in \Hom^k(A^{\otimes n},M)$. \\

Consider the $R$-module \index{Hochschild!cochain complex} \nomenclature{$CH^*(A;M)$}{Hochschild cochain complex of $A$ with coefficients in $M$}
\begin{equation*}
 CH^*(A;M) := \prod_{n \in \NN_0} \Hom^*(A^{\otimes n},M) \ .
\end{equation*}
An integer grading on $CH^*(A;M)$ is given by putting 
\begin{equation*}
 CH^j(A;M) := \prod_{n \in \NN_0} \Hom^{j-n}(A^{\otimes n},M) 
\end{equation*}
for all $j \in \ZZ$. We will write 
\begin{equation*}
\deg f := j
\end{equation*}
iff $f \in CH^*(A;M)$ lies in $CH^j(A;M)$. A differential on $CH^*(A;M)$ is given by the $R$-module homomorphism \index{Hochschild!codifferential} \nomenclature{$\beta$}{Hochschild codifferential}
\begin{equation*}
 \beta: CH^*(A;M) \to CH^*(A;M)
\end{equation*}
which is defined in the following way: If $f \in \Hom(A^{\otimes n},M)$ for some $n \in \NN_0$, then 
\begin{equation*}
\beta(f) \in \prod_{l \in \NN_0}  \Hom(A^{\otimes (n+l)},M) \ ,
\end{equation*}
where for fixed $l \in \NN_0$, the $\Hom(A^{\otimes (n+l)},M)$-component of $\beta(f)$, denoted by $(\beta(f))_l$, is for every $a_1,\dots,a_{n+l} \in A$ given by
\begin{align*}
 &(\beta(f))_l(a_1,\dots,a_{n+l}) := \sum_{i=1}^n (-1)^{\maltese_1^{i-1}} f(a_1,\dots,a_{i-1},\mu_{l+1}(a_i,\dots,a_{i+l}),a_{i+l+1},\dots,a_{n+l}) \\
    &\quad + \sum_{i=1}^{l+1} (-1)^{\deg f \cdot \left(\maltese_{1}^{i-1}+1 \right)+1} \mu^M_{i-1,l-i+1}(a_1,\dots,a_{i-1},f(a_i,\dots,a_{n+i-1}),a_{n+i},\dots,a_{n+l}) \ .
\end{align*}
We define $\beta$ on general elements of $CH^*(A;M)$ by extending the previous definition $R$-linearly. 

\begin{remark}
 One checks that for any $f \in \Hom^*(A^{\otimes n},M)$, $n \in \NN_0$, the internal degree of $f$ and its degree as a Hochschild cochain are related by 
 \begin{equation}
 \label{EqDegreeInternalCochains}
 \deg f = \mu_M(f)+n \ .
 \end{equation}
\end{remark}

The proof of the following statement can be given in the spirit of the computations in the previous sections. Therefore we omit it. 
\begin{theorem}
 The map $\beta: CH^*(A;M) \to CH^*(A;M)$ is a differential, i.e. $\beta \circ \beta = 0$, of degree $+1$.
\end{theorem}

This observation leads to the following mildly surprising definition.

\begin{definition} \index{Hochschild!cohomology} \nomenclature{$HH^*(A;M)$}{Hochschild homology of $A$ with coefficients in $M$}
 The complex $(CH^*(A;M),\beta)$ is called the \emph{Hochschild cochain complex of $A$ with coefficients in $M$}. Its cohomology is called the \emph{Hochschild cohomology of $A$ with coefficients in $M$}. We denote it by
 \begin{equation*}
  HH^*(A;M) := H^*(CH^*(A;M),\beta) \ .
 \end{equation*}
\end{definition}

The definition of Hochschild cohomology may look surprising to the reader, since it is not at all obvious if Hochschild homology and cohomology of some $A$ and $M$ are related and in some sense dual to each other. Moreover, the signs appearing in the Hochschild differential are slightly different to those used in the definition of Hochschild homology. We will clarify both aspects in the next theorem which requires some basic algebraic preparations. 

\bigskip

 The following remark and the subsequent theorem might clarify the relationship between Hochschild homology and cohomology.

 \begin{remark}
 \label{CohomBimodules}
  Let $f = \sum_{n \in \NN_0} f_n \in CH^d(A;M)$ for some $d \in \ZZ$ with $f_n \in \Hom(A^{\otimes n},M)$ for every $n \in \NN_0$. For all $r,s \in \NN_0$ we define a map
  \begin{align*}
   f_{r,s}: A^{\otimes r} \otimes A[1] \otimes A^{\otimes s} &\to M \ ,  \\
       a_1\otimes \dots \otimes a_r \otimes a_0 \otimes a_{r+1} \otimes \dots \otimes a_{r+s} &\mapsto f_{r+s+1}(a_1,\dots,a_r,a_0,a_{r+1},\dots,a_{r+s}) \ .
  \end{align*}
 One checks from the defining equations of $A_\infty$-bimodules that $\beta(f)=0$ if and only if the family $(f_{r,s})_{r,s \in \NN_0}$ is a morphism of $A_\infty$-bimodules over $A$ from $A[1]$ to $M$ of degree $d$. In other words, a Hochschild cocyle of a certain degree with coefficients in $M$ can be seen as a morphism of $A_\infty$-bimodules from $A[1]$ to $M$ of the same degree. 
 \end{remark}

For every $n \in \NN_0$ there is a canonical isomorphism of $R$-modules
\begin{equation*}
 \varphi_n : \Hom(M\otimes A^{\otimes n},R) \stackrel{\cong}{\to} \Hom( A^{\otimes n},M^*)
\end{equation*}
given by the property that for every $a_1,\dots,a_n \in A$ it holds that
\begin{equation*}
 \varphi_n(f)(a_1,\dots,a_n) = \left(m \mapsto f(m,a_1,\dots,a_n) \right) \ .
\end{equation*}
Identifying 
\begin{equation*}
 \Hom(CH_*(A;M),R) = \Hom  \Bigl( \bigoplus_{n \in \NN_0} M\otimes A^{\otimes n},R \Bigr) \cong \prod_{n \in \NN_0} \Hom \left(M\otimes A^{\otimes n},R \right)  ,
\end{equation*}
we define an isomorphism of graded $R$-modules by
\begin{align*}
 &\varphi:= \varphi_M: \Hom(CH_*(A;M),R) \to CH^*(A;M^{-*}) \ , \\
 &\left(\varphi(f)(a_1,\dots,a_n)\right)(m) := (-1)^{\mu_M(m)\cdot\maltese_1^n}  \left(\varphi_n(f)(a_1,\dots,a_n)\right)(m) \ ,
\end{align*}
for all $n \in \NN_0$, $f \in \Hom(M\otimes A^{\otimes n},R)$, $a_1,a_2,\dots,a_n \in A$ and $m \in M$, where $M^{-*}$ denotes the dual bimodule $M^*$ with inverted grading and equipped with the dual $A_\infty$-bimodule structure from Theorem \ref{DualBimodule}.

\begin{theorem} \index{Hochschild!cohomology!relation to Hochschild homology}
\label{HochschildDuality}
\label{HochschildCoVarphi}
 $\varphi$ is an isomorphism of cochain complexes of degree zero with respect to the Hochschild codifferential $\beta$ and the pull-back by the Hochschild differential $b^*$ defined on 
 \begin{equation*}
 \Hom(CH_*(A;M),R)
 \end{equation*} 
 equipped with the obvious grading. In other words, the following diagram commutes:
 \begin{equation*}
  \begin{CD}
   \Hom(CH_*(A;M),R) @>{b^*}>> \Hom(CH_*(A;M),R) \\
    @V{\varphi}V{\cong}V   @V{\varphi}V{\cong}V  \\
   CH^*(A;M^{-*}) @>{\beta}>> CH^*(A;M^{-*}) \ .
  \end{CD}
 \end{equation*}
\end{theorem}

\begin{proof}
 We already know that $\varphi$ is an isomorphism of $R$-modules, so we only need to show that $\varphi$ respects the degree and that it commutes with the differentials. 
 
 Let $j \in \ZZ$ and $j_0,j_1,\dots,j_n \in \ZZ$ for some $n \in \NN_0$ with
 \begin{equation}
 \label{explicitdegreej}
  j = n - \sum_{q=0}^n j_q
 \end{equation}
 and let $f \in \Hom \left(M_{j_0}\otimes A_{j_1}\otimes \dots \otimes A_{j_n},R\right) \subset \Hom(CH_j(A;M),R)$. For the degree statement, it suffices to show that $\deg \varphi(f) = j$ for any such $f$. 
 
 By definition, it holds that $\varphi(f)(a_1,\dots,a_n) \neq 0$ only if $a_q \in A_{j_q}$ for every $q \in \{1,2,\dots,n\}$ and that in this case $\varphi(f)(a_1,\dots,a_n) \in M^{j_0} = M^{-(-j_0)}$.
So we know that the degree of $\varphi(f)$ as a homomorphism of $R$-bimodules from $A^{\otimes n}$ to $M^{-*}$ is given by
\begin{equation*}
 -j_0 - \sum_{q=1}^n j_q \stackrel{(\ref{explicitdegreej})}{=} j-n \ .
\end{equation*}
Therefore $\varphi(f) \in \Hom^{j-n}(A^{\otimes n},M) \subset CH^j(A;M)$. Put differently, $\varphi$ is a degree-preserving map. 
 \bigskip

We will explicitly check the compatibility of $\varphi$ with the codifferentials. Let $n \in \NN_0$ and $f_n \in \Hom(M\otimes A^{\otimes n},R)$ for some $n \in \NN_0$. Then by definition, for any fixed $l \in \NN_0$, $a_1,\dots,a_{n+l} \in A$ and $m \in M$, we have
 \begin{align}
  &(\varphi(b^*f_n)(a_1,\dots,a_{n+l}))(m) =  (-1)^{\mu_M(m)\cdot \maltese_1^{n+l}}(f_n \circ b)(m,a_1,\dots,a_{n+l}) \notag \\
  &= \sum_{i=0}^n (-1)^{\mu_M(m)\cdot \maltese_1^{n+l}+\maltese_0^{i-1}}f_n(m,a_1,\dots,a_{i-1},\mu_{l+1}(a_i,\dots,a_{i+l}),a_{i+l+1},\dots,a_{n+l}) \notag \\
  &+ \sum_{i=n+1}^{n+l} (-1)^{\mu_M(m) \cdot \maltese_1^{n+l}+*_{i-1}}f_n(\mu^M_{n+l-i+1,i-n-1}(a_i,\dots,a_{n+l},m,a_1,\dots,a_{i-n-1}),\dots,a_{i-1}) \notag \\
  &= \sum_{i=0}^n (-1)^{\mu_M(m)\cdot \maltese_1^{n+l}+\maltese_0^{i-1}}f_n(m,a_1,\dots,a_{i-1},\mu_{l+1}(a_i,\dots,a_{i+l}),a_{i+l+1},\dots,a_{n+l}) \notag \\
  &+ \sum_{i=1}^{l} (-1)^{\mu_M(m)\cdot \maltese_1^{n+l}+*_{n+i-1}}f_n(\mu^M_{l-i+1,i-1}(a_{n+i},\dots,a_{n+l},m,a_1,\dots,a_{i-1}),a_{i},\dots,a_{n+i-1}) \notag \\
  &= \sum_{i=1}^n (-1)^{\mu_M(m) \cdot (\maltese_1^{n+l}+1)+\mu_M(m)\cdot \maltese_1^{n+l}+\maltese_0^{i-1}} \label{Eqpullbackf1}\\
  &\qquad \qquad (\varphi(f_n)(a_1,\dots,a_{i-1},\mu_{l+1}(a_i,\dots,a_{i+l}),a_{i+l+1},\dots,a_{n+l}))(m) \notag \\
  &+ \sum_{i=1}^{l+1} (-1)^{\mu_M(\mu^M_{l-i+1,i-1}(a_{n+i},\dots,a_{n+l},m,a_1,\dots,a_{i-1}))\cdot \maltese_i^{n+i-1}+\mu_M(m)\cdot \maltese_1^{n+l}+*_{n+i-1}} \label{Eqpullbackf2} \\
  &\qquad \qquad (\varphi(f_n)(a_{i},\dots,a_{n+i-1}))(\mu^M_{l-i+1,i-1}(a_{n+i},\dots,a_{n+l},m,a_1,\dots,a_{i-1})) \notag \ .
 \end{align}
 The congruence class of the exponent of the sign in \eqref{Eqpullbackf1} modulo two computes as 
 $$\mu_M(m) \cdot (\maltese_1^{n+l}+1)+\mu_M(m)\cdot \maltese_1^{n+l}+\maltese_0^{i-1} \equiv \mu_M(m)+\maltese_0^{i-1} = \maltese_1^{i-1} \ . $$
For the sign in \eqref{Eqpullbackf2} we consider the following congruence modulo two: 
\begin{align*}
 &\mu_M(\mu^M_{l-i+1,i-1}(a_{n+i},\dots,a_{n+l},m,a_1,\dots,a_{i-1}))\cdot \maltese_i^{n+i-1}+\mu_M(m)\cdot \maltese_1^{n+l}+*_{n+i-1} \\ 
 &\equiv (\maltese_{n+i}^{n+l} + \mu_M(m) + \maltese_1^{i-1}+1)\cdot \maltese_i^{n+i-1}+\mu_M(m)\cdot \maltese_1^{n+l}+(\mu_M(m)+\maltese_1^{n+i-1})\maltese_{n+i}^{n+l} \\ 
 &\equiv \mu_M(m) \cdot \maltese_1^{i-1} + \maltese_i^{n+i-1} + \maltese_{n+i}^{n+l} \cdot \maltese_i^{n+i-1} + \maltese_1^{i-1} \cdot \maltese_i^{n+i-1} + \maltese_1^{n+i-1} \cdot \maltese_{n+i}^{n+l} \\
 &\equiv \maltese_1^{i-1}(\mu_M(m)+\maltese_i^{n+l}) + \maltese_i^{n+i-1} \ . 
\end{align*}
Hence, we obtain 
\begin{align}
 &(\varphi(b^*f_n)(a_1,\dots,a_{n+l}))(m) \notag \\
 &= \sum_{i=1}^n (-1)^{\maltese_1^{i-1}} (\varphi(f_n)(a_1,\dots,a_{i-1},\mu_{l+1}(a_i,\dots,a_{i+l}),a_{i+l+1},\dots,a_{n+l}))(m) \notag \\
 &+ \sum_{i=1}^{l+1} (-1)^{\maltese_1^{i-1}(\mu_M(m)+\maltese_i^{n+l}) + \maltese_i^{n+i-1}} \notag \\
 &\qquad \qquad (\varphi(f_n)(a_{i},\dots,a_{n+i-1}))(\mu^M_{l-i+1,i-1}(a_{n+i},\dots,a_{n+l},m,a_1,\dots,a_{i-1})) \notag \\
 &= \sum_{i=1}^n (-1)^{\maltese_1^{i-1}} (\varphi(f_n)(a_1,\dots,a_{i-1},\mu_{l+1}(a_i,\dots,a_{i+l}),a_{i+l+1},\dots,a_{n+l}))(m) \notag \\
 &+ \sum_{i=1}^{l+1} (-1)^{\maltese_1^{i-1}(\mu_M(m)+\maltese_i^{n+l}) + \maltese_i^{n+i-1}+ \ddag_{i-1,l+i-1}(a_1,\dots,a_{i-1},\varphi(f_n)(a_i,\dots,a_{n+i-1}),a_{n+i},\dots,a_{n+l},m)} \label{Eqpullbackf3} \\
 &\qquad \qquad (\mu_{i-1,l-i+1}^*(a_1,\dots,a_{i-1},\varphi(f_n)(a_{i},\dots,a_{n+i-1}),a_{n+i},\dots,a_{n+l}))(m) \ . \notag
\end{align}
The sign appearing in \eqref{Eqpullbackf3} can be described independently of $m$. To reformulate the sign accordingly, we note that by definition of $f_n$, it holds for $i \in \{0,1,\dots,n\}$ that 
  \begin{align*}
  \deg \varphi(f_n) &= \mu_M(m) - \sum_{j=1}^{i-1} \mu(a_j) - \mu(\mu_{l+1}(a_i,\dots,a_{i+l})) - \sum_{j=i+l+1}^{n+l}\mu(a_j)- n \\
  &= \mu_M(m)- \sum_{j=1}^{n+l} \mu(a_j)-n+l-1 = \mu_M(m) - \maltese_1^{n+l} - 1  \ ,    
  \end{align*}
 which yields
 \begin{equation}
  \label{Eqmummodtwo}
  \mu_M(m) \equiv \deg \varphi(f_n) + \maltese_1^{n+l} +1 \ . 
 \end{equation}
Moreover, we compute for each $i \in \{0,1,\dots,n\}$ for the degree in $M^{*}$ that 
\begin{equation}
\label{Eqredphimodtwo}
\mu_{M^*}\left(\varphi(f_n)(a_i,\dots,a_{n+i-1})\right) = \deg \varphi(f_n) + \maltese_i^{n+i-1} \ .
\end{equation}
 We may thus simplify the sign from \eqref{Eqpullbackf3} as follows: 
 \begin{align*}
  &\maltese_1^{i-1}(\mu_M(m)+\maltese_i^{n+l}) + \maltese_i^{n+i-1} \\
  &\qquad \qquad \qquad + \ddag_{i-1,l+i-1}(a_1,\dots,a_{i-1},\varphi(f_n)(a_i,\dots,a_{n+i-1}),a_{n+i},\dots,a_{n+l},m) \\
  &\equiv \maltese_1^{i-1}(\mu_M(m)+\maltese_i^{n+l}) + \maltese_i^{n+i-1} + \mu_{M^*}\left(\varphi(f_n)(a_i,\dots,a_{n+i-1})\right) +1  \\
  &\qquad +\maltese_1^{i-1}\left(\maltese_{n+i}^{n+l} + \mu_{M^*}\left(\varphi(f_n)(a_i,\dots,a_{n+i-1})\right) + \mu_M(m)\right) \\
  &\stackrel{\eqref{Eqmummodtwo}}{\equiv} \maltese_1^{i-1}(\deg \varphi(f_n)+\maltese_1^{i-1}+1) + \maltese_i^{n+i-1} + \mu_{M^*}\left(\varphi(f_n)(a_i,\dots,a_{n+i-1})\right) +1\\
  &\qquad \qquad +\maltese_1^{i-1}(\maltese_1^{n+i-1}+\mu_{M^*}\left(\varphi(f_n)(a_i,\dots,a_{n+i-1})\right)+\deg \varphi(f_n)+1)   \\
  &\stackrel{\eqref{Eqredphimodtwo}}{\equiv} \maltese_1^{i-1} \cdot \deg \varphi(f_n) + \deg \varphi(f_n) +1+ \maltese_1^{i-1}(\maltese_1^{i-1} +1) \\
  &\equiv \deg \varphi(f_n)(\maltese_1^{i-1} +1)+1 \ . 
 \end{align*}
Consequently, 
\begin{align*}
 &(\varphi(b^*f_n)(a_1,\dots,a_{n+l}))(m) \notag \\
 &= \sum_{i=1}^n (-1)^{\maltese_1^{i-1}} (\varphi(f_n)(a_1,\dots,a_{i-1},\mu_{l+1}(a_i,\dots,a_{i+l}),a_{i+l+1},\dots,a_{n+l}))(m) \notag \\
 &+ \sum_{i=1}^{l+1} (-1)^{\deg \varphi(f_n)(\maltese_1^{i-1} +1)+1} (\mu_{i-1,l-i+1}^*(a_1,\dots,a_{i-1},\varphi(f_n)(a_{i},\dots,a_{n+i-1}),\dots,a_{n+l}))(m) \\
 &= \left(\left(\beta(\varphi(f_n))(a_1,\dots,a_{n+l}\right)\right)(m) \ .
\end{align*}
\end{proof}

\begin{cor}
\label{fupperstar} \index{morphism of $A_\infty$-bimodules!induced Hochschild cochain map}
 Let $M$ and $N$ be $A_\infty$-bimodules over $A$ and let $f=(f_{r,s})_{r,s\in\NN_0}: M \to N$ be a morphism of $A_\infty$-bimodules. Then $f$ induces a chain map
 \begin{align*}
  &f^*: CH^*(A;N^{-*}) \to CH^*(A;M^{-*}) \ ,\\
  &f^* := \varphi_M \circ \left(f_*\right)^* \circ \varphi_N^{-1} \ ,
 \end{align*}
where $\left(f_*\right)^*: \Hom(CH_*(A;N),R) \to \Hom(CH_*(A;M),R)$ denotes the pull-back of the map $f_*$ from Theorem \ref{inducedHochschildchainmap}.
\end{cor}

\begin{proof}
 This is an immediate consequence of Theorem \ref{inducedHochschildchainmap} and Theorem \ref{HochschildCoVarphi}.
\end{proof}

We end this section by introducing some additional terminology.

\begin{definition}  \nomenclature{$CH_*(A)$, $CH^*(A)$}{Hochschild (co)chain complex of $A$ with coefficients in $A[1]$} \nomenclature{$HH_*(A)$, $HH^*(A)$}{Hochschild (co)homology of $A$ with coefficients in $A[1]$}
We define a chain complex $CH_*(A)$ and a cochain complex $CH^*(A)$ by putting 
$$CH_j(A) := CH_{j-1}(A;A[1]) \quad \text{and} \quad CH^j(A) := CH^{j-1}(A;A[1])$$
for every $j \in \ZZ$ and equipping it with the Hochschild differential and the Hochschild codifferential, resp. We call $CH_*(A)$ the \emph{ Hochschild chain complex of $A$} and $CH^*(A)$ the \emph{Hochschild cochain complex of $A$}. We denote their (co)homology by $HH_*(A)$ and $HH^*(A)$ and call it the \emph{Hochschild homology of $A$} and the \emph{Hochschild cohomology of $A$}, resp. 

We redefine the degree of elements of $CH_*(A)$ and $CH^*(A)$ as follows: for $\sigma \in CH_*(A)$ we will write $\deg \sigma = j$ iff $\sigma \in CH_j(A)$ and for $\varphi \in CH^j(A)$ we write $\deg f=j$ iff $f \in CH^j(A)$. (These degrees differ from the previously introduced degrees by $1$.)
\end{definition}

The previous definition might look awkward to the reader because of the grading conventions. The purpose of defining $CH_j(A)$ as the $(j-1)$-st group of $CH_*(A;A[1])$ is that to a certain level we want to revoke the degree shift in the coefficient bimodule. The reason for considering the shifted $A_\infty$-algebra $A[1]$ was that with our sign conventions $A[1]$ is an $A_\infty$-bimodule over $A$ with the operations from Remark \ref{RemarkBimoduleonA1} while $A$ itself is not. This grading convention will ensure that the Hochschild cup product will be a homomorphism of degree zero as we will show in the next section.

Finally, we give an explicit description of the complexes $CH_*(A)$ and $CH^*(A)$ in terms of our new gradings. For $n \in \NN_0$ and $a_0\otimes a_1 \otimes \dots a_n \in A^{\otimes (n+1)} \subset CH_*(A)$ it holds that 
$$\deg (a_0 \otimes a_1 \otimes \dots \otimes a_n) = n - \sum_{j=0}^n \mu(a_j) \ . $$
We further compute that 
\begin{align*}
&b(a_0\otimes a_1 \otimes \dots \otimes a_n) \\
&= \sum_{l=1}^{n+1} \sum_{i=0}^{n-l+1} (-1)^{\maltese_0^{i-1}} a_0 \otimes a_1 \otimes \dots \otimes a_{i-1} \otimes \mu_l(a_i,\dots,a_{i+l-1})\otimes a_{i+l} \otimes \dots \otimes a_n \\
&+ \sum_{l=1}^{n+l} \sum_{i=n-l+2}^n (-1)^{*_{i-1}} \mu_{l}(a_i,\dots,a_n,a_0,a_1,\dots,a_{i+l-n-2}) \otimes a_{i+l-n-1} \otimes \dots \otimes a_{i-1} \ , 
\end{align*}
where 
$$\maltese_0^{i-1} = \sum_{j=0}^{i-1} \|a_j\| = \sum_{j=0}^{i-1} \mu(a_j)-i \quad \text{and} \quad *_{i-1} = \Bigl(\sum_{j=0}^{i-1} \|a_j\|\Bigr) \cdot \Bigl(\sum_{j=i}^n \|a_j\| \Bigr) \  $$
for all $i \in \{0,1,\dots,n\}$. 
Given $j \in \ZZ$ and $f \in \Hom^*(A^{\otimes n},A)$, it holds by definition of the degree that $f \in CH^j(A)$ if and only if $f \in \Hom^{j-n-1}(A^{\otimes n},A[1])$. In other words, for all $b_1,\dots,b_n \in A$ it holds that 
\begin{align*}
\|f(b_1,\dots,b_n)\| &= \sum_{j=1}^n \mu(b_j) + \deg f - n-1 \\
\Leftrightarrow \quad \mu(f(b_1,\dots,b_n))&=\sum_{j=1}^n \mu(b_j)+\deg f-n \ . 
\end{align*}

Moreover, if $(\beta(f))_l$ denotes the $\Hom(A^{\otimes (n+l)},A)$-component of $\beta(f)$, then
\begin{align*}
 &(\beta(f))_l(a_1,\dots,a_{n+l}) = \sum_{i=1}^n (-1)^{\maltese_1^{i-1}} f(a_1,\dots,a_{i-1},\mu_{l+1}(a_i,\dots,a_{i+l}),a_{i+l+1},\dots,a_{n+1}) \\
    &\quad + \sum_{i=1}^{l+1} (-1)^{(\deg f-1) \cdot \left(\maltese_{1}^{i-1}+1 \right)+1} \mu^M_{i-1,l-i+1}(a_1,\dots,a_{i-1},f(a_i,\dots,a_{n+i-1}),a_{n+i},\dots,a_{n+l}) \ 
\end{align*}
for all $l \in \NN_0$ and $a_1,\dots,a_{n+l} \in A$.

\section{The Hochschild cup product}
\label{SectionHochschildCupProduct}

In general, the Hochschild cohomology of an $A_\infty$-bimodule $M$ over an $A_\infty$-algebra $A$ does not possess a cup product. This is because a product $$HH^*(A;M)\otimes HH^*(A;M) \to HH^*(A;M)$$ would require operations having two elements of $M$ (and possibly elements of $A$) as inputs and one element of $M$ as output. 

Such operations do not exist for general $A_\infty$-bimodules, but in this section we show that if $M=A[1]$, then one can use the $A_\infty$-multiplications $(\mu_n)_{n \in \NN}$ on $A$ to define a cup product $CH^*(A)\otimes CH^*(A) \to CH^*(A)$. 

 In the classical case of Hochschild cohomology of associative algebras, the cup product was discovered by Murray Gerstenhaber in \cite{Gerstenhaber}, see also \cite[Section 1]{AbbaspourSurvey}. The cup product on the Hochschild cohomology of $A_\infty$-algebras is discussed by Tradler in \cite[Section 6]{TradlerInner} and \cite[Section 3.2]{TradlerBV}. Once again our sign conventions differ from those used by Tradler.

\begin{definition} \index{Hochschild!cup product}
 Let $m,n \in \NN_0$, $f \in \Hom^*(A^{\otimes m},A)$ and $g \in \Hom^*(A^{\otimes n},A)$ For all $k \in \NN_0$ and $j_1,j_2 \in \NN$ with $1 \leq j_1 \leq n+k$, $j_1+m \leq j_2 \leq m+k+1$, define \nomenclature{$\cup_{k,j_1,j_2}$}{component map of the Hochschild cup product}
\begin{align*}
 &f \cup_{k,j_1,j_2} g \in \Hom^*(A^{\otimes (m+n+k)},A) \\
 &\left( f \cup_{k,j_1,j_2} g \right)(a_1,\dots,a_{m+n+k}) := (-1)^{\mho_{j_1,j_2}(f,g,a_1,\dots,a_{m+n+k})} \\
 &\qquad \qquad \qquad \mu_{k+2}(a_1,\dots,f(a_{j_1},\dots,a_{j_1+m-1}),\dots,g(a_{j_2},\dots,a_{j_2+n-1}),\dots,a_{m+n+k}) \ , 
\end{align*}
for all $a_1,\dots,\dots,a_{m+n+k}$, where for all $s \in \NN$ with $s \geq j_2$ we put  \nomenclature{$\mho_{j_1,j_2}$}{exponents of the signs defining the Hochschild cup product}
$$\mho_{j_1,j_2}(f,g,a_1,\dots,a_{s}) := (\deg f-1)\cdot \maltese_1^{j_1-1} + (\deg g-1) \cdot (\maltese_1^{j_2-1}+\deg f) \ .$$
Define a homomorphism 
\begin{align*}
&\cup^{m,n}: \Hom^*(A^{\otimes m},A)\otimes \Hom^*(A^{\otimes n},A) \to \prod_{k \geq 0} \Hom^*(A^{\otimes(m+n+k)},A) \ , \\
&f \cup^{m,n} g := \sum_{k \in \NN_0} \sum_{j_1=1}^{n+k} \sum_{j_2=j_1+m}^{m+k+1} f \cup_{k,j_1,j_2} g \ .
\end{align*}
We extend every $\cup^{m,n}$ by zero to a homomorphism $CH^*(A)\otimes CH^*(A) \to CH^*(A)$.  The \emph{Hochschild cup product for $A$} is then defined by
\begin{align*}
 &\cup: CH^*(A)\otimes CH^*(A) \to CH^*(A) \ , \quad f \cup g := \sum_{m,n\in \NN_0} f \cup^{m,n} g \ .
\end{align*}
\end{definition}

\begin{theorem} \index{Hochschild!cup product}
\label{CupProductChainMap}
 The Hochschild cup product of $A$ is a chain map of degree zero.
\end{theorem}

\begin{proof}
 It suffices to show the claim for the cup product of $f \in \Hom^*(A^{\otimes m},A)$ and $g \in \Hom^*(A^{\otimes n},A)$, where $m,n \in \NN_0$. By definition of the degree in $CH^*(A)$, it holds that 
 \begin{align*}
 \mu(f(b_1,\dots,b_m)) &= \sum_{j=1}^m \mu(b_j) + \deg f-m \ , \\ 
 \mu(g(c_1,\dots,c_n)) &= \sum_{j=1}^n \mu(c_j)+\deg g-n \ , 
 \end{align*}
for all $b_1,\dots,b_m,c_1,\dots,c_n \in A$. We thus obtain for all $k,j_1,j_2$ and $a_1,\dots,a_{m+n+k} \in A$ that
\begin{align*}
 &\mu\left((f \cup_{k,j_1,j_2} g)(a_1,\dots,a_{m+n+k})\right) \\ 
 &= \mu\left( \mu_{k+2}(a_1,\dots,f(a_{j_1},\dots,a_{j_1+m-1}),\dots,g(a_{j_2},\dots,a_{j_2+n-1}),\dots,a_{m+n+k})\right) \\
 &= \sum_{q=1}^{j_1-1} \mu(a_q) + \mu\left(f(a_{j_1},\dots,a_{j_1+m-1})\right)  + \sum_{q=j_1+m}^{j_2-1}\mu(a_q) \\
 &\qquad \qquad \qquad \qquad +\mu\left(g(a_{j_2},\dots,a_{j_2+n-1})\right)+\sum_{q=j_2+n}^{m+n+k} \mu(a_q) -k \\
 &= \sum_{q=1}^{m+n-k}\mu(a_q) +\deg f+\deg g -m-n-k \ ,
\end{align*}
which yields 
$$\deg (f \cup_{k,j_1,j_2} g) = \deg f + \deg g  \ . $$
Thus, every $\cup_{k,j_1,j_2}$, and therefore $\cup$ as well, is a graded homomorphism of degree zero. 

Next we discuss the compatibility of the Hochschild cup product with the differentials. It suffices to show that
\begin{equation*}
 \beta(f \cup g) \stackrel{!}{=} \beta(f)\cup g + (-1)^{\deg f} f \cup \beta(g) 
\end{equation*}
 for all $m,n \in \NN_0$, $f \in \Hom^*(A^{\otimes m},A)$ and $g \in \Hom^*(A^{\otimes n},A)$. 
Given $m$, $n$, $f$ and $g$, we thus need to show that
\begin{align*}
 &(\beta(f\cup g))(a_1,\dots,a_{m+n+d}) \\
&\stackrel{!}{=}(\beta(f)\cup g)(a_1,\dots,a_{m+n+d}) + (-1)^{\deg f} (f \cup \beta(g))(a_1,\dots,a_{m+n+d}) 
\end{align*}
for all $d \in \NN_0$ and $a_1,\dots,a_{m+n+d} \in A$. The left-hand side of this equation computes as
\begin{align*}
(\beta(f \cup g))(a_1,\dots,a_{m+n+d}) &=\sum_{k\in \NN_0} \sum_{j_1=1}^{n+k} \sum_{j_2=j_1+m}^{m+k+1} \beta(f \cup_{k,j_1,j_2} g)(a_1,\dots,a_{m+n+d}) \\
&=\sum_{\stackrel{k,l \in \NN_0}{k+l=d}} \sum_{j_1=1}^{n+k} \sum_{j_2=j_1+m}^{m+k+1} (\beta(f \cup_{k,j_1,j_2} g))_l(a_1,\dots,a_{m+n+d}) \ .
\end{align*}
For fixed $k$, $l$, $j_1$ and $j_2$ we further compute that
\begin{align}
 &(\beta(f \cup_{k,j_1,j_2} g))_l(a_1,\dots,a_{m+n+d}) \\
&= \sum_{i=1}^{j_1-1} (-1)^{\maltese_1^{i-1}+\mho_{j_1,j_2}(f,g,a_1,\dots,\mu_{l+1}(a_i,\dots,a_{i+l}),\dots,a_{m+n+d})} \label{cupalg1} \\
&\qquad \qquad \mu_{k+2}(a_1,\dots,\mu_{l+1}(a_i,\dots,a_{i+l}),\dots,f(a_{j_1+l},\dots,a_{j_1+l+m-1}),\dots, \notag \\ &\phantom{booooooooooooooooooooooooooooooooooooooo} g(a_{j_2+l},\dots,a_{j_2+l+n-1}),\dots,a_{m+n+d}) \notag \\
&+ \sum_{i=j_1}^{j_1+m-1} (-1)^{\maltese_1^{i-1}+\mho_{j_1,j_2}(f,g,a_1,\dots,\mu_{l+1}(a_i,\dots,a_{i+l}),\dots,a_{m+n+d})} \label{cupalg2} \\
&\qquad \qquad \mu_{k+2}(a_1,\dots,f(a_{j_1},\dots,\mu_{l+1}(a_i,\dots,a_{i+l}),\dots,a_{j_1+l+m-1}),\dots, \notag \\ &\phantom{booooooooooooooooooooooooooooooooooooooo} g(a_{j_2+l},\dots,a_{j_2+l+n-1}),\dots,a_{m+n+d}) \notag \\
&+ \sum_{i=j_1+m}^{j_2-1} (-1)^{\maltese_1^{i-1}+\mho_{j_1,j_2}(f,g,a_1,\dots,\mu_{l+1}(a_i,\dots,a_{i+l}),\dots,a_{m+n+d})} \label{cupalg3} \\
&\qquad \qquad \mu_{k+2}(a_1,\dots,f(a_{j_1},\dots,a_{j_1+m-1}),\dots,\mu_{l+1}(a_i,\dots,a_{i+l}),\dots,\notag \\ &\phantom{booooooooooooooooooooooooooooooooooooooo} g(a_{j_2+l},\dots,a_{j_2+l+n-1}),\dots,a_{m+n+d}) \notag \\
&+ \sum_{i=j_2}^{j_2+n-1} (-1)^{\maltese_1^{i-1}+\mho_{j_1,j_2}(f,g,a_1,\dots,\mu_{l+1}(a_i,\dots,a_{i+l}),\dots,a_{m+n+d})} \label{cupalg4} \\
&\qquad \qquad \mu_{k+2}(a_1,\dots,f(a_{j_1},\dots,a_{j_1+m-1}),\dots,\notag \\ &\phantom{booooooooooooooooooooo}g(a_{j_2},\dots,\mu_{l+1}(a_i,\dots,a_{i+l}),\dots,a_{j_2+l+n-1}),\dots,a_{m+n+d}) \notag \\
&+ \sum_{i=j_2+n}^{m+n+k} (-1)^{\maltese_1^{i-1}+\mho_{j_1,j_2}(f,g,a_1,\dots,\mu_{l+1}(a_i,\dots,a_{i+l}),\dots,a_{m+n+d})} \label{cupalg5} \\
&\qquad \qquad \mu_{k+2}(a_1,\dots,f(a_{j_1},\dots,a_{j_1+m-1}),\dots,g(a_{j_2},\dots,a_{j_2+n-1}),\dots, \notag \\
&\phantom{boooooooooooooooooooooooooooooooooooooooooooo} \mu_{l+1}(a_i,\dots,a_{i+l}),\dots,a_{m+n+d}) \notag \\
&+\sum_{i=1}^{l+1} (-1)^{(\deg (f \cup_{k,j_1,j_2} g)-1)(\maltese_{1}^{i-1}+1)+1+\mho_{j_1,j_2}(f,g,a_i,\dots,a_{m+n+k+i-1})} \label{cupalg6} \\
&\qquad \qquad \mu_{l+1}(a_1,\dots,\mu_{k+2}(a_i,\dots,f(a_{j_1+i-1},\dots,a_{j_1+i+m-2}),\dots, \notag \\
&\phantom{boooooooooooooooooooooo} g(a_{j_2+i-1},\dots,a_{j_2+i+n-2}),\dots,a_{m+n+k+i-1}),\dots,a_{m+n+d}) \notag \ .
\end{align}
For all $s \in \NN$, $l \in \NN_0$ and $i \in \{1,2,\dots,s\}$ we put 
$$ \beta_{i,l}(h)(a_1,\dots,a_{s+l}) := (-1)^{\maltese_1^{i-1} } h(a_1,\dots,a_{i-1},\mu_{l+1}(a_i,\dots,a_{i+l}),\dots,a_{s+l}) $$ for all $h \in \Hom(A^{\otimes s},A)$ and $a_1,\dots,a_{s+l} \in A$, such that $\sum_{i=1}^{n}\beta_{i,l}(h)$ is a part of $(\beta(h))_l$.

Using this notation, we reformulate the sum in (\ref{cupalg2}) as
\begin{align*}
&\sum_{i=j_1}^{j_1+m-1} (-1)^{\maltese_{1}^{i-1}+\mho_{j_1,j_2}(f,g,a_1,\dots,\mu_{l+1}(a_i,\dots,a_{i+l}),\dots,a_{m+n+d})} \mu_{k+2}(a_1,\dots,\\
&\qquad f(a_{j_1},\dots,\mu_{l+1}(a_i,\dots,a_{i+l}),\dots,a_{j_1+l+m-1}),\dots, g(a_{j_2+l},\dots,a_{j_2+l+n-1}),\dots,a_{m+n+d})  \\
&=\sum_{i=j_1}^{j_1+m-1} (-1)^{\maltese_{1}^{j_1-1}+\mho_{j_1,j_2}(f,g,a_1,\dots,\mu_{l+1}(a_i,\dots,a_{i+l}),\dots,a_{m+n+d})}\\
&\qquad \mu_{k+2}(a_1,\dots,\beta_{i-j_1+1,l}(f)(a_{j_1},\dots,a_{j_1+l+m-1}),\dots,  g(a_{j_2+l},\dots,a_{j_2+l+n-1}),\dots)  \\
&=\sum_{i=j_1}^{j_1+m-1} (-1)^{\maltese_{1}^{j_1-1}+\mho_{j_1,j_2}(f,g,a_1,\dots,\mu_{l+1}(a_i,\dots,a_{i+l}),\dots,a_{m+n+d})+\mho_{j_1,j_2+l}(\beta_{j_1-i+1,l}(f),g,a_1,\dots,a_{m+n+d})} \\
&\qquad \qquad \qquad (\beta_{i-j_1+1,l}(f)\cup_{k,j_1,j_2+l} g)(a_1,\dots,a_{m+n+d}) \\
&=\sum_{i=j_1}^{j_1+m-1} (\beta_{i-j_1+1,l}(f)\cup_{k,j_1,j_2+l} g)(a_1,\dots,a_{m+n+d}) \\
&= \sum_{i=1}^{m} (\beta_{i,l}(f)\cup_{k,j_1,j_2+l} g)(a_1,\dots,a_{m+n+d}) \ , 
\end{align*}
where we have used part 1 of Lemma \ref{SignsCupProduct}. We further compute from the definition of the Hochschild codifferential that
\begin{align*}
 &\sum_{i=1}^{m} (\beta_{i,l}(f)\cup_{k,j_1,j_2+l} g)(a_1,\dots,a_{m+n+d}) = ((\beta(f))_l \cup_{k,j_1,j_2+l} g)(a_1,\dots,a_{m+n+d}) \\
 &\quad- \sum_{i=j_1}^{j_1+l} (-1)^{\mho_{j_1,j_2+l}((\beta(f))_l,g,a_1,\dots,a_{m+n+d})+(\deg f-1) (\maltese_{j_1}^{i-1}+1)+1} \\
 &\qquad \mu_{k+2}(a_1, \dots, \mu_{l+1}(a_{j_1},\dots,f(a_i,\dots,a_{i+m-1}),\dots,a_{j_1+m+l-1}),\dots, \\
 &\qquad \qquad \qquad \qquad \qquad \qquad \qquad \qquad  g(a_{j_2+l},\dots,a_{j_2+l+n-1}),\dots,a_{m+n+d}) \\
 &= ((\beta(f))_l \cup_{k,j_1,j_2+l} g)(a_1,\dots,a_{m+n+d}) \\
 &+ \sum_{i=j_1}^{j_1+l} (-1)^{\deg f + \deg g+\mho_{i,j_2+l}(f,g,a_1,\dots,a_{m+n+d})+\maltese_1^{j_1-1}} \\
 &\qquad \mu_{k+2}(a_1, \dots, \mu_{l+1}(a_{j_1},\dots,f(a_i,\dots),\dots,a_{j_1+m+l-1}),\dots, g(a_{j_2+l},\dots),\dots,a_{m+n+d}) \ ,
\end{align*}
where we have used part 2 of Lemma \ref{SignsCupProduct} for the last equality. In analogy with the above, one uses part 3 of Lemma \ref{SignsCupProduct} to show that the sum (\ref{cupalg4}) is given by 
$$ \sum_{i=1}^{n} (-1)^{\deg f} (f \cup_{k,j_1,j_2} \beta_{i,l}(g))(a_1,\dots,a_{m+n+d}) \ . $$
One further applies part 4 of Lemma \ref{SignsCupProduct} to show along the lines of the above computation that this sum equals
\begin{align*}
 &(-1)^{\deg f}( f \cup_{k,j_1,j_2} (\beta(g))_l)(a_1,\dots,a_{m+n+d}) \\
 &-  \sum_{i=j_2}^{j_2+l} (-1)^{\deg f + \mho_{j_1,j_2}(f,(\beta(g))_l,a_1,\dots,a_{m+n+d})+(\deg g -1) (\maltese_{j_2}^{i-1}+1)+1} \\
 &\quad \mu_{k+2}(a_1, \dots, f(a_{j_1},\dots), \dots, \mu_{l+1}(a_{j_2},\dots,g(a_i,\dots,a_{i+n-1}),\dots,a_{j_2+n+l-1}),\dots,a_{m+n+d}) \\
  &= (-1)^{\deg f}( f \cup_{k,j_1,j_2} \beta(g))(a_1,\dots,a_{m+n+d})  \\
  &+  \sum_{i=j_2}^{j_2+l} (-1)^{\deg f + \deg g+\mho_{j_1,i}(f,g,a_1,\dots,a_{m+n+d})+\deg f-1+\maltese_1^{j_2-1}} \\
 &\quad \mu_{k+2}(a_1, \dots, f(a_{j_1},\dots), \dots, \mu_{l+1}(a_{j_2},\dots,g(a_i,\dots,a_{i+n-1}),\dots,a_{j_2+n+l-1}),\dots,a_{m+n+d}) \ .
\end{align*}
If we insert the last results into our formula for $(\beta(f \cup_{k,j_1,j_2} g))_l(a_1,\dots,a_{m+n+d})$, we will obtain up to certain sign computations, which are covered by parts 5 to 8 of Lemma \ref{SignsCupProduct} that
\begin{align}
 &(\beta(f \cup_{k,j_1,j_2} g))_l(a_1,\dots,a_{m+n+d}) \notag \\
  &= ((\beta(f))_l\cup_{k,j_1,j_2} g)(a_1,\dots,a_{m+n+d}) + (-1)^{\deg f}(f \cup_{k,j_1,j_2} (\beta(g))_l)(a_1,\dots,a_{m+n+d}) \notag \\ 
&+ \sum_{i=1}^{j_1-1} (-1)^{\deg f + \deg g + \mho_{j_1+l,j_2+l}+\maltese_1^{i-1}} \label{algcup2}\\
&\qquad \mu_{k+2}(a_1,\dots,\mu_{l+1}(a_i,\dots,a_{i+l}),\dots,f(a_{j_1+l},\dots),\dots,g(a_{j_2+l},\dots),\dots,a_{m+n+d}) \notag \\
 &+ \sum_{i=j_1}^{j_1+l} (-1)^{\deg f + \deg g + \mho_{i,j_2+l}+\maltese_{1}^{j_1-1}} \label{algcup3}\\
 &\qquad \mu_{k+2}(a_1,\dots,\mu_{l+1}(a_{j_1},\dots,f(a_i,\dots,a_{i+m-1}),\dots),\dots, g(a_{j_2+l},\dots),\dots,a_{m+n+d}) \notag \\
&+ \sum_{i=j_1+m}^{j_2-1} (-1)^{\deg g-1+\mho_{j_1,j_2+l}+\maltese_1^{i-1}} \label{algcup4} \\
&\qquad \mu_{k+2}(a_1,\dots,f(a_{j_1},\dots),\dots,\mu_{l+1}(a_i,\dots,a_{i+l}),\dots,g(a_{j_2+l},\dots),\dots,a_{m+n+d}) \notag \\
 &+ \sum_{i=j_2}^{j_2+l} (-1)^{\deg g-1 +\mho_{j_1,i}+ \maltese_1^{j_2-1}} \label{algcup5} \\ 
 &\qquad \mu_{k+2}(a_1,\dots,f(a_{j_1},\dots),\dots,\mu_{l+1}(a_{j_2},\dots,g(a_i,\dots),\dots,a_{j_2+l+n-1}),\dots,a_{m+n+d}) \notag \\
&+\sum_{i=j_2+n}^{m+n+k} (-1)^{\mho_{j_1,j_2}+\maltese_1^{i-1}} \label{algcup6} \\
&\qquad \mu_{k+2}(a_1,\dots,f(a_{j_1},\dots),\dots,g(a_{j_2},\dots),\dots,\mu_{l+1}(a_i,\dots,a_{i+l}),\dots,a_{m+n+d}) \notag \\
&+\sum_{i=1}^{l+1} (-1)^{\deg f+\deg g+ \mho_{j_1+i-1,j_2+i-1}+\maltese_1^{i-1}}\mu_{l+1}(a_1,\dots, \label{algcup7} \\
&\qquad \mu_{k+2}(a_i,\dots,f(a_{j_1+i-1},\dots),\dots,g(a_{j_2+i-1},\dots),\dots,a_{m+n+k+i-1}),\dots,a_{m+n+d}) \ , \notag
\end{align}
where $\mho_{i_1,i_2} := \mho_{i_1,i_2}(a_1,\dots,a_{m+n+d})$ for all $i_1,i_2 \in \{1,2,\dots,m+n+d\}$. We compute
\begin{equation}
\label{explicitcup}
 (\beta(f\cup g))(a_1,\dots,a_{m+n+d}) = \sum_{k+l=d} \sum_{j_1=1}^{n+k} \sum_{j_2=j_1+m}^{m+k+1} (\beta(f\cup_{k,j_1,j_2} g))_l(a_1,\dots,a_{m+n+d}) \ .
\end{equation}
The right-hand side of \eqref{explicitcup} equals the sum of the left-hand side of the previous big equality over all $k$, $j_1$ and $j_2$. Thus, we will take the sum of the right-hand side of this big equality over all $k$, $j_1$ and $j_2$ to get another formula for $(\beta(f\cup g))(a_1,\dots,a_{m+n+d})$.

The sum of (\ref{algcup2}) over all $k$,$j_1$ and $j_2$ equals
\begin{align*}
 &\sum_{k+l=d} \sum_{j_1=1}^{n+k}  \sum_{j_2=j_1+m}^{m+k+1}\sum_{i=1}^{j_1-1} (-1)^{\deg f + \deg g + \mho_{j_1+l,j_2+l}+\maltese_1^{i-1}} \\
&\qquad \mu_{k+2}(a_1,\dots,\mu_{l+1}(a_i,\dots,a_{i+l}),\dots,f(a_{j_1+l},\dots),\dots,g(a_{j_2+l},\dots),\dots,a_{m+n+d}) \\
&=\sum_{l=0}^d \sum_{j_1=l+1}^{n+d}  \sum_{j_2=j_1+m}^{m+d+1}\sum_{i=1}^{j_1-l-1} (-1)^{\deg f + \deg g + \mho_{j_1,j_2}+\maltese_1^{i-1}} \\
&\qquad \mu_{d-l+2}(a_1,\dots,\mu_{l+1}(a_i,\dots,a_{i+l}),\dots,f(a_{j_1},\dots),\dots,g(a_{j_2},\dots),\dots,a_{m+n+d}) \\
&=\sum_{j_1=1}^{n+d} \sum_{l=0}^{j_1-1} \sum_{j_2=j_1+m}^{m+d+1}\sum_{i=1}^{j_1-l-1} \quad \cdots \quad =\sum_{j_1=1}^{n+d} \sum_{j_2=j_1+m}^{m+d+1} \sum_{i=1}^{j_1-1} \sum_{l=0}^{j_1-i-1}  \quad \cdots \\
&=\sum_{j_1=1}^{n+d} \sum_{j_2=j_1+m}^{m+d+1} (-1)^{\deg f + \deg g + \mho_{j_1,j_2}} \Bigl(\sum_{i=1}^{j_1-1} \sum_{l=0}^{j_1-i-1}(-1)^{\maltese_1^{i-1} }  \Bigr. \\
&\Bigl.\phantom{\sum_{i=1}^{j_1-1}} \qquad \mu_{d-l+2}(a_1,\dots,\mu_{l+1}(a_i,\dots,a_{i+l}),\dots,f(\dots),\dots,g(\dots),\dots,a_{m+n+d}) \Bigr) \ .
\end{align*}
Summing up (\ref{algcup3}) in the same way delivers
\begin{align*}
&\sum_{k+l=d} \sum_{j_1=1}^{n+k}  \sum_{j_2=j_1+l+m}^{m+d+1} \sum_{i=j_1}^{j_1+l} (-1)^{\deg f + \deg g + \mho_{i,j_2}+\maltese_{1}^{j_1-1}} \\
&\qquad \mu_{k+2}(a_1,\dots,\mu_{l+1}(a_{j_1},\dots,f(a_i,\dots),\dots,a_{j_1+l+m-1}),\dots, g(a_{j_2},\dots),\dots,a_{m+n+d}) \ .
\end{align*}
Switching the names of the indices $i$ and $j_1$ and replacing $k$ by $d-l$ results in
\begin{align*}
 &\sum_{l=0}^d \sum_{i=1}^{n+d-l}  \sum_{j_2=i+l+m}^{m+d+1} \sum_{j_1=i}^{i+l} (-1)^{\deg f + \deg g + \mho_{j_1,j_2}+\maltese_{1}^{i-1}} \\
&\qquad \mu_{d-l+2}(a_1,\dots,\mu_{l+1}(a_{i},\dots,f(a_{j_1},\dots),\dots,a_{i+l+m-1}),\dots, g(a_{j_2},\dots),\dots,a_{m+n+d}) \\
&= \sum_{j_1=1}^{n+d} \sum_{j_2=j_1+m}^{m+d+1} (-1)^{\deg f + \deg g + \mho_{j_1,j_2}}\left(\sum_{i=1}^{j_1}\sum_{l=j_1-i}^{j_2-m-i} (-1)^{\maltese_{1}^{i-1}} \right. \\
&\qquad \left. \mu_{d-l+2}(a_1,\dots,\mu_{l+1}(a_{i},\dots,f(a_{j_1},\dots),\dots,a_{i+l+m-1}),\dots, g(a_{j_2},\dots),\dots,a_{m+n+d}) \  \right) \ .
\end{align*}
Next we discuss (\ref{algcup7}). The summation yields:
\begin{align*}
 &\sum_{k+l=d}\sum_{j_1=1}^{n+k}  \sum_{j_2=j_1+m}^{m+k+1}\sum_{i=1}^{l+1} (-1)^{\deg f + \deg g + \mho_{j_1+i-1,j_2+i-1}+\maltese_1^{i-1}} \\
&\qquad \mu_{l+1}(a_1,\dots,\mu_{k+2}(a_i,\dots,f(a_{j_1+i-1},\dots),\dots, \\
&\phantom{boooooooooooooooooooooooooooooooo} g(a_{j_2+i-1},\dots),\dots,a_{m+n+k+i-1}),\dots,a_{m+n+d}) \\
&=\sum_{k+l=d}\sum_{i=1}^{l+1} \sum_{j_1=i}^{n+k+i-1}  \sum_{j_2=j_1+m}^{m+k+i} (-1)^{\deg f + \deg g + \mho_{j_1,j_2}+\maltese_1^{i-1}} \\
&\quad \mu_{l+1}(a_1,\dots,\mu_{k+2}(a_i,\dots,f(a_{j_1},\dots),\dots,g(a_{j_2},\dots),\dots,a_{m+n+k+i-1}),\dots,a_{m+n+d}) \ .
\end{align*}
Switching the names of $k$ and $l$ and replacing $k$ by $d-l$ leads to:
\begin{align*}
 &\sum_{l=0}^d \sum_{i=1}^{d-l+1} \sum_{j_1=i}^{n+l+i-1}  \sum_{j_2=j_1+m}^{m+l+i} (-1)^{\deg f + \deg g + \mho_{j_1,j_2}+\maltese_1^{i-1}} \\
&\mu_{d-l+1}(a_1,\dots,\mu_{l+2}(a_i,\dots,f(a_{j_1},\dots),\dots,g(a_{j_2},\dots),\dots,a_{m+n+l+i-1}),\dots,a_{m+n+d}) \\
&=\sum_{l=1}^{d+1} \sum_{i=1}^{d-l+2} \sum_{j_1=i}^{n+l+i-2}  \sum_{j_2=j_1+m}^{m+l+i-1} (-1)^{\deg f + \deg g + \mho_{j_1,j_2}+\maltese_1^{i-1}} \\
&\mu_{d-l+2}(a_1,\dots,\mu_{l+1}(a_i,\dots,f(a_{j_1},\dots),\dots,g(a_{j_2},\dots),\dots,a_{m+n+l+i-2}),\dots,a_{m+n+d}) \\
&=\sum_{i=1}^{d+1} \sum_{l=1}^{d-i+2} \sum_{j_1=i}^{n+l+i-2}  \sum_{j_2=j_1+m}^{m+l+i-1} \quad \cdots \quad  =\sum_{j_1=1}^{n+d} \sum_{i=1}^{j_1} \sum_{j_2=j_1+m}^{m+d+1} \sum_{l=j_2-m-i+1}^{d-i+2} \quad \cdots \\
&=\sum_{j_1=1}^{n+d} \sum_{j_2=j_1+m}^{m+d+1} (-1)^{\deg f + \deg g + \mho_{j_1,j_2}} \left(\sum_{i=1}^{j_1} \sum_{l=j_2-m-i+1}^{d-i+2} (-1)^{\maltese_1^{i-1}}\right. \\ &\left. \mu_{d-l+2}(a_1,\dots,\mu_{l+1}(a_i,\dots,f(\dots),\dots,g(\dots),\dots,a_{m+n+l+i-2}),\dots,a_{m+n+d})  \right) \ .
\end{align*}
Thus, abusing the notation by leaving out $f$ and $g$, summing up (\ref{algcup2}), (\ref{algcup3}) and (\ref{algcup7}) over all $l$, $j_1$ and $j_2$ and adding up the three big sums results in:
\begin{equation}
\label{ibelowj1}
\begin{aligned}
&\sum_{j_1=1}^{n+d} \sum_{j_2=j_1+m}^{m+d+1} (-1)^{\deg f + \deg g + \mho_{j_1,j_2}} \\ &\qquad \left(\sum_{i=1}^{j_1} \sum_{l=0}^{d-i+2}(-1)^{\maltese_1^{i-1}} \mu_{d-l+2}(a_1,\dots,\mu_{l+1}(a_i,\dots,a_{i+l}),\dots,a_{m+n+d})  \right) \ .
\end{aligned}
\end{equation}
Next we treat (\ref{algcup4}):
\begin{align*}
 &\sum_{k+l=d}\sum_{j_1=1}^{n+k}  \sum_{j_2=j_1+m}^{m+k+1}\sum_{i=j_1+m}^{j_2-1} (-1)^{\deg g- 1+\mho_{j_1,j_2+l}+\maltese_1^{i-1}} \\
&\qquad \mu_{k+2}(a_1,\dots,f(a_{j_1},\dots),\dots,\mu_{l+1}(a_i,\dots,a_{i+l}),\dots,g(a_{j_2+l},\dots),\dots,a_{m+n+d}) \\
&=\sum_{l=0}^d \sum_{j_1=1}^{n+d-l} \sum_{j_2=j_1+l+m}^{m+d+1}\sum_{i=j_1+m}^{j_2-l-1} (-1)^{\deg g- 1+\mho_{j_1,j_2}+\maltese_1^{i-1}} \\
&\qquad \mu_{d-l+2}(a_1,\dots,f(a_{j_1},\dots),\dots,\mu_{l+1}(a_i,\dots,a_{i+l}),\dots,g(a_{j_2},\dots),\dots,a_{m+n+d}) \\
&=\sum_{j_1=1}^{n+d} \sum_{j_2=j_1+m}^{m+d+1} \sum_{i=j_1+m}^{j_2-1} \sum_{l=0}^{j_2-i-1}(-1)^{\deg g- 1+\mho_{j_1,j_2}+\maltese_1^{i-1}} \\
&\qquad \mu_{d-l+2}(a_1,\dots,f(a_{j_1},\dots),\dots,\mu_{l+1}(a_i,\dots,a_{i+l}),\dots,g(a_{j_2},\dots),\dots,a_{m+n+d}) \ .
\end{align*}
For (\ref{algcup5}) we obtain analogously by switching the names of the indices $i$ and $j_2$:
\begin{align*}
&\sum_{k+l=d}\sum_{j_1=1}^{n+k}  \sum_{j_2=j_1+m}^{m+k+1}\sum_{i=j_2}^{j_2+l} (-1)^{\deg g-1 +\mho_{j_1,i}+ \maltese_1^{j_2-1}} \\ 
&\qquad \mu_{k+2}(a_1,\dots,f(a_{j_1},\dots),\dots,\mu_{l+1}(a_{j_2},\dots,g(a_i,\dots),\dots,a_{j_2+l+n-1}),\dots,a_{m+n+d}) \\
&=\sum_{l=0}^d \sum_{j_1=1}^{n+d-l}  \sum_{i=j_1+m}^{m+d-l+1}\sum_{j_2=i}^{i+l} (-1)^{\deg g-1 +\mho_{j_1,j_2}+ \maltese_1^{i-1}} \\ 
&\qquad \mu_{d-l+2}(a_1,\dots,f(a_{j_1},\dots),\dots,\mu_{l+1}(a_{i},\dots,g(a_{j_2},\dots),\dots,a_{i+l+n-1}),\dots,a_{m+n+d}) \\
&=\sum_{l=0}^d \sum_{j_1=1}^{n+d-l} \sum_{j_2=j_1+m}^{m+d+1}\sum_{i=j_2-l}^{j_2} \quad \cdots =\sum_{j_1=1}^{n+d} \sum_{j_2=j_1+m}^{m+d+1}\sum_{i=j_2-d}^{j_2} \sum_{l=j_2-i}^d \quad \cdots \\
&=\sum_{j_1=1}^{n+d} \sum_{j_2=j_1+m}^{m+d+1} (-1)^{\deg g -1 + \mho_{j_1,j_2}+\maltese_1^{i-1}} \\
&\qquad \mu_{d-l+2}(a_1,\dots,f(a_{j_1},\dots),\dots,\mu_{l+1}(a_{i},\dots,g(a_{j_2},\dots),\dots,a_{i+l+n-1}),\dots,a_{m+n+d}) \ .
\end{align*}
Again by abuse of notation, the summations of (\ref{algcup4}) and (\ref{algcup5}) add up to:
\begin{align}
 &\sum_{j_1=1}^{n+d} \sum_{j_2=j_1+m}^{m+d+1} (-1)^{\deg f + \deg g + \mho_{j_1,j_2}} \label{ibetweenj1j2}\\
 &\left(\sum_{i=j_1+m}^{j_2} \sum_{l=0}^{d-i+2} (-1)^{\maltese_1^{i-1} + \deg f -1} \mu_{d-l+2}(a_1,\dots, f(\dots),\dots,\mu_{l+1}(a_i,\dots),\dots,a_{m+n+d}) \right) \ . \notag
\end{align}
Finally we consider (\ref{algcup6}). Taking the sum of this part over $l$, $j_1$ and $j_2$ yields:
\begin{align}
 &\sum_{l=0}^d\sum_{j_1=1}^{n+d-l}  \sum_{j_2=j_1+m}^{m+d-l+1} \sum_{i=j_2+n}^{m+n+d-l} (-1)^{\mho_{j_1,j_2}+\maltese_1^{i-1}}  \notag \\
&\qquad \mu_{d-l+2}(a_1,\dots,f(a_{j_1},\dots),\dots,g(a_{j_2},\dots),\dots,\mu_{l+1}(a_i,\dots),\dots,a_{m+n+d}) \notag \\
&=\sum_{j_1=1}^{n+d} \sum_{l=0}^d \sum_{j_2=j_1+m}^{m+d-l+1} \sum_{i=j_2+n}^{m+n+d-l} \quad \cdots \notag \quad =\sum_{j_1=1}^{n+d} \sum_{j_2=j_1+m}^{m+d+1} \sum_{i=j_2+n}^{m+n+d} \sum_{l=0}^{m+n+d-i} \quad \cdots \notag \\
&=\sum_{j_1=1}^{n+d} \sum_{j_2=j_1+m}^{m+d+1}(-1)^{\deg f + \deg g + \mho_{j_1,j_2}} \left(\sum_{i=j_2+n}^{m+n+d} \sum_{l=0}^{m+n+d-i} (-1)^{\maltese_1^{i-1} + (\deg f-1)+(\deg g -1)} \right.  \label{ibiggerj2} \\ 
&\qquad \left. \phantom{\sum_{l=0}^{m+d} booo} \mu_{d-l+2}(a_1,\dots,f(\dots),\dots,g(\dots),\dots,\mu_{l+1}(a_i,\dots,a_{i+l}),\dots,a_{m+n+d}) \right) \ . \notag 
\end{align}
Inserting all the last results (\ref{ibelowj1}), (\ref{ibetweenj1j2}) and (\ref{ibiggerj2}) into our big formula for the cup product implies:
\begin{align*}
 &\beta(f\cup g)(a_1,\dots,a_{m+n+d}) = \left(\beta(f)\cup g+ (-1)^{\deg f} f\cup \beta(g)\right)(a_1,\dots,a_{m+n+d}) \\ 
 &+\sum_{j_1=1}^{n+d} \sum_{j_2=j_1+m}^{m+d+1}(-1)^{\deg f + \deg g + \mho_{j_1,j_2}} \\
&\Bigl(\sum_{i=1}^{j_1} \sum_{l=0}^{d-i+2} (-1)^{\maltese_1^{i-1}} \mu_{d-l+2}(a_1,\dots,a_{i-1},\mu_{l+1}(a_i,\dots),\dots,a_{m+n+d}) \Bigr. \\
 &+\sum_{i=j_1+m}^{j_2-1} \sum_{l=0}^{d} (-1)^{\maltese_1^{i-1}+(\deg f-1)} \mu_{d-l+2}(a_1,\dots,f(\dots),\dots,\mu_{l+1}(a_i,\dots),\dots,a_{m+n+d}) \\
 &+\sum_{i=j_2+n}^{m+n+d} \sum_{l=0}^{m+n+d-i} (-1)^{\maltese_1^{i-1}+(\deg f-1)+(\deg g-1)} \\ 
&\Bigl. \qquad \qquad \qquad  \phantom{\sum_{i=1}^{j_1}} \mu_{d-l+2}(a_1,\dots,f(\dots),\dots,g(\dots),\dots,\mu_{l+1}(a_i,\dots),\dots,a_{m+n+d}) \ \Bigr) \ .
\end{align*}
One checks that for any fixed $j_1$ and $j_2$ the sum inside the brackets coincides with the left-hand side of the $(d+2)$-nd defining equation of the $A_\infty$-algebra $A$, considered on the $(d+2)$ elements 
\begin{equation*}
a_1,\dots,a_{j_1-1},f(a_{j_1},\dots,a_{j_1+m-1}),a_{j_1+m},\dots,a_{j_2-1},g(a_{j_2},\dots,a_{j_2+n-1}),\dots,a_{m+n+d} \in A \ .
\end{equation*}
 Therefore, since $A$ is an $A_\infty$-algebra, the big sum in the last equation vanishes and we obtain
\begin{equation*}
 (\beta(f\cup g))(a_1,\dots,a_{m+n+d}) = (\beta(f)\cup g+ (-1)^{\deg f} f \cup \beta(g))(a_1,\dots,a_{m+n+d})\ , 
\end{equation*}
which is precisely what we had to show.
\end{proof}

The next lemma collects some useful identities which were used in the proof of Theorem \ref{CupProductChainMap}.

\begin{lemma}
\label{SignsCupProduct}
Let ``$\equiv$'' always denote the congruence of integers modulo two. 
 \begin{enumerate}[1.]
  \item If $j_1 \leq i \leq j_1+m-1$, then
\begin{align*}
&\mho_{j_1,j_2+l}(\beta_{i-j_1+1,l}(f),g,a_1,\dots,a_{m+n+d}) \\
 &\equiv \maltese_1^{j_1-1} + \mho_{j_1,j_2}(f,g,a_1,\dots,\mu_{l+1}(a_i,\dots,a_{i+l}),\dots,a_{m+n+d}) \ .
\end{align*}
\item If $j_1 \leq i \leq j_1+l$, then
\begin{align*}
 &\mho_{j_1,j_2+l}((\beta(f))_l,g,a_1,\dots,a_{m+n+d}) \\
&\equiv \maltese_1^{j_1-1}+\deg g-1 +(\deg f-1)\maltese_{j_1}^{i-1} + \mho_{i,j_2+l}(f,g,a_1,\dots,a_{m+n+d}) \ .
\end{align*}
  \item If $j_2 \leq i \leq j_2+n-1$, then
\begin{align*}
&\mho_{j_1,j_2}(f,\beta_{i-j_2+1,l}(g),a_1,\dots,a_{m+n+d}) \\
 &\equiv \deg f + \maltese_1^{j_2-1} + \mho_{j_1,j_2}(f,g,a_1,\dots,\mu_{l+1}(a_i,\dots,a_{i+l}),\dots,a_{m+n+d}) \ .
\end{align*}
\item If $j_2 \leq i \leq j_2+l$, then
\begin{align*}
 &\mho_{j_1,j_2}(f,(\beta(g))_l,a_1,\dots,a_{m+n+d}) \\
&\equiv \maltese_1^{j_2-1} +\deg f +(\deg g-1) \maltese_{j_2}^{i-1} + \mho_{j_1,i}(f,g,a_1,\dots,a_{m+n+d}) \ .
\end{align*}
\item If $1 \leq i \leq j_1-1$, then
\begin{align*}
 &\mho_{j_1,j_2}(f,g,a_1,\dots,a_{i-1},\mu_{l+1}(a_i,\dots,a_{i+l}),a_{i+l+1},\dots,a_{m+n+d}) \\
&\equiv \deg f + \deg g + \mho_{j_1+l,j_2+l}(f,g,a_1,\dots,a_{m+n+d}) \ .
\end{align*}
\item If $j_1+m \leq i \leq j_2-1$, then
\begin{align*}
&\mho_{j_1,j_2}(f,g,a_1,\dots,a_{i-1},\mu_{l+1}(a_i,\dots,a_{i+l}),a_{i+l+1},\dots,a_{m+n+d}) \\
&\equiv \deg g-1+\mho_{j_1,j_2+l}(f,g,a_1,\dots,a_{m+n+d}) \ .
\end{align*}
\item If $j_2+n \leq i \leq m+n+k$, then
\begin{align*}
 &\mho_{j_1,j_2}(f,g,a_1,\dots,a_{i-1},\mu_{l+1}(a_i,\dots,a_{i+l}),a_{i+l+1},\dots,a_{m+n+d}) \\ &\equiv \mho_{j_1,j_2}(f,g,a_1,\dots,a_{m+n+d}) \ .
\end{align*}
\item If $1 \leq i \leq l$, then:
\begin{align*}
 &(\deg (f \cup_{k,j_1,j_2} g)-1)\cdot \maltese_1^{i-1}+\mho_{j_1,j_2}(f,g,a_i,\dots,a_{m+n+k+i-1}) \\
&\equiv \maltese_1^{i-1}+\mho_{j_1+i-1,j_2+i-1}(f,g,a_1,\dots,a_{m+n+d})  \ .
\end{align*}
 \end{enumerate}
\end{lemma}
\begin{proof}
\begin{enumerate}[1.]
\item One checks that 
\begin{align*}
 &\mho_{j_1,j_2+l}(\beta_{i-j_1+1,l}(f),g,a_1,\dots,a_{m+n+d}) \\
 &\equiv (\deg \beta_{i-j_1+1,l}(f)-1)\maltese_1^{j_1-1} + (\deg g-1)(\maltese_1^{j_2-1}+\deg \beta_{i-j_1+1,l}(f)) \\
 &\equiv \maltese_1^{j_1-1} + (\deg f-1)\maltese_1^{j_1-1} + (\deg g-1) (\maltese_1^{j_2-1}+\deg f+1) \\
 &\equiv \maltese_1^{j_1-1} + (\deg f-1)\maltese_1^{j_1-1} \\ 
 &\phantom{booooooooooooo} + (\deg g-1) (\maltese_1^{i-1} + \|\mu_{l+1}(a_i,\dots,a_{i+l})\| + \maltese_{i+l+1}^{j_2-1}+\deg f) \\
 &\equiv \maltese_1^{j_1-1} + \mho_{j_1,j_2}(f,g,a_1,\dots,a_{i-1},\mu_{l+1}(a_i,\dots,a_{i+l}),a_{i+l+1},\dots,a_{m+n+d}) \ . 
\end{align*}
\item We reformulate
\begin{align*}
 &\mho_{j_1,j_2+l}((\beta(f))_l,g,a_1,\dots,a_{m+n+d}) \\
 &\equiv \maltese_1^{j_1-1}+\deg g -1 + (\deg f-1)\maltese_1^{j_1-1} + (\deg g-1) (\maltese_1^{j_2-1}+\deg f) \\
 &\equiv \maltese_1^{j_1-1}+\deg g -1 + (\deg f-1)\maltese_{j_1}^{i-1}  \\
 &\phantom{booooooooooooo} +  (\deg f-1)\maltese_1^{i-1} + (\deg g-1) (\maltese_1^{j_2-1}+\deg f) \\
 &\equiv \maltese_1^{j_1-1}+\deg g -1 + (\deg f-1)\maltese_{j_1}^{i-1}+ \mho_{i,j_2+l}(f,g,a_1,\dots,a_{m+n+d}) \ .
\end{align*}

\item We compute
\begin{align*}
  &\mho_{j_1,j_2}(f,\beta_{i-j_2+1,l}(g),a_1,\dots,a_{m+n+d}) \\
  &=(\deg f-1) \maltese_1^{j_1-1} + (\deg \beta_{i-j_2+1,l}(g)-1)(\maltese_1^{j_2-1}+\deg f) \\
  &\equiv \deg f +\maltese_1^{j_2-1} + (\deg f-1) \maltese_1^{j_1-1} + (\deg g-1)(\maltese_1^{j_2-1}+\deg f) \\
  &\equiv \deg f +\maltese_1^{j_2-1} + \mho_{j_1,j_2}(f,g,a_1,\dots,a_{i-1},\mu_{l+1}(a_i,\dots,a_{i+l}),a_{i+l+1},\dots,a_{m+n+d}) \ .
\end{align*}	
\item Modulo two, it holds that
\begin{align*}
 &\mho_{j_1,j_2}(f,\beta_{i-j_2+1,l}(g),a_1,\dots,a_{m+n+d}) \\
 &\equiv \deg f +\maltese_1^{j_2-1} + (\deg f-1) \maltese_1^{j_1-1} + (\deg g-1)(\maltese_1^{j_2-1}+\deg f) \\
 &\equiv \deg f +\maltese_1^{j_2-1} + (\deg g-1) \maltese_{j_2}^{i-1} \\
 &\phantom{booooooooooooo} +  (\deg f-1)\maltese_1^{i-1} + (\deg g-1) (\maltese_1^{i-1}+\deg f) \\
 &\equiv \deg f +\maltese_1^{j_2-1} + (\deg g-1) \maltese_{j_2}^{i-1} + \mho_{j_1,i}(f,g,a_1,\dots,a_{m+n+d}) \ .
\end{align*}
\item By a simple computation, we obtain \begin{align*}
 &\mho_{j_1,j_2}(f,g,a_1,\dots,a_{i-1},\mu_{l+1}(a_i,\dots,a_{i+l}),a_{i+l+1},\dots,a_{m+n+d}) \\
&\equiv (\deg f-1)(\maltese_1^{j_1-1}+1)+(\deg g-1)(\maltese_1^{j_2+l-1}+\deg f+1) \\
&\equiv \deg f + \deg g + \mho_{j_1+l,j_2+l}(f,g,a_1,\dots,a_{m+n+d}) \ .
\end{align*}
\item follows in strict analogy with the fifth part of this lemma.
\item is obvious.
\item Since $\deg (f\cup_{k,j_1,j_2} g) = \deg f+\deg g$, we deduce
\begin{align*}
       &(\deg (f \cup_{k,j_1,j_2} g)-1)+\mho_{j_1,j_2}(f,g,a_i,\dots,a_{m+n+k+i-1}) \\
       &\equiv (\deg f + \deg g -1)\maltese_1^{i-1} + (\deg f-1)\maltese_i^{j_1+i-1} + (\deg g-1)(\maltese_i^{j_2+i-1}+\deg f) \\
       &\equiv \maltese_1^{i-1} + (\deg f-1)\maltese_1^{j_1+i-1} + (\deg g-1)(\maltese_1^{j_2+i-1}+\deg f) \\
       &\equiv \maltese_1^{i-1} + \mho_{j_1+i,j_2+i}(f,g,a_1,\dots,a_{m+n+d}) \ .
\end{align*}
\end{enumerate}
\end{proof}

\begin{remark} \index{Gerstenhaber algebra} 
 Getzler and Jones have shown in \cite{GetzlerJonesOperads}, generalizing a classic result of Gerstenhaber from \cite{Gerstenhaber}, that the cup product on the Hochschild cohomology of an $A_\infty$-algebra can always be extended to a Gerstenhaber algebra structure. Tradler has shown in \cite{TradlerBV} that in the presence of a certain extra structure, a so-called $\infty$-inner product, this Gerstenhaber algebra structure can be shown to arise from a Batalin-Vilkovisky algebra structure.
\end{remark}
  
  \section{The length filtration of a Hochschild chain complex and its spectral sequence}
\label{SectionLengthFiltration}

Throughout this section, we let $(A, (\mu_d)_{d\in \NN})$ be an $A_\infty$-algebra over a commutative ring $R$ and let $(M,(\mu^M_{r,s})_{r,s \in \NN_0})$ be an $A_\infty$-bimodule over $A$.

\begin{definition}
 For every $m \in \NN_0$ we define
 $$F_mCH_*(A;M) := \bigoplus_{i=0}^m M \otimes A^{\otimes i} \subset CH_*(A;M) \ . $$
 For formal reasons, we additionally define $F_{-m}CH_*(A;M):= \{0\}$ for every $m \in \NN$. 
 Every $F_mCH_*(A;M)$ is a graded submodule of $CH_*(A;M)$. It obviously holds that $F_mCH_*(A;M) \subset F_{m+1}CH_*(A;M)$ for every $m \in \NN_0$ and that 
 $$CH_*(A;M) = \bigcup_{m=0}^\infty F_mCH_*(A;M) \ . $$
 The filtration 
 $$F_0CH_*(A;M) \subset F_1 CH_*(A;M) \subset \dots \subset F_mCH_*(A;M) \subset \dots $$
 is called \emph{the length filtration of $CH_*(A;M)$. }
\end{definition}

We again write the Hochschild differential as $b = \sum_{l=0}^\infty b_{l+1}$, where every $b_{l+1}$ is constructed using only those $A_\infty$-operations with precisely $l+1$ inputs. Considering the length filtration, one observes that 
\begin{equation}
\label{EqDecreasbi}
b_{l+1}(F_mCH_*(A;M)) \subset F_{m-l}CH_*(A;M) \quad \forall l,m \in \NN_0 \ .
\end{equation}
This particularly implies that 
$$b(F_mCH_*(A;M)) \subset F_mCH_*(A;M) \ , $$
i.e. that $F_m(CH_*(A;M))$ is a subcomplex of $CH_*(A;M)$ for every $m \in \NN_0$. 

\begin{definition}
For $m \in \NN_0$ we define
 \begin{equation*}
  \pi_m: F_mCH_*(A;M) \to  M \otimes A^{\otimes m} \ , 
 \end{equation*}
as the $\ZZ$-linear extension of 
\begin{equation*}
 \pi_m(a_0 \otimes a_1 \otimes \dots \otimes a_k) = \begin{cases}
                                                     a_0 \otimes a_1 \otimes \dots \otimes a_k & \text{if} \ \ k = m \ , \\
                                                     \ 0 & \text{if} \ \ k \neq m \ ,
                                                    \end{cases} \quad \forall \ a_0 \in M \ , \ \  a_1,\dots,a_k \in A \ . 
\end{equation*}
\end{definition}

Apparently, $\pi_m$ is surjective for every $m \in \NN_0$ and
\begin{equation}
\label{kerpim}
\ker \pi_m = F_{m-1}CH_*(A;M) \ . 
\end{equation}

\begin{prop}
\label{Proppimchainmap}
For every $m \in \NN_0$, the map $\pi_m: (F_mCH_*(A;M),b) \to (M \otimes A^{\otimes m},b_1)$ is a chain map. 
\end{prop}
\begin{proof}
Combining \eqref{EqDecreasbi} and \eqref{kerpim} yields
$$\pi_m(b_{l+1}(a)) = 0 \qquad \forall l \geq 1 \ , \ \ a \in F_mCH_*(A;M) \ . $$
Moreover, it obviously holds that 
$$\pi_m(b_1(a))= \begin{cases}
                  b_1(a) & \text{if} \ \ a \in M \otimes A^{\otimes m} \ , \\
                  \ 0 & \text{if} \ \ a \in F_{m-1}CH_*(A;M) \ . 
                 \end{cases} $$
For $a \in F_{m-1}(CH_*(A;M))$ we derive from this computation and \eqref{kerpim} that
$$ (\pi_m \circ b)(a) = \sum_{l=1}^\infty \pi_m(b_l(a)) = 0 = b_1(0) = (b_1\circ \pi_m)(a) \ . $$
For $a \in A^{\otimes (m+1)}$ we derive 
$$(\pi_m \circ b)(a) = \sum_{l=1}^\infty \pi_m(b_l(a)) = \pi_m(b_1(a)) = b_1(a)= (b_1\circ \pi_m)(a) \ . $$
Since $\pi_m$ is obviously a group homomorphism and since $F_m(CH_*(A;M))$ decomposes as $$F_mCH_*(A;M) = \left(M \otimes A^{\otimes m}\right) \oplus F_{m-1}CH_*(A;M)  $$ as a free abelian group, the previous computations show that $\pi_m$ is a chain map. 
\end{proof}

Together with the above considerations, Proposition \ref{Proppimchainmap} implies:

\begin{cor}
\label{Corbarpim}
 For every $m \in \NN_0$, the map $\pi_m:F_mCH_*(A;M) \to M \otimes A^{\otimes m} $ induces an isomorphism of chain complexes
$$\bar{\pi}_m: \quot{F_mCH_*(A;M)}{F_{m-1}CH_*(A;M)} \stackrel{\cong}{\longrightarrow} M \otimes A^{\otimes m} \ . $$
\end{cor}

With every filtered chain complex, one associates a spectral sequence which converges under good circumstances against the homology of the original chain complex. We will next discuss the spectral sequence of the length filtration of $CH_*(A;M)$. For details on the following algebraic constructions we refer to \cite[Section XI.3]{MacLane} or \cite[Section 2.2]{McCleary}. \bigskip 

We define a family of $R$-bimodules $\left\{E_{p,q}\right\}_{p,q \in \ZZ}$ by putting
$$E_{p,-q} := F_pCH_{p-q}(A;M) \qquad \forall p,q \in \NN_0 \ $$
and $E_{p,q}:=0$ if $(p,q) \notin \NN_0 \times (-\NN_0)$. Then 
 $$CH_n(A;M) = \bigoplus_{\stackrel{p,q \in \NN_0}{p-q=n}} E_{p,-q} \qquad \forall n \in \ZZ \ . $$
 For the purpose of constructing the spectral sequence of the filtration, one first defines the sets $Z^r_{p,q} \subset E_{p,q}$ of elements which survive to the $r$-th page of the spectral sequence and the sets $B^r_{p,q} \subset E_{p,q}$ which induce boundaries on the $r$-th page of the spectral sequence. In the present case of a spectral sequence of a filtration, they are given as follows (see \cite[Section 2.2]{McCleary} for details): \bigskip 
 
 Let $r \in \NN_0$. For $p,q \in \NN_0$ we consider 
\begin{align*}
Z_{p,-q}^r &= \left\{x \in E_{p,q} \ | \ b(x) \in E_{p-r,-q+r-1} \right\} \\
&= \left\{x \in F_pCH_{p-q}(A;M) \ | \ b(x) \in F_{p-r}CH_{p-q-1}(A;M)  \right\} \ . 
\end{align*}
For $(p,q) \in \ZZ^2 \setminus \left(\NN_0 \times (-\NN_0)\right)$ we put $Z_{p,q}^r :=\{0\}$. Since the filtration is increasing, it is apparent that $Z_{p,q}^{r+1} \subset Z_{p,q}^{r}$ for all $p,q\in \ZZ$ and $r \in \NN_0$. 

Moreover, if we put $Z^r_{p,*} = \bigoplus_{q \in \ZZ} Z^r_{p,q}$ for every $p \in \NN_0$ and $r \in \NN$, it will follow that 
\begin{align*}
Z^r_{p,*} &= \left\{ x \in F_pCH_*(A;M) \ | \ b(x) \in F_{p-r}CH_*(A;M) \right\} \\
	  &= \Bigl\{ x \in M \otimes A^{\otimes p} \ \Bigl| \ \sum_{l=1}^{r} b_l(x) = 0 \Bigr. \Bigr\}\oplus F_{p-1}CH_*(A;M) \ , 
\end{align*}
where we used \eqref{EqDecreasbi} for the last identification. For $p,q \in \NN_0$ we further put 
\begin{align*}
 B^r_{p,-q} &= \left\{ x \in E_{p,q} \ | \ x \in b\left(E_{p+r,-q-r+1}\right) \right\} \\
 &= \left\{ x \in F_{p}CH_{p-q}(A;M) \ | \ x \in b\left(F_{p+r}CH_{p-q+1}(A;M)\right) \right\} \ .
\end{align*}
For $(p,q) \in \ZZ^2 \setminus \left(\NN_0 \times (-\NN_0)\right)$ we put $B_{p,q}^r :=\{0\}$. Since the filtration is increasing, it is apparent that $B_{p,q}^r \subset B_{p,q}^{r+1}$ for every $p,q\in \ZZ$, $r \in \NN_0$. We further define
\begin{align*}
 Z^\infty_{p,q} &= \left\{ x \in E_{p,q} \ |  \ b(x)=0 \right\}= \left\{x \in F_pCH_{p-q}(A;M) \ | \ b(x)=0 \right\}  \ , \\
 B^\infty_{p,q} &= \left\{ x \in E_{p,q} \ | x \in \im b \right\} = \left\{x \in F_pCH_{p-q}(A;M) \ | \ x \in \im b \right\} \ .
\end{align*}
Since $b$ is a differential, it holds that $B^\infty_{p,q} \subset Z^\infty_{p,q}$ for every $p,q \in \NN_0$. 

\begin{lemma}
\label{LemmaFiltConv}
 The length filtration is weakly convergent, i.e. it holds that $$Z^\infty_{p,-q} = \bigcap_{r =0}^\infty Z^r_{p,-q}$$ for every $p,q \in \NN_0$. 
\end{lemma}
\begin{proof}
 We compute that
 \begin{align*}
  \bigcap_{r=0}^\infty Z^r_{p,-q} &= \bigcap_{r =0}^\infty \left\{b \in F_{p}CH_{p-q}(A;M) \ | \ b(x) \in F_{p-r}CH_{p-q-1}(A;M) \right\} \\
				   &= \Bigl\{b \in F_pCH_{p-q}(A;M) \ \Bigl| \ b(x) \in \bigcap_{r=0}^\infty F_{p-r}CH_{p-q-1}(A;M)\Bigr. \Bigr\} \ . 
 \end{align*}
By definition, it holds that $F_{p-r}CH_{p-q-1}(A;M) = \{0\}$ whenever $r > p$. Hence 
$$\bigcap_{r=0}^\infty F_{p-r}CH_{p-q-1}(A;M)= \{0\} \quad \forall p,q \in \NN_0 \ , $$
which implies
$$\bigcap_{r=0}^\infty Z^r_{p,-q} = \left\{b \in F_pCH_{p-q}(A;M) \ | \ b(x)=0 \right\} = Z^\infty_{p,-q} \ . $$
\end{proof}

In particular, it holds that $Z^\infty_{p,q} \subset Z^r_{p,q}$ for all $p,q \in \ZZ$ and $r \in \NN_0$. One further checks without difficulties that $B^r_{p,q} \subset B^\infty_{p,q} $ for every $p,q \in \ZZ$ and $r \in \NN_0$. 

Summarizing, we have defined an increasing bigraded family of $R$-modules 
$$ B^0_{p,q}\subset  B^1_{p,q}  \subset \dots \subset B^r_{p,q} \subset \dots \subset B^\infty_{p,q} \subset Z^\infty_{p,q} \subset \dots \subset Z^r_{p,q} \subset \dots \subset Z^1_{p,q} \subset Z^0_{p,q} \ .$$
Using these sets, one constructs the spectral sequence of $\left\{F_pCH_*(A;M) \right\}_{p\in\NN_0}$. The most important result for our purposes is the following:

\begin{theorem}
 The spectral sequence of the length filtration converges against $HH_*(A;M)$. 
\end{theorem}
\begin{proof}
 This is a straightforward application of \cite[Theorem 3.2]{McCleary}, since the length filtration of $CH_*(A;M)$ is weakly convergent by Lemma \ref{LemmaFiltConv} and exhaustive, i.e. it holds that $CH_*(A;M) = \bigcup_{p=0}^\infty F_pCH_*(A;M).$
\end{proof}

We will construct the zeroth and the first page of the spectral sequence of the length filtration in detail. 
Considering the zeroth page, we observe that 
$$Z^0_{p,q}=E_{p,q}=B^0_{p,q} \quad \forall p,q \in \ZZ \ . $$
The modules constituting the zeroth page are thus given by 
\begin{equation*}
 E^0_{p,-q} := \quot{E_{p,-q}}{E_{p-1,-q}} = \quot{F_pCH_{p-q}(A;M)}{F_{p-1}CH_{p-q}(A;M)} \ . 
\end{equation*}
and the Hochschild differential $b$ induces a quotient differential 
\begin{equation*}
 \bar{b}^0: E^0_{p,-q} \to E^0_{p,-q-1} \ .
\end{equation*}
One checks that $\bar{b}^0$ is a differential of degree $-1$ on the complex $E^0_{p,*}$ for every $p \in \NN$.

Applying Corollary \ref{Corbarpim} and checking the degrees involved, one derives that for all $p,q \in \ZZ$ there is an isomorphism
\begin{equation}
\label{E0ChainIso}
E^0_{p,-q} \cong \left(M \otimes A^{\otimes p}\right)_q = \bigoplus_{q_0+\dots+q_p=q} M_{q_0} \otimes A_{q_1} \otimes \dots \otimes A_{q_p} \ , 
\end{equation}
which commutes with the differentials $\bar{b}^0$ and $b_1$. The $R$-modules which constitute the  first page of the spectral sequence of the length filtration are now given by  
$$E^1_{p,-q} := H_{-q}\left(E^0_{p,*} \ , \ \bar{b}^0\right) \quad \forall p, q\in \NN_0$$
and putting $E^1_{p,q} := \{0\}$ for every $(p,q) \in \ZZ^2  \setminus\left( \NN_0 \times (-\NN_0)\right)$. 

Since the isomorphism in \eqref{E0ChainIso} is an isomorphism of chain complex, it thus holds that 
\begin{equation*}
 E^1_{p,-q} \cong H_q\left( M \otimes A^{\otimes p},b_1\right) \quad \forall p,q \in \NN_0 \ . 
\end{equation*}

The following is a formulation of a basic result of morphisms of spectral sequences adapted to our situation: 

\begin{theorem}
\label{ThmMappingIso}
Let $A'$ be another $A_\infty$-algebra over $R$, let $M'$ be an $A_\infty$-bimodule over $A'$ and put $F_\infty CH_*(A;M) := CH_*(A;M)$ and $F_\infty CH_*(A';M'):=CH_*(A',M')$. 

Let $m \in \NN_0\cup \{\infty\}$ and let $g:F_mCH_*(A;M) \to F_mCH_*(A';M')$ be a chain map with 
 $$f(F_pCH_*(A,M))\subset F_pCH_*(A',M') \qquad \forall p \in \NN_0 \ \ \text{with} \ \ p \leq m \ . $$
 Let $f_0: E^0_{p,*} \to E^0_{p,*}$ be the chain map induced by $f$ on the zeroth page of the spectral sequences of the length filtrations. If $f_0$ induces an isomorphism in homology $$[f_0]: E^1_{p,*} \stackrel{\cong}{\longrightarrow} E^1_{p,*}$$ for every $p \in \NN_0$ with $p \leq m$, then $f$ induces an isomorphism in homology 
 $$[f]:H_*(F_mCH_*(A;M),b) \stackrel{\cong}{\longrightarrow}H_*(F_mCH_*(A';M'),b) \ . $$
\end{theorem}

\begin{proof}
 This are special cases of \cite[Theorem XI.3.4]{MacLane}. Here, one consider the restrictions of the length filtrations to $F_mCH_*(A;M)$ and $F_mCH_*(A';M')$ for some given $m \in \NN_0\cup \{\infty\}$.
\end{proof}

We will finally prove Theorem \ref{firstpageisenough} which is nothing but a special case of Theorem \ref{ThmMappingIso} for chain maps on $CH_*(A;M)$ which are induced by morphisms of $A_\infty$-bimodules. 

In this theorem we assume that $R=\ZZ$, $A=A'$ and that $(N,(\mu^N_{r,s})_{r,s\in\NN_0})$ with $N=M'$ are torsion-free abelian groups and consider a morphism of $A_\infty$-bimodules $$f=(f_{r,s})_{r,s \in \NN_0}:M \to N \ .$$ Remember that for every $r,s \in \NN_0$, the map $$f_{r,s}: A^{\otimes r} \otimes M \otimes A^ {\otimes s} \to N$$ defines a group homomorphism $\bfrs: CH_*(A;M) \to CH_*(A;N)$, as defined on page \ref{Defbfrs}.  

\begin{proof}[Proof of Theorem \ref{firstpageisenough}]
 We observe from the definition of the maps $$\bfrs: CH_*(A;M)\to CH_*(A;N)$$ that 
\begin{equation}
\label{bfrsfiltpreserv}
\bfrs(F_pCH_*(A;M)) \subset F_{p-r-s}CH_*(A;N) \qquad \forall p,r,s \in \NN_0 \  . 
 \end{equation}
 In particular, the map $f_*:CH_*(A;M) \to CH_*(A;N)$ is filtration-preserving, thus induces a morphism of spectral sequences between the spectral of the length filtrations of $CH_*(A;M)$ and $CH_*(A;N)$. We will denote the different pages of the former spectral sequence by $\{E_{p,-q}^r\}_{p,q,r \in \NN_0}$ and those of the latter one by $\{G_{p,-q}^r\}_{p,q,r \in \NN_0}$. 
 
 If $\pi: F_pCH_*(A;N) \to \quot{F_pCH_*(A;N)}{F_{p-1}CH_*(A;N)}$ denotes the projection, \eqref{bfrsfiltpreserv} implies that 
 $$\pi \circ \bfrs = 0 \quad \text{if} \ \ r+s>0  \ .$$
 Moreover, one observes that $$F_{p-1}CH_*(A;M) \subset \ker \left(\pi \circ \bar{f}_{0,0} \right) \ . $$
 
 Consequently, the map that $f_*$ induces between the first pages of the spectral sequences, that we will denote by $f_*^1: E^1_{p,*}\to G^1_{p,*}$ for every $p \in \NN_0$, is thus the map induced by $\pi \circ \bar{f}_{0,0}$ between the quotient complexes, denotes by 
 $$\bar{f}^1_{0,0}: \quot{F_pCH_*(A;M)}{F_{p-1}CH_*(A;M)} \to \quot{F_pCH_*(A;N)}{F_{p-1}CH_*(A;N)} \ . $$
 By Theorem \ref{ThmMappingIso}, $f_*$ induces an isomorphism of Hochschild homology groups if $\bar{f}^1_{0,0}$ induces an isomorphism in homology. To show the latter, we will write down $\bar{f}^1_{0,0}$ in greater detail. 
 Denoting the elements of the quotient complexes in square brackets, it holds by definition of $\bar{f}_{0,0}$ for every $p \in \NN_0$ that 
 \begin{equation}
 \label{bf100ex}
 \bar{f}^1_{0,0}\left(\left[m \otimes a_1 \otimes a_2 \otimes \dots \otimes a_p\right]\right) = \left[f_{0,0}(m)\otimes a_1\otimes a_2 \otimes \dots \otimes a_p \right] \ . 
 \end{equation}
 If we define $f_0:M \otimes A^{\otimes p}$ as the map that makes the following diagram commutative: 
 \begin{equation*}
  \begin{CD}
   \quot{F_pCH_*(A;M)}{F_{p-1}CH_*(A;M)} @>{\bar{f}^1_{0,0}}>> \quot{F_pCH_*(A;N)}{F_{p-1}CH_*(A;N)} \\
    @V{\bar{\pi}_m}V{\cong}V @V{\bar{\pi}_m}V{\cong}V \\
    M \otimes A^{\otimes p} @>{f_0}>> N \otimes A^{\otimes p} \ , 
   \end{CD}
 \end{equation*}
where $\bar{\pi}_m$ is the isomorphism from Corollary \ref{Corbarpim}, then one will derive from \eqref{bf100ex} that 
$$f_0(m\otimes a_1 \otimes a_2 \otimes \dots \otimes a_p) = f_{0,0}(m)\otimes a_1 \otimes a_2 \otimes \dots \otimes a_p$$
for all $m \in N$, $a_1,a_2,\dots,a_p \in A$, or in shorthand notation 
$$ f_0 = f_{0,0} \otimes \id_{A^{\otimes p}} \ . $$
By assumption, $f_{0,0}$ induces an isomorphism $$(f_{0,0})_*: H^*(M):= H^*(M,\mu^M_{0,0}) \to H^*(N,\mu^N_{0,0})=:H^*(N) \ . $$ The identity obviously induces an isomorphism on $H^*(A^{\otimes p})$ with respect to the differential induced by $\mu_1$. 

Since $M$ and $N$ are torsion-free, the K\"unneth theorem for abelian groups, see \cite[Theorem V.10.4]{MacLane} implies that there are horizontal maps, such that the rows in the following commutative diagram are short exact sequences: 
\begin{equation*}
 \xymatrix{0 \ar[r]  & H^*(M) \otimes H^*\left(A^{\otimes p}\right) \ar[r] \ar[d]^{(f_{0,0})_* \otimes \id}& H^*\left(M \otimes A^{\otimes p}\right) \ar[r] \ar[d]^{(f_0)_*} & \Tor(H^*(M),H^*(A^{\otimes p})) \ar[r] \ar[d]^{\Tor((f_{0,0})_*,\id)} & 0 \phantom{ \ . }\\
	   0 \ar[r] & H^*(N) \otimes H^*\left(A^{\otimes p}\right) \ar[r]& H^*\left(N \otimes A^{\otimes p}\right) \ar[r] & \Tor(H^*(N),H^*(A^{\otimes p})) \ar[r] & 0  \ . }
\end{equation*}
By the functoriality of $\Tor$ and the assumption on $f_{0,0}$, the maps $$(f_{0,0})_* \otimes \id \quad \text{and} \quad \Tor((f_{0,0})_*,\id)$$ are isomorphisms. Hence, the five-lemma implies that $(f_0)_*$ is an isomorphism as well. 

Thus, we have eventually shown that $f_*$ induces an isomorphism between the first pages of the spectral sequences of the length filtrations and therefore, by Theorem \ref{ThmMappingIso}, shown the claim.
\end{proof}

  \bibliography{/home/stephan/Dokumente/LaTeX/Draft/diss}
 \bibliographystyle{amsalpha}

\end{document}